\documentclass[a4paper,11pt]{article} 

\usepackage{fancybox} 
\usepackage{pifont} 

\usepackage{alphalph}

\usepackage{dsfont}
\usepackage{amsmath} 
\usepackage[applemac]{inputenc}
\usepackage{amsfonts}
\usepackage{amssymb}
\usepackage{amsthm}
\usepackage{stmaryrd}
\usepackage{xcolor}
\usepackage{mathrsfs} 

\newtheorem{lemma}{Lemma}

\newcommand{\mysection}{\setcounter{equation}{0} \section}

\renewcommand{\P}{\mathbb{P}}

\newcommand{\E}{\mathbb{E}}

\newcommand{\Z}{\mathbb{Z}}

\newcommand{\C}{\mathbb{C}}  
\newcommand{\N}{\mathbb{N}}

\newcommand{\R}{\mathbb{R}}

\usepackage[allcolors={blue}]{hyperref}

\newtheorem{THM}{Theorem}   
\newtheorem{remark}{Remark}

\newtheorem{PROP}[THM]{Proposition}

\def\1{\mbox{1\hspace{-0.25em}l}}

\newcommand \A[1]{{\bf (#1)}}

\def\leftB{[\![}
\def\rightB{]\!]}

\def\det{{{\rm det}}}

\def\0{{\mathbf{0}}}

\title{\textbf{Parabolic bootstrap for some non-linear equations}}
\author{\textbf{Igor Honor\'e}\footnote{Univ Lyon, CNRS, Université Claude Bernard Lyon 1, UMR5208, Institut Camille Jordan, F-69622 Villeurbanne, France. E-mail: honore@math.univ-lyon1.fr}}

\begin{document}
	\maketitle
	
	\begin{abstract}
		
		We 
	obtain the well-posedness and Schauder estimates for a class of system of linear, quasi-linear and non-linear second order partial differential equations.
	We deduce  existence and uniqueness of a global smooth solution of a non-linear and non-local equation that we call ``semi"	
	incompressible Navier Stokes equation in $\R^3$.
	
	
		\end{abstract}
	
	

	
	\mysection{Introduction}
%
%
%
%
%
	
	
	The main goal of this article is to establish the global well-posedness of a smooth solution $\mathbf u$ for a class of non-linear parabolic equation.
	Specifically, for any $x \in \R^d$, we consider the Cauchy problem, for any $(t,x) \in (0,T] \times  \R^d$ 
		\begin{equation}
	\label{Quasi_equation_intro}
	\begin{cases}
	\partial_t\mathbf u(t,x)+  \mathbb A (\mathbf u)(t,x) \cdot  \mathbf D\mathbf u (t,x) 
	= \mathbf  D^2 \mathbf u(t,x):   a(t)+ \mathbb C(\mathbf u)(t,x) + \mathbf f(t,x) 
	,\\
	\mathbf u(0,x)=\mathbf g(x),
	\end{cases}
	\end{equation}
	where the operator $\mathbb A$ is point-wisely bounded and possibly non-local, $\mathbb C$ is ``sub-linear", see Assumptions \A{P${}_{\mathbb A}$}, \A{P${}_{\mathbb C}$} and \A{F} further; and $\mathbf f, \mathbf  g, \mathbf u$ are $\R^r$-valued functions, $r \in \N$.
	The Burgers' equation in a particular case of \eqref{Quasi_equation_intro}.
	\\	
	
The strategy developed here is based on the Schaefer fixed point theorem whose the cornerstone is the \textit{a priori} estimates that we obtain by first considering the usual linear parabolic problem,
	\begin{equation}
	\label{KOLMO_intro}
	\begin{cases}
	\partial_t \mathbf u(t,x)+  \mathbf b(t,x)\cdot  \mathbf D \mathbf  u(t,x)  \textcolor{black}{+\mathbf c(t,x) \otimes \mathbf u(t,x)}=\mathbf D^2 \mathbf u(t,x):a(t)+\mathbf  f(t,x), 
	\\
	\mathbf  u(0,x)= \mathbf  g(x),\ x\in \R^{d},
	\end{cases}
	\end{equation}
	where 
	 $\mathbf b$ is $\R^r \otimes \R^d$-valued, $\mathbf c$ is $\R^r $-valued and $a$ is $\R^d \otimes \R^d$-valued.
	 The function $\mathbf b$ will depend on $\mathbf u$, this dependency on ${\mathbf u}$ must be locally bounded.
	 
	 Importantly, we obtain a \textit{parabolic bootstrap} which provides a additional controls of the solution than stated in the classical books by Ladyjenska\"ia , Solonnikov and Ural'ceva \cite{lady:solo:ural:68}, and by Lieberman \cite{lieb:96}.
	\\
	
	The backbone of the controls of the above linear equation \eqref{KOLMO_intro}, is to consider the probabilistic representation of the solution by the Feynman-Kac stochastic formula, which easily yields the boundedness of the uniform norm. Next, some Gr\"onwall's lemma applications allow us to obtain Schauder estimates for the \textit{a priori} controls and to use compactness argument requires for the Schaefer fixed point theorem.
	
	
		Usual Picard iteration
		for non-linear parabolic 
		yields difficulties to obtain a global existence of smooth solution, i.e. for a general initial condition $\mathbf g$ and time horizon $T>0$. 
	\\
	
	Thanks to our approach, we can also perform an other fixed-point argument to consider the equation, that we call ``semi" Navier-Stokes equation,
	 	\begin{equation}
	 \label{NS_equation_v1}
	 \begin{cases}
	 \partial_t\mathbf u(t,x)+ \P [\mathbf u](t,x)   \cdot \mathbf D \mathbf u(t,x)
	 = \nu \Delta\mathbf u(t,x)+ \mathbf f(t,x) 
	 ,\ t\in (0,T],\\
	 \mathbf u(0,x)=\mathbf g(x),
	 \end{cases}
	 \end{equation}
where $\P$ stands for the Leray-Hopf projector, which is a projector on the divergence free function space, see Section \ref{sec_def}.
The analysis is essentially based on the energy estimates inspired by the Leray's estimates for the usual Navier-Stokes equation \cite{lera:34}, associated with Feynman-Kac representation; additional controls in Lebesgue space, thanks to harmonic analysis, 
 are also used.
	
	In the standard Navier-Stokes equation, the Leray-Hopf projector is applied on $ \mathbf u \cdot \mathbf D \mathbf u$ instead of only $\mathbf u$ in equation \eqref{NS_equation_v1}.
	\\ 	
	
	From the quasi-linear equation, we deal with some non-linear equations of the type :
\begin{equation}
\label{Non_linear_equation_intro}
\begin{cases}
\partial_t\mathbf u(t,x)+  \mathbf P (  \mathbf u, \mathbf D \mathbf u(t,x) )(t,x) + \mathbf c(t) \otimes \mathbf u(t,x) 
= \mathbf  D^2 \mathbf u(t,x):   a(t)+ \mathbf f(t,x) 
,\ t\in (0,T],\\
\mathbf u(0,x)=\mathbf g(x),
\end{cases}
\end{equation}
where $\mathbf P$ is a locally bounded operator, see Assumption \A{P${}_{\mathbf P}$} latter.

When $\mathbf P (  \mathbf u, \mathbf D \mathbf u(t,x) )(t,x)= \mathbf D \mathbf u(t,x)^{\otimes 2} $ and $a(t)=\kappa>0$, this is a multidimensional version of KPZ equation without 
white noise as source function, see \cite{k:p:z:1986}.


	The method developed in this article could be adapted to equation associated with other differential operator. Some non-local operators could be considered like a $\alpha$-stable operator or a fractional Laplacian operator, see e.g. \cite{chau:meno:prio:19}, as soon as there is a suitable probabilistic representation. 
	\\
	%

	The paper is organised as following.
	We define useful notations in Section \ref{sec_def}.
	In Section \ref{sec_guide_proof}, we gather some tools required for our analysis: fixed-point argument and some computations rules.
	Next in Section \ref{sec_linear}, we develop the crucial analysis of linear parabolic equation \eqref{KOLMO_intro}.
	Thanks to these computations, we state in Section \ref{sec_quasi} our first main result on non-linear equation, the Schauder estimates with uniqueness of quasi-linear equation \eqref{Quasi_equation_intro}. The proof of the strong well-posedness of a solution of the quasi-linear equation \eqref{Quasi_equation_intro} satisfying Schauder estimates is developed in Section \ref{sec_proof_Quasi}.
	With some substantial extra computations, we succeed in adapting the method for ``semi" Navier-Stokes equation \eqref{NS_equation_v1} in Section \ref{sec_semi_NS}.
	We extend the results to a fully non-linear of the first order equation \eqref{Non_linear_equation_intro} in Section \ref{sec_non_line}.
	Finally, we develop in Appendix Section \ref{sec_technical_lemma} a precise proof of the Schauder estimates for a usual heat equation.

	
	\mysection{Notations and Definitions}
	\label{sec_def}
	
	\subsection{Constants}
	
	From now on, we denote by $C>0$ and $c>1$ generic constants that may change from line to line but only depends on known parameters such as $\gamma$, $d$, $r$. 
	
	
	In our analysis we state some controls where the upper-bounds are of the type $\mathbf N_{(\cdot)}(\cdot )>0$ 
	which is defined in the identity given by the index; and the dependency in the parenthesis is increasing, namely if $x_1\leq x_2$ then $\mathbf N_{(\cdot)}(\cdot,x_1,\cdot  )\leq \mathbf N_{(\cdot)}(\cdot,x_2,\cdot  )$.
	
	
		These cumbersome notations are useful to track the dependency on some non-linearity in the \textit{a priori} controls.

	\subsection{Tensor and Differential notations for real valued } 
	\subsubsection{Unidimensional valued-problem}

	For any $z \in  \R^d$, we use the decomposition $z= z_1 e_1+ \hdots z_d e_d$, where $(e_1,\hdots, e_d)$ is the canonical base of $\R^d$.
	
	We usually use the notation $\partial_t$ for the derivative in time $t \in [0,T]$, also $\partial_{z_k}$, $k \in \N$, is the derivative in the variable $z_k$.

	Next, $D_z$ denotes the gradient in the variable $z \in \R^d$, in other words $D_z=\partial_{z_1} e_1 + \hdots + \partial_{z_d} e_d$.
	When there is no ambiguity, we will also write $D$  for the gradient or the Jacobian matrix.
	
	
	The divergence writes $D_{z}\cdot$ also denoted by $D\cdot$, and is defined for any $\R$-function $f: \R^d \mapsto \R$ by $D_z \cdot f= \sum_{k=1}^d \partial_{z_k} f$.
	%
		\\
	
	For any $f: \R^d \mapsto \R$, we define the Hessian matrix $D_z^2f=\big (\partial_{z_i} \partial_{z_j} f \big )_{1 \leq i,j\leq d}$, and the usual Laplacian operator $\Delta f=\sum_{1 \leq i,j \leq d} \partial_{z_i} \partial_{z_j} f$.
	
	More generally, for any $k \in \N$, $D^k_zf$ denotes the  order $k$ tensor $(\partial_{z_{i_1}} \hdots \partial_{z_{i_k}}f)_{(i_1, \hdots, i_k) \in \leftB 1,d \rightB^k}$. For any multi-index $\alpha=(\alpha_1,\hdots,\alpha_d) \in \N_0^d$, we write $D^\alpha_z f =  \partial_{z_1}^{\alpha_1} \hdots \partial_{z_k}^{\alpha_k} f$, in particular if for $ i \in  \leftB 1,d \rightB$, $\alpha_i=0$ there is no derivative in $z_i$ in the expression of $D^\alpha_z f$. 
	
	We also denote for any $\alpha=(\alpha_1,\cdots,\alpha_m)\in \N^m $, the order of this multi-index by $|\alpha|=\sum_{i=1}^{m} \alpha_i $.
	
	
	The symbols $\lfloor \cdot \rfloor$ and $\lceil \cdot \rceil$ stand respectively for the well-known floor and ceiling functions defined for any $\chi \in \R$ by:
	\begin{eqnarray}
	\lfloor \chi \rfloor&:=&\max \{n\in\Z \mid n\le \chi \},
	\nonumber \\
	\lceil \chi \rceil&:=&\min \{n\in \Z\mid n\ge \chi \}.
	\end{eqnarray}
	
	\subsubsection{Multidimensional valued problem}
	
		From now on, the symbol ``$\cdot$" between two tensors is the usual tensor contraction. 
	For example, if $M \in \R^d \otimes \R^d \otimes\R^d$ and $N \in \R^d$ then $M \cdot N$  is a $d \times d$ matrix. If the two considered tensors are vectors then ``$\cdot$" matches with the scalar product.

	Let us explicitly precise the used tensor notations in the linear Cauchy problem \eqref{KOLMO}.
	In the whole article, `` $\cdot$ " denotes the usual tensor contraction, namely for any $(t,x)\in (0,T]\times \R^{d}$, 
	\begin{equation}\label{def_contract_tensor}
	\mathbf b(t,x)\cdot  \mathbf D\mathbf u(t,x) = \Big (\big \langle \mathbf b_i,  D \mathbf u_i \big \rangle(t,x) \Big )_{i \in   \leftB 1,r \rightB}, 
	\end{equation}
	for 
	$\mathbf  u=  (\mathbf u_i )_{i \in  \leftB 1,r \rightB}$ and $\mathbf  b=  (\mathbf b_i )_{i \in  \leftB 1,r \rightB}$, also $D$ is the $\R^d$ gradient operator and $\langle\cdot , \cdot  \rangle$ stand for the usual $\R^d$ scalar product.
	Furthermore, `` $:$ " corresponds to a double tensor contraction. Namely, we set
	\begin{equation}\label{def_double_contract_tensor}
	\mathbf D^2\mathbf u(t,x):a(t) = \Big ( {\rm Tr} \big (  D^2 \mathbf u_i  a(t) \big ) \Big )_{i \in \leftB 1,r \rightB}, 
	\end{equation}
	where $ D^2 \mathbf u_i= (\partial_{x_j}\partial_{x_k} \mathbf u_i)_{1\leq j,k \leq d}$ is the usual  Hessian matrix of $\mathbf u_i$, and ${\rm Tr}$ is the matrix trace operator. In \eqref{def_double_contract_tensor} above, $D^2 \mathbf u_i a$ is a usual matrix product.
	
	
	Finally, we denote tensor product by 
	\begin{equation}\label{def_c_otimes_u}
		\textcolor{black}{\mathbf c(t,x) \otimes \mathbf u(t,x)} :=
		\Big ( \sum_{j=1}^r\mathbf c_{ij}(t,x)  \mathbf u_j(t,x)\Big )_{1 \leq i \leq r}.
	\end{equation}


	\subsection{H\"older spaces}
	\label{SEC_HOLDER_BESOV_SPACE}

	\label{sec_def_Holder}

	In this section, we provide some useful notations and functional spaces. 
	For all  $k \in \N_0$ and $\beta\in (0,1) $,  $\|\cdot\|_{C^{k+\delta}(\R^m,\R^\ell)}$, with $m\in \{1,d\} $ and $\ell \in \{1,d, d\otimes d \} $ the considered dimensions\footnote{we write $\R^{d\otimes d}$ for $\R^d \otimes \R^d$ the space of square matrices of size $d$.}, $\delta \in (0,1)$, is the usual homogeneous H\"older norm, see for instance  
	\cite{luna:95}. Precisely, for all $\psi \in C^{k+\delta}(\R^m,\R^\ell) $,   $\alpha=(\alpha_1,\cdots,\alpha_m)\in \N^m $, 
	we set the semi-norm:
	\begin{eqnarray}
	\|\psi\|_{C^{k+\delta}(\R^m,\R^\ell)}&:=& \sum_{i=1}^k \sup_{|\alpha|=i} \|\mathbf  D^{\alpha} \psi\|_{L^\infty(\R^m,\R^\ell)}+ \sup_{|\alpha|=k} [\mathbf D^\alpha \psi]_\delta,\notag \\
	\phantom{BOUHHH} [\mathbf  D^\alpha \psi  ]_\delta &:=& \sup_{(x,y)\in (\R^m)^2,x\neq y} \frac{|\mathbf  D^\alpha \psi(x)-\mathbf  D^\alpha \psi(y)|}{|x-y|^\delta},
	\label{USUAL_HOLDER_SPACE}
	\end{eqnarray} 
	the notation $|\cdot| $ is the Euclidean norm on the considered space. 
	
	If $\delta=1$, the space matches with the usual Lipschitz space we write:
	\begin{eqnarray}
	\|\psi\|_{C^{k+1}(\R^m,\R^\ell)}&:=& \sum_{i=1}^k \sup_{|\alpha|=i} \|\mathbf  D^{\alpha} \psi\|_{L^\infty(\R^m,\R^\ell)}+ \sup_{|\alpha|=k} [\mathbf  D^\alpha \psi]_1,\notag \\
	\phantom{BOUHHH} [\mathbf  D^\alpha \psi  ]_1 &:=& \sup_{(x,y)\in (\R^m)^2,x\neq y} \frac{|\mathbf  D^\alpha \psi(x)-\mathbf D^\alpha \psi(y)|}{|x-y|}.
	\label{USUAL_Lipschitz_SPACE}
	\end{eqnarray}

	We denote by:
	$$C_b^{k+\delta}(\R^m,\R^\ell):=\{\psi \in C^{k+\delta}(\R^m,\R^\ell) : \|\psi\|_{L^{\infty}(\R^m,\R^{\ell})}<+\infty\},$$  
	the associated subspace with bounded elements (non-homogeneous H\"older space).
	The corresponding H\"older norm is defined by:
	\begin{equation}
	\label{BD_HOLDER}
	\|\psi\|_{C_b^{k+\delta}(\R^m,\R^\ell)}:=\|\psi\|_{C^{k+\delta}(\R^m,\R^\ell)}+\|\psi\|_{L^\infty(\R^m,\R^\ell)}.
	\end{equation}
	For the sake of notational simplicity, from now on, we write:
	\begin{equation*}
	\|\psi\|_{L^\infty}:=\|\psi\|_{L^\infty(\R^{d},\R^\ell)}, \
	\|\psi\|_{C^{k+\delta}}:=\|\psi\|_{C^{k+\delta}(\R^{d},\R^\ell)}
	, \
	\|\psi\|_{C_{b}^{k+\delta}}:= \|\psi\|_{C_{b}^{k+\delta}(\R^{d},\R^\ell)}.
	\end{equation*}
	For time dependent functions, $\varphi_1 \in L^{\infty} ([0,T ];C_{b}^{k+\delta}(\R^{m},\R^\ell) ) $ and $ \varphi_2 \in   L^{\infty} ([0,T ];C^{k+\delta}(\R^{m},\R^\ell) )$ we define the norms:
	\begin{equation*}
	\|\varphi_1\|_{L^\infty(C_{b}^{k+\delta})}:= \sup_{t \in [0,T]}\|\varphi_1(t, \cdot)\|_{C_{b}^{k+\delta}(\R^{m},\R^\ell)},
	\
	\|\varphi_2\|_{L^\infty(C^{k+\delta})}:= \sup_{t \in [0,T]} \|\varphi_2(t, \cdot)\|_{C^{k+\delta}(\R^{m},\R^\ell)}
	.
	\end{equation*}


	In this article, we will be as precise as possible on the norm controls. We do not necessarily upper-bound by the complete norm of the consider function space.
	For instance, even if $f \in L^\infty([0,T];C^\gamma_b(\R^d,\R^\ell ))$, we will sometimes give some controls in term of $\|f\|_{L^\infty(C^\gamma)}$ instead of $\|f\|_{L^\infty(C_b^\gamma)}$ which is crucial to our proof of the Schauder estimates.
	\\
	
For the study of the ``semi" Navier-Stokes like equation, we need to define for all $\beta  >0$, $\psi  \in C_0^\infty(\R^d, \R^d)$ and $\varphi \in C_0^\infty(\R^d, \R^d)$, the useful notation:
	\begin{equation}\label{def_norm_beta}
	\|\psi  \|_{\beta} := \sup_{x \in \R^d} (1+|x|^\beta )|\psi(x)|,  \ \ \ \|\varphi \|_{L^\infty,\beta} := \sup_{t \in [0,T]}\sup_{x \in \R^d} (1+|x|^\beta )|\varphi(t,x)|.
	\end{equation}
	For this problem, we also consider the Lebesgue space with usual notations that we do not detail here.
	
	Finally, the Leray projector $\mathbb P$ is defined for any function $\varphi : \R^3 \rightarrow \R^3$, sufficiently integrable  by
	\begin{equation}\label{def_P}
	\forall x \in \R^3,	\P \varphi (x)= \varphi(x)- D (-\Delta)^{-1}D\cdot \varphi(x).
	\end{equation}
	This is a projector on the divergence free functions space, i.e. we have $D \cdot \P \varphi =0$.
	
	
	\mysection{Tools for the proof}
	\label{sec_guide_proof}
	
	\subsection{The fixed-point theorem}
	
	Our proof of the global existence of a smooth solution to non-linear equations relies on 
	a fixed-point strategy,  the Schaefer also called Leray Schauder theorem 
	established in \cite{lera:scha:34} and \cite{scha:55}.
	\begin{THM}[Schaefer/Leray-Schauder's Fixed Point Theorem]\label{theo_schaef}
		Let $E$ be a Banach space and $\mathscr H : E \longrightarrow E$ a continuous mapping.
		We suppose that $\mathscr H$ is compact and if there is a constant $M >0$ such that for any $\mu \in [0,1]$, $e=\mu \mathscr H(e) \implies \|e\|_{E} \leq M$.
		Hence, there is $e \in E$ such that $\mathscr H(e)=e$.
	\end{THM}
	This topological result has the major property that it does not require any contraction with usual Banach fixed-point theorem.
	That is why we succeed in getting global Schauder estimates without any assumption of smallness in $T$ or $\mathbf g$.
	\\
	
	All long the paper, we use some crucial Gaussian estimates stated in the following section.
	
		\subsection{Fundamental tools for the Gaussian function}
	\label{sec_Gaussian_properties}
	
	\subsubsection{Absorbing property and cancellation}
	
	Let us recall a well-known and important result about the Gaussian function: 
	for any $\delta >0$, there is $ C=C(\delta)>1$ such that:
	\begin{equation}\label{ineq_absorb}
	\forall x \in \R^d, \ |x|^\delta e^{-|x|^2} \leq Ce^{-C^{-1} |x|^2}.
	\end{equation}
	
	Furthermore, we will also often use the cancellation principle: for all $f \in C^\gamma$, $\gamma \in (0,1)$, $x \in \R^d$ and $\sigma>0$
	\begin{equation}\label{eq_cancell}
	D_x \int_{\R^d} e^{-\frac{|x-y|^2}{2\sigma}} f(y) dy =  \int_{\R^d} D_x e^{-\frac{|x-y|^2}{2\sigma }} [f(y)-f(x)] dy,
	\end{equation}
	as the Gaussian function, up to a renormalisation by a multiplicative constant, is a probabilistic distribution we get $D_x \int_{\R^d} e^{-\frac{|x-y|^2}{2\sigma}} dy =0$ and
	\begin{eqnarray}\label{ineq_cancel}
		(2\pi \sigma)^{\frac{d}{2}}\Big | D_x \int_{\R^d} e^{-\frac{|x-y|^2}{2 \sigma}} f(y) dy \Big | 
		&\leq& 
		(2\pi \sigma)^{\frac{d}{2}} [f]_\gamma  \int_{\R^d} e^{-\frac{|x-y|^2}{2\sigma}}\frac{|y-x|}{\sigma} |y-x|^\gamma dy
		\nonumber \\
		&\leq& 
		C
		(2\pi \sigma)^{\frac{d}{2}} [f]_\gamma \sigma^{\frac{\gamma-1}{2}} \int_{\R^d} e^{-C^{-1}\frac{|x-y|^2}{2\sigma}} dy
		\nonumber \\
		&=& C [f]_\gamma \sigma^{\frac{\gamma-1}{2}}.
	\end{eqnarray}
	The penultimate identity comes from the absorbing property \eqref{ineq_absorb}.

	For $\tilde G \psi (tx):=\int_0^t  \int_{\R^3} h_\nu (t-s,x -y) \psi(s,y) dy \, ds$, we also have the standard results.
	\begin{PROP}\label{prop_norm_hold_heat_kernel}
		There is a constant  $C=C(d,\gamma)>0$ such that for any $\zeta \in L^\infty([0,T];C^\gamma(\R^d,\R))$,  $\gamma \in (0,1]$, we have for the uniform norms controls
		\begin{equation}
		\|\tilde G \zeta\|_{L^\infty}
		\leq  C T \|\zeta\|_{L^\infty} ,
		\end{equation}
		for the spatial derivatives
		\begin{equation}
		\|D \tilde G \zeta\|_{L^\infty}
		\leq  C \nu^{\frac{-1+\gamma}{2}} T^{\frac{1+\gamma}{2}}\|\zeta\|_{L^\infty(C^\gamma)}, 
		\  
		\|D^2\tilde G \zeta\|_{L^\infty}
		\leq  C \nu^{-1+\frac{\gamma}{2}} T^{\frac \gamma 2} \|\zeta\|_{L^\infty(C^\gamma)}.
		\end{equation}
		We also have the H\"older moduli controls
		\begin{equation}\label{ineq_holder_prop_heat}
		\|D^2\tilde G \zeta\|_{L^\infty(C^\gamma)} 
		\leq   C \nu^{-1} \|f\|_{L^\infty(C^\gamma)},
		\
		\|\partial_t \tilde G \zeta\|_{L^\infty(C^\gamma)} 
		\leq    C \|\zeta\|_{L^\infty(C^\gamma)}.
		\end{equation}
	\end{PROP}
	
	For the sake of completeness, the proof is recalled in Appendix.

	\subsubsection{Queues and integrability controls}
	
	In this section, we consider $d=r=3$ as for the ``semi" Navier-Stokes equation.
For any $g \in L^\infty(\R^3,\R)$ satisfying $\|g\|_{\beta}<+ \infty$, $\beta >0$, there is $C_\beta=C_\beta(\beta)>0$ such that 
\begin{eqnarray}\label{ineq_h_nu_beta}
&&
(1+|x|)^\beta \Big |\int_{\R^3} h_\nu (t,x-y) g(y) dy  \Big |
\nonumber \\
&\leq & 
2^\beta 	\int_{\R^3} |x-y|^\beta h_\nu (t,x-y) |g(y)| dy + 2^\beta  \int_{\R^3} h_\nu (t,x-y) (1+|y|)^\beta  |g(y)| dy 
\nonumber \\
&\leq & 
C_\beta  \int_{\R^3} (\nu t)^{\frac{\beta }{2}} h_{c\nu} (t,x-y) |g(y) |dy +2^\beta  \|  g \|_{\beta} 
\nonumber \\
&\leq & 
C_ \beta  (\nu t)^{\frac{\beta }{2}} \|g\|_{L^\infty}+ 2^\beta  \|  g \|_{\beta} 
\nonumber \\
&\leq & 
C_\beta   (1+[\nu t]^{\frac{\beta }{2}}) \|  g \|_{\beta} .
\end{eqnarray}

Also, for all $\varphi \in L^p(\R^3, \R^3)$, $ p \in (1, + \infty)$, by H\"older inequality we get for any $1<q<+ \infty$ such that $\frac{1}{p}+ \frac{1}{q}=1$ :
\begin{equation*}
\Big | \int_{\R^3} h_\nu (t,x-y) \varphi(y) dy \Big |
\leq 
\|\varphi\|_{L^p} \| h_\nu (t,x-\cdot )\|_{L^q},
\end{equation*}
with
\begin{equation}\label{ineq_h_nu_Lp}
\| h_\nu (t,x-\cdot )\|_{L^q}
= \Big ( \int_{\R^3} \frac{e^{-\frac{p |x-y|^2}{4 \nu t }}}{(4 \pi \nu t)^{\frac{3p}{2}}} dy \Big )^{\frac 1p} 
= \Big ( \frac{(p^{-1}4 \pi \nu t)^{\frac{3}{2}}} {(4 \pi \nu t)^{\frac{3p}{2}}} \Big )^{\frac 1p} 
= p^{-\frac{3}{2p}} (4 \pi \nu t)^{\frac{3}{2}(\frac{1-p}{p})}  .
\end{equation}
We can  write, thanks to integral Minkowski inequality, see for instance \cite{hard:litt:poly:52},
for any $p \geq 1$,
\begin{equation*}
	\Big \| \tilde G \psi (t,\cdot) 
	\Big \|_{L^p}
	\leq 
	\int_0^t  \Big \|  \int_{\R^3} h_\nu (t,\cdot -y) \psi(s,y) dy \Big \|_{L^p} ds
	\leq 
	\int_0^t  \|\psi(s,\cdot)\|_{L^p} ds.
\end{equation*}


	
	
			\subsection{A circular argument}
	\label{sec_classic_results}
	
	We provide a short result which is used in 
	the \textit{a priori} controls.
	\begin{lemma}\label{lemme_Taupe}
		For any $x \in \R_+$, if there are $a,b \in \R_+$ and $\eta \in (0,1)$ such that
		\begin{equation}\label{Taupe_ineq_lemme}
		x \leq a x^\eta + b,
		\end{equation}
		then 
		%
		\begin{equation*}
		x \leq 2 b+ 2^{\frac{1}{1-\eta}}a^{\frac{1}{1-\eta}} .
		\end{equation*}
	\end{lemma}
	\begin{proof}[Proof of Lemma \ref{lemme_Taupe}]
		
		From \eqref{Taupe_ineq_lemme}, we get
		\begin{equation*}
			x \leq 2 \max(a x^\eta , b),
		\end{equation*}
		there are two possibilities:
		\begin{equation*}
			\begin{cases}
			ax^\eta \leq b \ &\implies \ x \leq 2 b,
			\\
			ax^\eta \geq b \ &\implies \ x \leq 2^{\frac{1}{1-\eta}} a^{\frac{1}{1-\eta}},
			\end{cases}
		\end{equation*}
		the result follows.
		
	\end{proof}
	
	\mysection{Linear parabolic equation}
	\label{sec_linear}
	
	The goal of this section is to study the operator giving the solution of the linear parabolic equation
		\begin{equation}
\label{KOLMO}
	\begin{cases}
	\partial_t \mathbf u(t,x)+  \mathbf b(t,x)\cdot  \mathbf D \mathbf  u(t,x) -   \mathbf D^2 \mathbf u(t,x):a(t)+\textcolor{black}{\mathbf c(t,x) \otimes \mathbf u(t,x)}  =-\mathbf  f(t,x),\ (t,x)\in (0,T]\times \R^{d},\\
	\mathbf  u(0,x)= \mathbf  g(x),\ x\in \R^{d}.
	\end{cases}
	\end{equation}
	\textbf{\large{Assumptions}}
\begin{trivlist}
	\item[\A{E}]
	There is a positive real $\nu>0$ such that for all $x \in \R^d$ and $t \in [0,T]$
	\begin{equation*}
	\nu |x|^2 \leq \langle x a(t), x \rangle \leq \|a\|_{L^\infty} |x|^2 .
	\end{equation*}
\end{trivlist}

This uniform ellipticity hypothesis is important to ensure that the solution of \eqref{KOLMO} exists.
We could consider degenerate function $a$ with H\"ormander hypo-elliptic condition, see for instance \cite{chau:hono:meno:18}.

		\begin{THM}[Schauder estimates for linear parabolic equation] \label{THEO_SCHAU}
			Let us suppose \A{E}.
		For  $\gamma \in (0,1)$  be given.
		For all $\mathbf b \in L^\infty([0,T],C_b^\gamma(\R^d,\R^r \otimes \R^d))$, \textcolor{black}{$\mathbf c \in L^{\infty}([0,T],C_b^\gamma(\R^d,\R^r ))$}, $\mathbf f \in L^\infty([0,T];  C_b^{\gamma}(\R^d,\R^r))$ and $\mathbf g \in C^{2+\gamma}_b(\R^d,\R^r)$, there is 
	a	unique strong solution $\mathbf u$ lying in $L^\infty([0,T];C_b^{2+\gamma}(\R^{d},\R^r)) \cap C^1_b([0,T];C_b^{\gamma}(\R^{d},\R^r)) $ of \eqref{KOLMO}.
	\end{THM}

	The control of $\|\mathbf u\|_{L^\infty(C_b^{2+\gamma})}$ is already known, see \cite{frie:64}, but the novelty here, is the sharpness in each control in terms of regularity required for $\mathbf b$, $\mathbf f$ and $\mathbf g$ in the Schauder estimates.
	After preliminaries in Section \ref{sec_identifiation_H}  to introduce corresponding notations, we develop in Section \ref{sec_Schauder_linear} the proof of these Schauder estimates.

	\begin{remark}
We can consider more general matrix diffusion $a$, precisely depending on the space $x$, but in this case we cannot derive a suitable probabilistic representation of the parabolic equation.
To bypass this problem, we can use a suitable \textit{proxy}. 
We introduce a \textit{freezing} parameter $\xi \in \R^d$ which allows us to linearise parabolic equation \eqref{KOLMO} around this freezing point.

\begin{equation*}
\begin{cases}
\partial_t \mathbf u (t,x) + \mathbf b(t, x)\cdot \mathbf D \mathbf u(t,x) 
+ {\rm Tr}\big( \mathbf  D^2\mathbf u (t,x) a(t,\xi)\big)
= \mathbf f(t, x) + \frac 12{\rm Tr}\big( \mathbf D^2\mathbf u(t,x) a_{\Delta}(t,x,\xi)\big),\\
 \mathbf u(0,x)=\mathbf g (x),
\end{cases}
\end{equation*}
where
\begin{equation*}
a_{\Delta}(t,x,\xi):=a(t,\xi)-a(t,x).
\end{equation*}
The idea is to consider a new source function $-f (t,x) + \frac 12{\rm Tr}\big( D^2\mathbf u(t,x)a_{\Delta} (t,x,\xi) \big)$ where the second term is supposed to have a small contribution.
This can be done thanks to a \textit{cut locus} like in the choice of the freezing point through a separation of a \textit{diagonal}/\textit{off-diagonal} regimes and for a small final time $T$. After this procedure, we could perform a circular argument to conclude as it was done in \cite{chau:hono:meno:18}.
	\end{remark}

	\subsection{Identification of the linear parabolic operator.}
	\label{sec_identifiation_H}

	Let us 
	define the corresponding parabolic differential operator for all
	$\varphi \in C_b^\infty(\R_+\times \R^d,\R^r)$ and $(t,x) \in [0,T] \times \R^d$:
	\begin{equation}\label{def_L_m}
	\mathbf L\varphi(t,x):=
	\mathbf b (t,x) \cdot \mathbf D  \varphi(t,x) + \mathbf D^2 \varphi(t,x):  a(t) .
	\end{equation}
	With our assumption, we know from \cite{frie:64} that equation \eqref{KOLMO} has a unique strong (point-wise) solution satisfying Schauder estimates. The following subsection set a constructive method in order to get Schauder estimates of $\mathbf u$.
	It remains to compute these Schauder estimates as sharp as possible.
	
	\subsection{Approximation procedure}\label{sectionGreen}
	
	We first suppose that $\mathbf g= \mathbf 0$, the study of the general case is performed further. 
	In other words, we consider for $x \in \R^d$ the following Cauchy problem
	\begin{equation}
	\label{KOLMO_approx}
	\begin{cases}
	\partial_t \mathbf u(t,x)+ \mathbf  b (t,x) \cdot \mathbf  D \mathbf u(t,x) + \mathbf  D^2 \mathbf u(t,x):   a(t)+\textcolor{black}{\mathbf c(t,x) \otimes \mathbf u(t,x)}=-\mathbf  f(t,x) 
	,\ t\in (0,T],
	\\
	\mathbf u(0,x)= \mathbf 0. 
	\end{cases}
	\end{equation}
	We approximate this Kolmogorov equation by the \textit{proxy} heat equation
	\begin{equation}
	\label{KOLMO_TILDE}
	\begin{cases}
	\partial_t \mathbf {\tilde u}(t,x)+ \mathbf {\tilde L}_t \mathbf {\tilde  u}(t,x)=-\mathbf  f(t,x), 
	\\
	\mathbf  {\tilde u} (0,x)=\mathbf  0, 
	\end{cases}
	\end{equation}
	where for all $\varphi \in C_b^\infty(\R_+\times \R^d,\R^r)$ and $(t,x) \in [0,T] \times \R^d$:
	\begin{equation}
		\mathbf  {\tilde L}\varphi(t,x):=
	\mathbf  D^2 \varphi(t,x):   a(t).\label{DEF_TILDE_L}
	\end{equation}
	The associated heat kernel is the Gaussian density:
	\begin{eqnarray}\label{def_tilde_p}
	\tilde p (s,t,x,y)
	&=& \frac{1}{(4\pi)^{\frac d 2}\det(A_{s,t})^{\frac{ 1}{2}} } 
	\exp\left( -\frac {1}4 \left\langle A_{s,t}^{-1}(x-y),x-y\right\rangle\right),
	\end{eqnarray}
	where
	\begin{equation}\label{def_A_st_xi}
	A_{s,t}:= \int_t^s  a(\tilde s) d\tilde s.
	\end{equation}
	Observe that because for any $t \in [0,T]$, ${\det }\big (  a(t)\big )>0$ from assumption \A{E},   the kernel $\tilde p(s,t,x,y)$ is a probabilistic density.
	In particular, for all $x \in \R^d$ and $s,t \in [0,T]$ we have $\int_{\R^d} \tilde p(s,t,x,y) dy=1$ and $\tilde p(s,t,x,y)>0$, for any $y \in \R^d$. Moreover, from assumption \A{UE} we have:
	\begin{equation}\label{FIRST_CTR_DENS}
	|\tilde p(s,t,x,y)| 
	\leq 
	\big (4\pi \nu (s-t)\big )^{-\frac{d}2}\exp\left(- \frac{|x-y|^2}{4 \|a\|_{L^\infty} (s-t) } \right)
	=:  \bar p(s,t,x,y).
	\end{equation}
	and from \eqref{ineq_absorb},  for each $\alpha\in \N^d$ there is a constant $C=C(\alpha,d)>1$ s.t.
	\begin{eqnarray}\label{FIRST_deriv_CTR_DENS}
	|D^\alpha \tilde p^{\xi}(s,t,x,y)| 
	&\leq&
		C \big (4\pi \nu (s-t)\big )^{-\frac{d}2} \big ( \nu (s-t)\big ) ^{-\frac {|\alpha|}2} \exp\left(- C^{-1}\frac{|x-y|^2}{4 \|a\|_{L^\infty} (s-t) } \right)
		\nonumber \\
		&=:&
	  C
	\big ( \nu (s-t)\big ) ^{-\frac {|\alpha|}2}\bar p_{C^{-1}  }(s,t,x,y).
	\end{eqnarray}
	For the sake of notational simplicity, we will identify $\bar p_{C^{-1}  }(s,t,x,y)$ with $\bar p(s,t,x,y)$.

	For any $ \mathbf f\in C^{1,2}_0((0,T]\times \R^{d},\R^r )$, 
	we define the \textit{proxy} Green operator:
	\begin{equation}\label{GREEN_KERNEL}
	\forall (t,x) \in (0,T]\times\R^{d}, \ \tilde G \mathbf f(t,x) := \int_0^{t}  \int_{\R^{d}}  \tilde{p}(s,t, x,y) \mathbf f(s,y) \ dy \ ds.
	\end{equation}
	We define as well,  the corresponding semi-group, i.e. for any $ \mathbf g\in C^{2}_0( \R^{d},\R^r)$, 
	\begin{equation}\label{def_tilde_P}
	\tilde P \mathbf g(t,x):=\int_{\R^{d}}\tilde p (0,t,x,y) \mathbf g(y) \ dy.
	\end{equation}

	We are now in position to give the PDE associated with the density $\tilde p(s,t,s,y)$, and $\tilde G \mathbf  f$.
	
	\begin{PROP} \label{QUASI_CAUCHY_FROZEN}
		Let $\mathbf f$ in $C^{1,2}_0(\R^{d},\R^r)$ 
		then,  
		point-wisely, for any $s \in [0,T]$,
		\begin{equation}\label{relation differentielle}
		\forall (t,x,z) \in (s,T] \times (\R^{d})^2,  \Big(\partial_t + \mathbf {\tilde{L}}_t \Big) \tilde{p} (s,t,x,z) = 0,
		\end{equation}
		and
		\begin{equation}\label{relation differentielle 2}
		\begin{cases}
		\partial_t \tilde G \mathbf f(t,x)+ \mathbf {\tilde L}_{t} \tilde G \mathbf  f(t,x) = \mathbf  f (t,x),\ \forall (t,x)\in (0,T]\times \R^{d},\\
		\tilde G \mathbf  f(0,x)=\mathbf  0,
		\end{cases}
		\end{equation}
		the above differential relation is to be understood point-wise.
	\end{PROP}	

	We carefully  point out that 
	obtaining \eqref{relation differentielle 2} 
	requires the derivatives to be point-wise defined. 
	Due to the specific form of $\tilde{p}$ and $a$,
	we  directly perform this operation. 
	

	\subsection{Schauder estimates}\label{sec_Schauder_linear}
	
	We  describe, in this section, the various steps that will lead to our main result and provide a new approach to Schauder's estimates.
For the uniform control of the solution, we crucially use a stochastic representation, which allows to regard the fundamental solution of the PDE as a probability density.
Next, for the other Schauder controls we combine the Duhamel formula, considered as a perturbative formula around the heat equation, with the already estimated uniform norm through a Gr\"onwall lemma.

\subsubsection{Feynman-Kac representation}
\label{sec_Feynman_Kac}

For the particular case of the uniform $L^\infty$ control of the solution of the smooth linear parabolic equation \eqref{KOLMO}, we can readily use the Feynman-Kac formula associated with the stochastic process:
\begin{equation}\label{def_Xt}
dX_t = \begin{cases}
	d X_t^1 &= -\mathbf b_1(T-s,X_t^1) ds + \sqrt{2}\sigma(t)d B_t^1,
	 \\
	 &\vdots 
	 \\
	 d X_t^r &=  - \mathbf  b_r(T-s,X_t^r) dt + \sqrt{2}\sigma(t) dB_t^r,
	\end{cases}
	\in \R^d \otimes \R^r,
\end{equation}
where the coefficients are such that
\begin{eqnarray*}
	\mathbf b &=& (\mathbf b_i)_{1 \leq i \leq r},
	\nonumber \\
	a &=& \sigma^* \sigma,
\end{eqnarray*}
where $\sigma$ matches with the Cholesky decomposition of the second order term $a$ of the linear parabolic equation \eqref{KOLMO}, which is still regular thanks to the elliptic hypothesis \A{E}.

\begin{lemma}
For any $(t,x) \in [0,T] \times \R^3$, we can write the solution to \eqref{KOLMO} by
	\begin{equation*}
	\mathbf u (t,x)= \E_{X_t=x} \Big [\mathbf g (X_T) \Big ]
	- \E_{X_t=x}  \Big [\int_0^t \mathbf f (s,X_{T-s}) - \textcolor{black}{{\mathbf c}(s,X_{T-s}) \otimes {\mathbf u}(s,X_{T-s})}ds \Big ] .
	\end{equation*}
\end{lemma}
\begin{proof}
For any $(t,x) \in [0,T]\times \R ^d$, let us introduce $\tilde {\mathbf u}(t,x):= \mathbf u(T-t,x)$ which solves:
\begin{equation}\label{NS_LAMBDA_SHORT_tilde}
\begin{cases}
-\partial_t \tilde {\mathbf u}(t,x) + \tilde {\mathbf b} (t,x)\cdot \mathbf D  \tilde {\mathbf u} (t,x) +\textcolor{black}{\tilde {\mathbf c}(t,x) \otimes \tilde {\mathbf u}(s,x)}
= \mathbf D^2 \tilde {\mathbf u}(t,x):\tilde a(t)+\tilde {\mathbf f}(t, x), \, t\in [0,T) ,\\
\tilde {\mathbf u}(T,x)= \mathbf g (x),
\end{cases}
\end{equation}
with $\tilde {\mathbf f}(t, x):= \mathbf f(T- t, x)$, $\tilde {\mathbf b}(t, x):= \mathbf b(T- t, x)$, $\tilde {\mathbf c}(t, x):= \mathbf c(T- t, x)$  and  $\tilde a(t):=a(T-t)$.
With these notations, it is well-known that for bounded continuous matrix diffusion (which is guaranteed by assumptions on $a$) and for $\tilde {\mathbf b} \in L^\infty([0,T]; C^\gamma_b(\R^d,\R^r\otimes \R^d))$ continuous in time, we can apply Itô's lemma 
\begin{eqnarray*}
&&\tilde {\mathbf u}(T,X_{T}) - \tilde {\mathbf u}(t,X_{t})
\nonumber \\
&=&\int_t^T \Big (  \partial_t  \tilde {\mathbf u}(s,X_{s})-\tilde {\mathbf b} (s,X_{s})\cdot \mathbf D \tilde {\mathbf u}(s,X_{s}) + \mathbf D^2 \tilde {\mathbf u}(s,X_{s}): \tilde a(s) \Big )ds
+ \int_t^T   \tilde \sigma (s) \mathbf D \tilde {\mathbf u}(s,X_{s}) d B_s.
\end{eqnarray*}
Because $\tilde {\mathbf u}$ is solution of \eqref{NS_LAMBDA_SHORT_tilde}, we obtain
\begin{equation*}
\tilde {\mathbf u}(t,X_{t})=
\mathbf g(X_{T}) 
- \int_t^T   \tilde {\mathbf f}(s,X_{s}) - \textcolor{black}{\tilde {\mathbf c}(s,X_s) \otimes \tilde {\mathbf u}(s,X_s)} ds
- \int_t^T \tilde \sigma (s)   \mathbf D \tilde {\mathbf u}(s,X_{s}) d B_s.
\end{equation*}
Taking the expecting value, we get
\begin{equation*}
\tilde {\mathbf u}(t,x)=\E_{X_t=x}[\tilde {\mathbf u}(t,X_{t})]=
\E_{X_t=x}[\mathbf g(X_{T}) ]
-  \E_{X_t=x}\Big [  \int_t^T   \tilde {\mathbf f} (s,X_{s})- \textcolor{black}{\tilde {\mathbf c}(s,X_s) \otimes \tilde {\mathbf u}(s,X_s)}ds \Big ],
\end{equation*}
by definition and by variable change, $\tilde s = T-s$, we can write 
\begin{equation*}
\mathbf u( t,x)
=
\tilde {\mathbf u}(T-t,x)=
\E_{X_t^\lambda=x}[\mathbf g (X_{T}) ]
-  \E_{X_t=x}\Big [  \int_0^t   \tilde {\mathbf f}(T-\tilde s,X_{T-\tilde s})- \textcolor{black}{\tilde {\mathbf c}(T-s,X_{T-s}) \otimes \tilde {\mathbf u}(T-s,X_{T-s})}d\tilde s \Big ].
\end{equation*}
The result follows directly.

\end{proof}
This formulation readily leads to the control:
\begin{equation*}
\|\mathbf u(t,\cdot) \|_{L^\infty}
\leq 
\|\mathbf g\|_{L^\infty}+
t \|\mathbf f \|_{L^\infty} + \int_0^t \|\mathbf c(s,\cdot)\|_{L^\infty}\|\mathbf u(s,\cdot)\|_{L^\infty} ds
.
\end{equation*}
Next, by Gr\"onwall's lemma we get
\begin{eqnarray}\label{unif_u_Feynman_Kac}
\|\mathbf u(t,\cdot) \|_{L^\infty}
&\leq& 
\Big (\|\mathbf g\|_{L^\infty}+
t \|\mathbf f \|_{L^\infty} \Big ) \exp \Big ( \int_0^t \|\mathbf c(s,\cdot)\|_{L^\infty} ds\Big )
\nonumber \\
&\leq &
\Big (\|\mathbf g\|_{L^\infty}+
t \|\mathbf f \|_{L^\infty} \Big ) \exp \Big ( t \|\mathbf c \|_{L^\infty} \Big )
\nonumber \\
&=:& \mathbf N_{\eqref{unif_u_Feynman_Kac}}(t, \|\mathbf f\|_{L^\infty},\|\mathbf g\|_{L^\infty}, \|\mathbf c\|_{L^\infty})
.
\end{eqnarray}

Let us to point out that the above control \eqref{unif_u_Feynman_Kac} does not depend on $\mathbf b$ (neither on $a$). This crucial fact allows us to perform suitable \textit{a priori} controls for a fixed-point argument to get solution of non-linear equations.


\subsubsection{Control of gradient}
\label{sec_grad_linear}

From Duhamel formula, we readily write
\begin{eqnarray*}
| \mathbf D \mathbf u(t,x) | &\leq& C [\nu^{-1} t]^{\frac{1}{2}} \| \mathbf f \|_{L^\infty}+ \| \mathbf D \mathbf g \|_{L^\infty} + \Big | \int_0^t \int_{\R^d} \mathbf D \tilde{p}(s,t,x,y) \otimes \big ( \mathbf  b (s,y) \cdot \mathbf  D \mathbf u(s,y) \big )dy \, ds \Big |
\nonumber \\
&& 
+ \Big | \int_0^t \int_{\R^d} \mathbf D \tilde{p}(s,t,x,y) \otimes \big ( \mathbf c (s,y) \otimes  \mathbf u(s,y) \big ) dy \, ds \Big |.
\end{eqnarray*}
Hence,
\begin{eqnarray*}
\| \mathbf D \mathbf u(t,\cdot ) \|_{L^\infty} &\leq &C [\nu^{-1} t]^{\frac{1}{2}} \| \mathbf f \|_{L^\infty}+ \| \mathbf D \mathbf g \|_{L^\infty} + C \|\mathbf b \|_{L^\infty} \int_0^t [\nu (t-s)]^{-\frac{1}{2}} \|\mathbf  D \mathbf u(s,\cdot )\|_{L^\infty}  ds
\nonumber \\
&&+ C \mathbf N_{\eqref{unif_u_Feynman_Kac}}(t, \|\mathbf f\|_{L^\infty},\|\mathbf g\|_{L^\infty}, \|\mathbf c\|_{L^\infty}) \int_0^t \|\mathbf c(s,\cdot)\|_{L^\infty}ds .
\end{eqnarray*}
By Gr\"onwall's lemma, we then derive
\begin{eqnarray}\label{ineq_grad_linear}
&&\| \mathbf D \mathbf u(t,\cdot ) \|_{L^\infty} 
\nonumber \\
&\leq& C \Big ([\nu^{-1} t]^{\frac{1}{2}} \| \mathbf f \|_{L^\infty}+ \| \mathbf D \mathbf g \|_{L^\infty} +\mathbf N_{\eqref{unif_u_Feynman_Kac}}(t, \|\mathbf f\|_{L^\infty},\|\mathbf g\|_{L^\infty}, \|\mathbf c\|_{L^\infty}) \int_0^t \|\mathbf c(s,\cdot)\|_{L^\infty}ds  \Big )
\exp \big ( C \|\mathbf b \|_{L^\infty}  [\nu^{-1} t]^{\frac{1}{2}} \big ) 
\nonumber \\
&=:& \mathbf N_{\eqref{ineq_grad_linear}} (t,\|\mathbf b \|_{L^\infty}, \|\mathbf f \|_{L^\infty}, \| \mathbf g \|_{C^1_b},  \|\mathbf c\|_{L^\infty}).
\end{eqnarray}

\subsubsection{Control of Hessian}
\label{sec_D2_linear}

Similarly, we obtain from cancellation techniques, see \eqref{ineq_cancel}
\begin{eqnarray*}
\| \mathbf D^2 \mathbf u(t,\cdot ) \|_{L^\infty} 
&\leq& C \nu^{-1+ \frac \gamma 2} t^{\frac{\gamma}{2}} \| \mathbf f \|_{L^\infty(C^\gamma)}+ \| \mathbf D^2 \mathbf g \|_{L^\infty} + C \int_0^t [\nu (t-s)]^{-1+\frac{\gamma}{2}} [\mathbf  b (s,\cdot ) \cdot \mathbf  D \mathbf u(s,\cdot ) ]_\gamma  ds
\nonumber \\
&& 
+ \Big | \int_0^t [\nu (t-s)]^{-1+\frac{\gamma}{2}} [ \mathbf c (s,\cdot) \otimes  \mathbf u(s,\cdot) ]_\gamma ds \Big | ,
\end{eqnarray*}
with
\begin{eqnarray*}
	[\mathbf  b (s,\cdot ) \cdot \mathbf  D \mathbf u(s,\cdot ) ]_\gamma
&\leq& 	\|\mathbf  b \|_{L^\infty(C^\gamma)} 	\|\mathbf  D \mathbf u  \|_{L^\infty}+	\|\mathbf  b  \|_{L^\infty}[\mathbf  D \mathbf u(s,\cdot ) ]_\gamma,
\nonumber \\
\phantom{Bhou}	[\mathbf  c (s,\cdot ) \otimes  \mathbf u(s,\cdot ) ]_\gamma
&\leq& 	\|\mathbf  c\|_{L^\infty(C^\gamma)} 	\|\mathbf u  \|_{L^\infty}+	\|\mathbf  c  \|_{L^\infty}[ \mathbf u(s,\cdot ) ]_\gamma.
\end{eqnarray*}
By interpolation inequality, we write 
\begin{eqnarray}\label{ineq_Du_Holder_linear1}
	[\mathbf  D \mathbf u(s,\cdot ) ]_\gamma
	&\leq&  2^{1-\gamma}	\mathbf N_{\eqref{ineq_grad_linear}} (t,\|\mathbf b \|_{L^\infty}, \|\mathbf f \|_{L^\infty}, \| \mathbf g \|_{C^1_b},  \|\mathbf c\|_{L^\infty})^{1-\gamma}
	\|\mathbf  D^2 \mathbf u(s,\cdot  )\|_{L^\infty}^\gamma,
	\\
\phantom{Bhou}	[ \mathbf u(s,\cdot ) ]_\gamma
	&\leq & 2^{1-\gamma} \mathbf N_{\eqref{unif_u_Feynman_Kac}}(t, \|\mathbf f\|_{L^\infty},\|\mathbf g\|_{L^\infty}, \|\mathbf c\|_{L^\infty})^{1-\gamma}
	\mathbf N_{\eqref{ineq_grad_linear}} (t,\|\mathbf b \|_{L^\infty}, \|\mathbf f \|_{L^\infty}, \| \mathbf g \|_{C^1_b},  \|\mathbf c\|_{L^\infty})^\gamma
	\nonumber 
	\\ 
	&=:& \mathbf N_{\eqref{ineq_Holder_linear}}(t, \|\mathbf f\|_{L^\infty},\|\mathbf g\|_{C_b^1}, \|\mathbf c\|_{L^\infty}).
	\label{ineq_Holder_linear}
\end{eqnarray}
Then, we obtain
\begin{eqnarray}\label{ineq_Hess_linear_prelim}
&&\| \mathbf D^2 \mathbf u(t,\cdot)\|_{L^\infty}
\nonumber \\
 &\leq& 
C\nu^{-1+ \frac \gamma 2} t^{\frac{\gamma}{2}}  \| \mathbf f \|_{L^\infty(C^\gamma)}+ \| \mathbf D^2 \mathbf g \|_{L^\infty} 
+ C \nu^{-1+ \frac \gamma 2} t^{\frac{\gamma}{2}}  \Big (
\|\mathbf  b \|_{L^\infty(C^\gamma)} \mathbf N_{\eqref{ineq_grad_linear}} (t,\|\mathbf b \|_{L^\infty}, \|\mathbf f \|_{L^\infty}, \| \mathbf g \|_{C^1_b},  \|\mathbf c\|_{L^\infty})
\nonumber \\
&&+ 2^{1-\gamma}
\|\mathbf  b  \|_{L^\infty}\mathbf N_{\eqref{unif_u_Feynman_Kac}}(t, \|\mathbf f\|_{L^\infty},\|\mathbf g\|_{L^\infty}, \|\mathbf c\|_{L^\infty})^{1-\gamma}
\|\mathbf  D^2 \mathbf u\|_{L^\infty}^\gamma \Big )
\nonumber \\
&&
+ 	\|\mathbf  c\|_{L^\infty(C^\gamma)} \mathbf N_{\eqref{unif_u_Feynman_Kac}}(t, \|\mathbf f\|_{L^\infty},\|\mathbf g\|_{L^\infty}, \|\mathbf c\|_{L^\infty})
+ 	\|\mathbf  c  \|_{L^\infty} \mathbf N_{\eqref{ineq_Holder_linear}}(t, \|\mathbf f\|_{L^\infty},\|\mathbf g\|_{C_b^1}, \|\mathbf c\|_{L^\infty})
\nonumber \\
&=:&
\mathbf N_{\eqref{ineq_Hess_linear_prelim}}(t, \|\mathbf f\|_{L^\infty(C^\gamma_b)},\|\mathbf g\|_{C^2_b}, \|\mathbf c\|_{L^\infty(C^\gamma_b)})
\nonumber \\
&&
+  C \nu^{-1+ \frac \gamma 2} t^{\frac{\gamma}{2}} 
\|\mathbf  b  \|_{L^\infty}\mathbf N_{\eqref{unif_u_Feynman_Kac}}(t, \|\mathbf f\|_{L^\infty},\|\mathbf g\|_{L^\infty}, \|\mathbf c\|_{L^\infty})^{1-\gamma}
\|\mathbf  D^2 \mathbf u\|_{L^\infty}^\gamma .
\end{eqnarray}
Therefore, we then derive from Lemma \ref{lemme_Taupe},
\begin{eqnarray}\label{ineq_Hess_linear}
&&\| \mathbf D^2 \mathbf u(t,\cdot)\|_{L^\infty} 
\nonumber \\
&\leq& 
2
\mathbf N_{\eqref{ineq_Hess_linear_prelim}}(t, \|\mathbf f\|_{L^\infty(C^\gamma_b)},\|\mathbf g\|_{C^2_b}, \|\mathbf c\|_{L^\infty(C^\gamma_b)})
+  (2 C \nu^{-1+ \frac \gamma 2} t^{\frac{\gamma}{2}}  )^{\frac{1}{1-\gamma}}
\|\mathbf  b  \|_{L^\infty}\mathbf N_{\eqref{unif_u_Feynman_Kac}}(t, \|\mathbf f\|_{L^\infty},\|\mathbf g\|_{L^\infty}, \|\mathbf c\|_{L^\infty})
\nonumber \\
&=:& \mathbf N_{\eqref{ineq_Hess_linear}}(T, \|\mathbf b \|_{L^\infty (C^\gamma_b)}, \|\mathbf f \|_{L^\infty (C^\gamma_b)}, \|\mathbf g\|_{C^2_b}, \|\mathbf c\|_{L^\infty(C^\gamma_b)})  .
\end{eqnarray}

From this inequality and \eqref{ineq_Du_Holder_linear1}, we also deduce readily, by interpolation, that
\begin{eqnarray}\label{ineq_Du_Holder_linear}
[\mathbf  D \mathbf u(t,\cdot ) ]_\gamma
&\leq& 2^{1-\gamma} 	\mathbf N_{\eqref{ineq_grad_linear}} (t,\|\mathbf b \|_{L^\infty}, \|\mathbf f \|_{L^\infty}, \| \mathbf g \|_{C^1_b},  \|\mathbf c\|_{L^\infty})^{1-\gamma}
\nonumber \\
&&\times  \mathbf N_{\eqref{ineq_Hess_linear}}(T, \|\mathbf b \|_{L^\infty (C^\gamma_b)}, \|\mathbf f \|_{L^\infty (C^\gamma_b)}, \|\mathbf g\|_{C^2_b}, \|\mathbf c\|_{L^\infty(C^\gamma_b)})^\gamma
\nonumber \\
&=:& \mathbf N_{\eqref{ineq_Du_Holder_linear}}(T, \|\mathbf b \|_{L^\infty (C^\gamma_b)}, \|\mathbf f \|_{L^\infty (C^\gamma_b)}, \|\mathbf g\|_{C^2_b}, \|\mathbf c\|_{L^\infty(C^\gamma_b)}).
\end{eqnarray}

\subsubsection{Control of the H\"older modulus of the Hessian}
\label{sec_D2_Holder_linear}

From Proposition \ref{prop_norm_hold_heat_kernel},
\begin{eqnarray*}
\| \mathbf D^2 \mathbf u(t,\cdot)\|_{L^\infty(C^\gamma)}
 &\leq& C \nu^{-1+ \frac{\gamma}{2}}(1+\nu^{-\frac{1}{2}}) \| \mathbf f \|_{L^\infty(C^\gamma)}+ [\mathbf D^2 \mathbf g ]_\gamma 
 \nonumber \\
 && + C \nu^{-1+ \frac{\gamma}{2}}(1+\nu^{-\frac{1}{2}}) \Big ( \|\mathbf  b  \cdot \mathbf  D \mathbf u\|_{L^\infty(C^\gamma)}
 +  \|\mathbf  c  \otimes \mathbf u\|_{L^\infty(C^\gamma)} \Big )
 \nonumber \\
 &\leq& C \nu^{-1+ \frac{\gamma}{2}}(1+\nu^{-\frac{1}{2}}) \| \mathbf f \|_{L^\infty(C^\gamma)}+ [\mathbf D^2 \mathbf g ]_\gamma \nonumber \\
 &&+ C \nu^{-1+ \frac{\gamma}{2}}(1+\nu^{-\frac{1}{2}})  \Big ( \|\mathbf  b  \|_{L^\infty(C^\gamma)} \|\mathbf  D \mathbf u\|_{L^\infty}
 +  \|\mathbf  b  \|_{L^\infty} \| \mathbf  D \mathbf u\|_{L^\infty(C^\gamma)}
 \nonumber \\
 && +  \|\mathbf  c  \|_{L^\infty(C^\gamma)} \| \mathbf u\|_{L^\infty}
 +  \|\mathbf  c  \|_{L^\infty} \|  \mathbf u\|_{L^\infty(C^\gamma)}\Big )
 .
\end{eqnarray*}
Hence, we obtain,
\begin{eqnarray}\label{ineq_Hess_Holder_linear}
\| \mathbf D^2 \mathbf u(t,\cdot)\|_{L^\infty(C^\gamma)} 
 &\leq& C \nu^{-1+ \frac{\gamma}{2}}(1+\nu^{-\frac{1}{2}}) \| \mathbf f \|_{L^\infty(C^\gamma)}+ [\mathbf D^2 \mathbf g ]_\gamma 
\nonumber \\
&&+ C \nu^{-1+ \frac{\gamma}{2}}(1+\nu^{-\frac{1}{2}}) \Big ( \|\mathbf  b  \|_{L^\infty(C^\gamma)} \mathbf N_{\eqref{ineq_grad_linear}} (t,\|\mathbf b \|_{L^\infty}, \|\mathbf f \|_{L^\infty}, \|\mathbf D \mathbf g \|_{L^\infty}, \|\mathbf  c  \|_{L^\infty})
\nonumber \\
&&+  \|\mathbf  b  \|_{L^\infty} \| \mathbf N_{\eqref{ineq_Du_Holder_linear}} (t,\|\mathbf b \|_{L^\infty}, \|\mathbf f \|_{L^\infty}, \|\mathbf D \mathbf g \|_{L^\infty}, \|\mathbf  c  \|_{L^\infty(C^\gamma)} )
\nonumber \\
&& +  \|\mathbf  c  \|_{L^\infty(C^\gamma)} \mathbf N_{\eqref{unif_u_Feynman_Kac}} (t,\|\mathbf b \|_{L^\infty}, \|\mathbf f \|_{L^\infty}, \|\mathbf D \mathbf g \|_{L^\infty}, \|\mathbf  c  \|_{L^\infty})
\nonumber \\
&&+  \|\mathbf  c  \|_{L^\infty} \mathbf N_{\eqref{ineq_Holder_linear}} (t,\|\mathbf b \|_{L^\infty}, \|\mathbf f \|_{L^\infty}, \|\mathbf D \mathbf g \|_{L^\infty}, \|\mathbf  c  \|_{L^\infty}) \Big )
\nonumber \\
&=:&   \mathbf N_{\eqref{ineq_Hess_Holder_linear}}(T, \|\mathbf b \|_{L^\infty (C^\gamma_b)}, \|\mathbf f \|_{L^\infty (C^\gamma_b)}, \|\mathbf g\|_{C^{2+\gamma}_b}\|\mathbf  c  \|_{L^\infty(C_b^\gamma)} ).
\end{eqnarray}

\subsubsection{Control of the uniform norm of the time derivative}
\label{sec_partial_t_unif_linear}

It is direct from the Duhamel formula and Proposition \ref{prop_norm_hold_heat_kernel} that, for any $(t,x) \in [0,T] \times \R^d$,
\begin{eqnarray*}
| \partial_t \mathbf u(t,x) | 
&\leq &
\| \mathbf f \|_{L^\infty} + C \nu^{-1+ \frac{\gamma}{2}} t^{\frac{\gamma}{2}} \| \mathbf f \|_{L^\infty(C^\gamma)}
+\nu  \| \mathbf D^2 \mathbf g \|_{L^\infty} + \| \mathbf  b \cdot \mathbf  D \mathbf u  \|_{L^\infty}
+ C \nu^{-1+ \frac{\gamma}{2}}t^{\frac{\gamma}{2}}\| \mathbf  b \cdot \mathbf  D \mathbf u \|_{L^\infty(C^\gamma)}
\nonumber \\
&&+  \| \mathbf  c \otimes \mathbf u  \|_{L^\infty}
+ C \nu^{-1+ \frac{\gamma}{2}} t^{\frac{\gamma}{2}}\| \mathbf  c \otimes \mathbf u  \|_{L^\infty(C^\gamma)},
 .
\end{eqnarray*}
which yields that
\begin{eqnarray}\label{ineq_partial_t_unif_linear}
	| \partial_t \mathbf u(t,x) | 
	&\leq & 
	\| \mathbf f \|_{L^\infty} + C \nu^{-1+ \frac{\gamma}{2}} t^{\frac{\gamma}{2}} \| \mathbf f \|_{L^\infty(C^\gamma)}
	+ \| \mathbf D^2 \mathbf g \|_{L^\infty} 
	\nonumber \\
	&&+ C(\| \mathbf  b\|_{L^\infty}+t^{\frac{\gamma}{2}}\| \mathbf  b\|_{L^\infty(C^\gamma)}) \mathbf N_{\eqref{ineq_grad_linear}} (t,\|\mathbf b \|_{L^\infty}, \|\mathbf f \|_{L^\infty}, \|\mathbf D \mathbf g \|_{L^\infty}, \|\mathbf  c\|_{L^\infty} )
	\nonumber \\
	&&+ C \nu^{-1+ \frac{\gamma}{2}} t^{\frac{\gamma}{2}}
	\| \mathbf  b \|_{L^\infty} \mathbf N_{\eqref{ineq_Du_Holder_linear}}(t, \|\mathbf b \|_{L^\infty (C^\gamma_b)}, \|\mathbf f \|_{L^\infty (C^\gamma_b)}, \|\mathbf g\|_{C^2_b}, \|\mathbf  c\|_{L^\infty(C^\gamma)} )
\nonumber \\
&&+ C(\|\mathbf  c\|_{L^\infty} + \nu^{-1+ \frac{\gamma}{2}}t^{\frac{\gamma}{2}}	\|\mathbf  c\|_{L^\infty(C^\gamma)} )\mathbf N_{\eqref{unif_u_Feynman_Kac}}(t, \|\mathbf f\|_{L^\infty},\|\mathbf g\|_{L^\infty}, \|\mathbf c\|_{L^\infty})
\nonumber \\
&&+ 	\|\mathbf  c  \|_{L^\infty}
\mathbf N_{\eqref{ineq_Holder_linear}}(t, \|\mathbf f\|_{L^\infty},\|\mathbf g\|_{C_b^1}, \|\mathbf c\|_{L^\infty})
\nonumber \\
&=:&
\mathbf N_{\eqref{ineq_partial_t_unif_linear}} (t, \|\mathbf b \|_{L^\infty (C^\gamma_b)}, \|\mathbf f \|_{L^\infty (C^\gamma_b)}, \|\mathbf g\|_{C^2_b}, \|\mathbf  c\|_{L^\infty(C^\gamma)} )	.
\end{eqnarray}

\subsubsection{Control of the spatial H\"older of the time derivative}
\label{sec_partial_t_Holder_linear}

Similarly to the previous control and from Proposition \ref{prop_norm_hold_heat_kernel} we get
\begin{eqnarray}\label{ineq_partial_t_Holder_linear}
\| \partial_t \mathbf u \|_{L^\infty(C^\gamma)}
&\leq & 
\big ( 1+ C \nu^{-1+ \frac{\gamma}{2}}(1+\nu^{-\frac{1}{2}})\big ) \| \mathbf f \|_{L^\infty(C^\gamma)}
+  C \nu^{-1+ \frac{\gamma}{2}}(1+\nu^{-\frac{1}{2}})  [\mathbf D^2 \mathbf g ]_ \gamma 
\nonumber \\
&&+ \big ( 1+ C \nu^{-1+ \frac{\gamma}{2}}(1+\nu^{-\frac{1}{2}})\big )  \big (  \| \mathbf  b \cdot \mathbf  D \mathbf u \|_{L^\infty(C^\gamma)}
+    \| \mathbf  c \otimes \mathbf u \|_{L^\infty(C^\gamma)} \big )
\nonumber \\
&\leq & 
\big ( 1+ C \nu^{-1+ \frac{\gamma}{2}}(1+\nu^{-\frac{1}{2}})\big )  \| \mathbf f \|_{L^\infty(C^\gamma)}
+ C  [\mathbf D^2 \mathbf g ]_ \gamma 
\nonumber \\
&&
+ 
\Big ( 1+ C \nu^{-1+ \frac{\gamma}{2}}(1+\nu^{-\frac{1}{2}})\Big )  \Big (\| \mathbf  b \|_{L^\infty} \mathbf N_{\eqref{ineq_Du_Holder_linear}}(T, \|\mathbf b \|_{L^\infty (C^\gamma_b)}, \|\mathbf f \|_{L^\infty (C^\gamma_b)}, \|\mathbf g\|_{C^2_b}, \|\mathbf  c\|_{L^\infty(C^\gamma)} )
\nonumber \\
&&+ \| \mathbf  b \|_{L^\infty(C^\gamma)}
\mathbf N_{\eqref{ineq_grad_linear}} (T,\|\mathbf b \|_{L^\infty}, \|\mathbf f \|_{L^\infty}, \|\mathbf D \mathbf g \|_{L^\infty},\|\mathbf  c\|_{L^\infty} ) 
\nonumber \\
&&+ 
\| \mathbf  c \|_{L^\infty} \mathbf N_{\eqref{ineq_Holder_linear}}(T, \|\mathbf b \|_{L^\infty (C^\gamma_b)}, \|\mathbf f \|_{L^\infty}, \|\mathbf g\|_{C^1_b}, \|\mathbf  c\|_{L^\infty} )
\nonumber \\
&&+ \| \mathbf  c \|_{L^\infty(C^\gamma)}
\mathbf N_{\eqref{unif_u_Feynman_Kac}} (T,\|\mathbf b \|_{L^\infty}, \|\mathbf f \|_{L^\infty}, \| \mathbf g \|_{L^\infty},\|\mathbf  c\|_{L^\infty} ) \Big )
\nonumber \\
&=:&
\mathbf N_{\eqref{ineq_partial_t_Holder_linear}} (T, \|\mathbf b \|_{L^\infty (C^\gamma_b)}, \|\mathbf f \|_{L^\infty (C^\gamma_b)}, \|\mathbf g\|_{C^{2+\gamma}_b},\|\mathbf  c\|_{L^\infty(C^\gamma_b)})	.
\end{eqnarray}

Thus the stated Schauder estimates.

\mysection{Quasi-linear equations}
\label{sec_quasi}
	This current section is dedicated to the general case of quasi-linear equation when the first order term depends on the solution itself. 
	In the next part, i.e. Section \ref{sec_semi_NS}, a non-local case, that we call ``semi" Navier-Stokes equation, is treated  where the proof requires some extra computations.
\\
	
	For $0<r\leq d$, we consider the non-linear equation defined for a given $T>0$ (arbitrary big) by
	\begin{equation}
	\label{Quasi_equation_v1}
	\begin{cases}
	\partial_t\mathbf u(t,x)+  \mathbb A (\mathbf u)(t,x) \cdot  \mathbf D\mathbf u (t,x) 
	= \mathbf  D^2 \mathbf u(t,x):   a(t)+ \C(\mathbf u)(t,x)  + \mathbf f(t,x) 
	,\ t\in (0,T],\\
	\mathbf u(0,x)=\mathbf g(x),
	\end{cases}
	\end{equation}
	where we recall that $\mathbf D\mathbf u $ stands for the Jacobian matrix by blocks, ``$\cdot$" for the tensor contraction.
	%
	%
	In other words,
	\begin{equation}\label{def_Du_u}
	\mathbb A (\mathbf u) \cdot \mathbf D\mathbf u 
	= \big (\sum_{1 \leq j \leq d}(\mathbb A (\mathbf u))_j \partial_{x_j} (\mathbf u)_i\big )_{1 \leq i \leq r} .
	\end{equation}
The dimensions $d$ and $r$ can be different and can be arbitrary ``big".  
	The associated non-linear differential operator is defined as following,  for any $\mathbf \varphi \in C^\infty_0([0,T] \times \R^d, \R^r)$, by 
	\begin{equation}\label{def_Ln_t}
	\forall (t,x) \in [0,T] \times \R^d, \ \mathbf L \mathbf \varphi(t,x) := 
	\mathbb A (\varphi)(t,x) \cdot  \mathbf D \mathbf {\varphi}  (t,x) 
	- \mathbf  D^2 \varphi(t,x): a(t)+ \C(\mathbf \varphi)(t,x).
	\end{equation}
	
	\textbf{\large{Assumptions}}
	\begin{trivlist}
\item[\A{P${}_{\mathbb A}$}]
	There is a non-negative real function $\mathscr M_{\mathbb A}: \R_+ \rightarrow \R_+$ locally bounded such that, for all $\mathbf b \in C^\infty_0([0,T] \times \R^d, \R^r)$ and $\gamma \in (0,1]$: 
	\begin{eqnarray}\label{hypo_A}
		\|\mathbb A(\mathbf b)\|_{L^\infty} + 	\|\mathbb C(\mathbf b)\|_{L^\infty}&\leq &\mathscr M_{\mathbb A}(\|\mathbf b\|_{L^\infty}),
		\nonumber \\
		\|\mathbb A(\mathbf b)\|_{L^\infty(C_b^\gamma)} + \|\mathbb C(\mathbf b)\|_{L^\infty(C_b^\gamma)} &\leq & \mathscr M_{\mathbb A} ( \|\mathbf b\|_{L^\infty(C_b^\gamma)}),
	\end{eqnarray}
	for $\gamma=1$, we naturally suppose $
	\|\mathbf D \mathbb A(\mathbf b)\|_{L^\infty} +  \|\mathbf D \mathbb C(\mathbf b)\|_{L^\infty} \leq  \mathscr M_{\mathbb A} ( \|\mathbf D \mathbf b\|_{L^\infty})$.
	\item[\A{P${}_{\mathbb C}$}]
	There is a non-negative function $c \in L^\infty([0,T] ,\R_+)$ such that, for all $\mathbf b \in C^\infty_0([0,T] \times \R^d, \R^r)$, and  $t \in [0,T]$
	\begin{eqnarray}\label{hypo_A}
	\|\mathbb C(\mathbf b)(t,\cdot)\|_{L^\infty} &\leq &c(t)\|\mathbf b(t,\cdot)\|_{L^\infty},
	\nonumber \\
 \|\mathbb C(\mathbf b)(t,\cdot)\|_{L^\infty(C_b^\gamma)} &\leq &  c(t)  \|\mathbf b(t,\cdot)\|_{C_b^\gamma}.
	\end{eqnarray}
\item[\A{F}]
There are non-negative real functions $\tilde{\mathscr M}_{\mathbb A}, \tilde{\mathscr M}_{\mathbb C}: \R_+ \times \R_+ \rightarrow \R_+$ locally bounded such that, for any $\mathbf b_1, \mathbf b_2 \in C^\infty_0([0,T] \times \R^d, \R^r)$, 
\begin{equation}\label{hypo_A_func}
\|\mathbb A(\mathbf b_1)-\mathbb A(\mathbf b_2)\|_{L^\infty} \leq \|\mathbf b_1-\mathbf b_2\|_{L^\infty} \tilde {\mathscr M}_{\mathbb A}(\|\mathbf b_1\|_{L^\infty}, \|\mathbf b_2\|_{L^\infty}),
\end{equation}
and
\begin{equation}\label{hypo_C_func}
\|\mathbb C(\mathbf b_1)-\mathbb C(\mathbf b_2)\|_{L^\infty} \leq \|\mathbf b_1-\mathbf b_2\|_{L^\infty} \tilde {\mathscr M}_{\mathbb C}(\|\mathbf b_1\|_{L^\infty}, \|\mathbf b_2\|_{L^\infty}).
\end{equation}
\end{trivlist}
	Importantly, the operators $\mathbb A$ and $\mathbb C$ can be non-local, as it will be the case for ``semi" Navier-Stokes equation (which requires a particular analysis as it does not satisfy the above assumptions).
	\begin{remark}
We can suppose some dependency on the current time in the upper-bounds, even with some time singularities (\textit{a priori} only integrable ones).
But for the sake of clarity, we choose to avoid this case which should not imply substantial difficulties.
	\end{remark}
	For example, we can choose the operators:
	\begin{eqnarray}
		\mathbb A (\mathbf b)(t,x) &=& \mathbf b(t,x),  \text{ \textit(multidimensional Burgers like equation)},
		\nonumber \\
		\mathbb A (\mathbf b)(t,x) &=& |\mathbf b|^k(t,x) \mathbf b(t,x), \ k \in \R_+,  \text{ \textit(generalized multidimensional Burgers like equation)},
		\nonumber \\
		\mathbb A (\mathbf b) (t,x) &=& \int_{\R^d} \rho(t,x-y) |\mathbf b|^k(t,y) \mathbf b(t,y) dy, \ k \in \R_+,  \text{ for any } \rho \in L^\infty([0,T]; L^1(\R^d,\R)),
		\nonumber \\
		\mathbb A (\mathbf b) (t,x) &=& \exp \big ( |\mathbf b|(t,x) \big ) \mathbf b(t,x) .
	\end{eqnarray}

\begin{remark}
			The first example, the ``multidimensional	viscous Burgers' equation" is a particular case of this problem whose the Schauder estimates are already established in \cite{unte:17} but with high regularity for $f$ called therein the \textit{forcing term}.
	
\end{remark}
	
	
	
	\begin{THM}[Schauder estimates for quasi-linear equation]\label{Theo_Quasi}
		We suppose \A{E}, \A{P${_{\mathbb A}}$}, \A{P${_{\mathbb C}}$} and \A{F}.
		For  $\gamma \in (0,1)$ be given.
		For all $\mathbf f \in L^\infty([0,T];  C_b^{\gamma}(\R^d,\R^r))$ and $\mathbf g \in C^{2+\gamma}_b(\R^d,\R^r)$, 
		 there is  a unique strong solution $\mathbf u \in  L^\infty([0,T];C_b^{2+\gamma}(\R^{d},\R^r)) \cap C^1_b([0,T];C_b^{\gamma}(\R^{d},\R^r)) $ of \eqref{Quasi_equation_v1}. 

	\end{THM}

The idea of the proof is to combine the Schaefer/Leray Schauder theorem, see Theorem \ref{theo_schaef},
with our Schauder estimates in Theorem \ref{THEO_SCHAU}, which allows us to bypass some usual difficulties related with quasi-linear equations, see e.g. Remark and warning page 507 in \cite{evan:98} where we see the difficulty to build a sequence which at the limit converges towards the quasi-linear equation (no Cauchy sequence without smallness assumption on the initial condition).
	A fixed point approach is not new for non-linear equations, we can refer for instance to the steady-states solutions of incompressible Navier-Stokes in \cite{lema:16} p530.
	
	\section{Proof of Theorem \ref{Theo_Quasi}}%
	\label{sec_proof_Quasi}		
%
		
		For a given $\mathbb B \in L^\infty([0,T]; C^1_b(\R^d,\R^r))$, we approximate this quasi-linear equation by a multidimensional version of the linear parabolic equation previously studied. 
		We 
		 first consider the following parabolic equation, for any $(t,x) \in [0,T] \times \R^d$
		\begin{equation}\label{KOLMO_suite}
		\begin{cases}
		\partial_t \tilde{\mathbf u}(\mathbf b)(t,x)+ \mathbb B(t,x) \cdot \mathbf  D \tilde{\mathbf u}(\mathbf b)  (t,x) 
		+ \mathbb D(t,x)
		= \mathbf  D^2 \mathbf u (\mathbf b)(t,x):   a(t)+ \mathbf f(t,x) 
		,
		\\
		\tilde {\mathbf u}(\mathbf b)(0,x)=\mathbf g(x),
		\end{cases}
		\end{equation}
		where the drift term and the zero order term write:
		\begin{equation}\label{def_mathbb_B}
			\mathbb B (t,x) = \mathbb A(\mathbf b)(t,x), \ \ \ \mathbb D(t,x):=\mathbb C(\mathbf b)(t,x).
		\end{equation}
		
		We define the associated differential operator,  for any $\varphi \in C^\infty_0([0,T] \times \R^d, \R^r)$, by 
		\begin{equation}\label{def_Ln_t_b}
		\forall (t,x) \in [0,T] \times \R^d, \ \mathbf {\tilde L}_t(\mathbf b) \varphi(t,x) := 
			\mathbb B (t,x) \cdot   \mathbf D \varphi  (t,x) 
		- \mathbf  D^2 \varphi(t,x): a(t)+\mathbb D(t,x).
		\end{equation}
		It is well-known, see \cite{frie:64}, that there are Green operator and semi-group, respectively denoted by $\mathbf G(\mathbf b)$ and $\mathbf P(\mathbf b)$, such that the solution of \eqref{KOLMO_suite} writes,
		for any $(t,x) \in [0,T] \times \R^d$,
		\begin{equation}\label{ident_tilde_u}
		\tilde {\mathbf u}(\mathbf b)(t,x) = \mathbf G(\mathbf b)\mathbf f(t,x) +\mathbf P(\mathbf b)\mathbf g(t,x)
		=: \mathscr{H}_L(\mathbf f,\mathbf g)[\mathbf b](t,x) .
		\end{equation}
	
		The aim is to prove that we can consider $\mathbf b=\tilde {\mathbf u}$ by use of Schaeffer fixed-point argument.

		\subsection{\textit{A priori} boundedness of the fixed-point}
		\label{sec_apriori_Burgers}
		
		In this section, we perform, the crucial \textit{a priori} controls of a fixed-point.
		Namely, we suppose that $\mathbf u= \mathscr{H}_L(\mathbf f,\mathbf g)[\mathbf u]$; the case $\mathbf u= \mu \mathscr{H}_L(\mathbf f,\mathbf g)[\mathbf u]$, $\mu \in [0,1)$, is even more direct. 
		In particular, the boundedness of such a $\mathbf u$
		is quiet direct thanks to Schauder estimates in Theorem \ref{THEO_SCHAU}, replacing in the upper bounds $\mathbf b$ by $\mathbb A (\mathbf u)$, $\mathbf c$ 
		by $\mathbb{C}(\mathbf u)$, and 
		 from assumptions \A{E}, \A{P${}_{\mathbb A}$}, \A{P${}_{\mathbb C}$}:
			\begin{eqnarray*} 
	\|\mathbf D^{2} \mathbf u\|_{L^\infty(C^\gamma)}
&\leq&  
\mathbf N_{\eqref{ineq_Hess_Holder_linear}}(T, \mathscr M_{\mathbb A}(\|\mathbf u \|_{L^\infty (C^1_b)}), \|\mathbf f \|_{L^\infty (C^\gamma_b)}, \|\mathbf g\|_{C^{2+\gamma}_b}, \|c\|_{L^\infty} ),
\nonumber \\
\|\mathbf D^{2} \mathbf u \|_{L^\infty} 
&\leq&   
\mathbf N_{\eqref{ineq_Hess_linear}}(T, \mathscr M_{\mathbb A}( \|\mathbf u \|_{L^\infty (C^\gamma_b)}), \|\mathbf f \|_{L^\infty (C^\gamma_b)}, \|\mathbf g\|_{C^2_b}, \|c\|_{L^\infty} )	 ,
\nonumber \\
\|\mathbf D\mathbf u\|_{L^\infty}
&\leq & 
 \mathbf N_{\eqref{ineq_grad_linear}} (T,\mathscr M_{\mathbb A}(\|\mathbf u \|_{L^\infty}), \|\mathbf f \|_{L^\infty}, \| \mathbf g \|_{C^1_b},  \|c\|_{L^\infty})
\nonumber \\
\| \mathbf u \|_{L^\infty} &\leq&   \mathbf N_{\eqref{unif_u_Feynman_Kac}}(T, \|\mathbf f\|_{L^\infty},\|\mathbf g\|_{L^\infty}, \|c\|_{L^\infty}) , 
\\ 
\|\partial_t \mathbf u\|_{L^\infty}
&\leq&  
\mathbf N_{\eqref{ineq_partial_t_unif_linear}} (T, \mathscr M_{\mathbb A}(\|\mathbf u \|_{L^\infty (C^\gamma_b)}), \|\mathbf f \|_{L^\infty (C^\gamma_b)}, \|\mathbf g\|_{C^2_b}, \|c\|_{L^\infty} )		,
\nonumber \\
\|\partial_t \mathbf u\|_{L^\infty(C^\gamma)}
&\leq&
\mathbf N_{\eqref{ineq_partial_t_Holder_linear}} (T, \mathscr M_{\mathbb A}(\|\mathbf u \|_{L^\infty (C^\gamma_b)}), \|\mathbf f \|_{L^\infty (C^\gamma_b)}, \|\mathbf g\|_{C^2_b}, \|c\|_{L^\infty} )	,
\nonumber 
	\end{eqnarray*}
	Importantly, from the fourth inequality, the boundedness of the norm $\| \mathbf D \mathbf u\|_{L^\infty}$ is already done:
	\begin{eqnarray}\label{ineq_grad_quasi}
\| \mathbf D \mathbf u \|_{L^\infty} &\leq&  
 \mathbf N_{\eqref{ineq_grad_linear}} \Big (T,\mathscr M_{\mathbb A}\big (\mathbf N_{\eqref{unif_u_Feynman_Kac}}(T, \|\mathbf f\|_{L^\infty},\|\mathbf g\|_{L^\infty}, \|c\|_{L^\infty}) \big ), \|\mathbf f \|_{L^\infty}, \| \mathbf g \|_{C^1_b},  \|c\|_{L^\infty} \Big )
\nonumber \\
&=:& \mathbf N_{\eqref{ineq_grad_quasi}} (T, \|\mathbf f \|_{L^\infty }, \|\mathbf g\|_{C^1_b},  \|c\|_{L^\infty})
.
	\end{eqnarray}
	We also deduce from interpolation inequality that
		\begin{eqnarray}\label{ineq_Holder_quasi}
	\|  \mathbf u \|_{L^\infty(C^\gamma)} &\leq&  2^{1-\gamma} \|  \mathbf u \|_{L^\infty}^{1-\gamma}\| \mathbf D \mathbf u \|_{L^\infty}^\gamma
\nonumber \\
&\leq& 2^{1-\gamma} \mathbf N_{\eqref{unif_u_Feynman_Kac}}(t, \|\mathbf f\|_{L^\infty},\|\mathbf g\|_{L^\infty}, \|c\|_{L^\infty}) ^{1-\gamma}
	 \mathbf N_{\eqref{ineq_grad_quasi}} (T, \|\mathbf f \|_{L^\infty }, \|\mathbf g\|_{C^1_b},  \|c\|_{L^\infty})^{\gamma}
	\nonumber \\
	&=:& \mathbf N_{\eqref{ineq_Holder_quasi}} (T, \|\mathbf f \|_{L^\infty }, \|\mathbf g\|_{C^1_b},  \|c\|_{L^\infty})
	.
	\end{eqnarray}
	Thanks to this inequality, we can directly derive the other Schauder estimates:
	\begin{eqnarray}\label{ineq_D2_Holder_Quasi}
		\|\mathbf D^{2} \mathbf u\|_{L^\infty(C^\gamma)}
		&\leq&  
		\mathbf N_{\eqref{ineq_Hess_Holder_linear}}(T, \mathscr M_{\mathbb A}(\mathbf N_{\eqref{ineq_Holder_quasi}} (T, \|\mathbf f \|_{L^\infty }, \|\mathbf g\|_{C^1_b},  \|c\|_{L^\infty})), \|\mathbf f \|_{L^\infty (C^\gamma_b)}, \|\mathbf g\|_{C^{2+\gamma}_b}, \|c\|_{L^\infty})
		\nonumber \\
		&=:& \mathbf N_{\eqref{ineq_D2_Holder_Quasi}}(T,  \|\mathbf f \|_{L^\infty (C^\gamma_b)}, \|\mathbf g\|_{C^{2+\gamma}_b}, \|c\|_{L^\infty}) ,
		 \\
		 \label{ineq_D2_unif_Quasi}
		\|\mathbf D^{2} \mathbf u \|_{L^\infty} 
		&\leq&   
		\mathbf N_{\eqref{ineq_Hess_linear}}(T, \mathscr M_{\mathbb A}(\mathbf N_{\eqref{ineq_Holder_quasi}} (T, \|\mathbf f \|_{L^\infty }, \|\mathbf g\|_{C^1_b},  \|c\|_{L^\infty})), \|\mathbf f \|_{L^\infty (C^\gamma_b)}, \|\mathbf g\|_{C^2_b}, \|c\|_{L^\infty})  
			\nonumber \\
		&=:& \mathbf N_{\eqref{ineq_D2_unif_Quasi}}(T,  \|\mathbf f \|_{L^\infty (C^\gamma_b)}, \|\mathbf g\|_{C^{2+\gamma}_b}, \|c\|_{L^\infty}) ,
		\\
			 \label{ineq_partial_t_unif_Quasi}
		\|\partial_t \mathbf u\|_{L^\infty}
		&\leq&  
		\mathbf N_{\eqref{ineq_partial_t_unif_linear}} (T, \mathscr M_{\mathbb A}(\mathbf N_{\eqref{ineq_Holder_quasi}} (T, \|\mathbf f \|_{L^\infty }, \|\mathbf g\|_{C^1_b},  \|c\|_{L^\infty})), \|\mathbf f \|_{L^\infty (C^\gamma_b)}, \|\mathbf g\|_{C^2_b}, \|c\|_{L^\infty})	
		\nonumber \\
			&=:& \mathbf N_{\eqref{ineq_partial_t_unif_Quasi}}(T,  \|\mathbf f \|_{L^\infty (C^\gamma_b)}, \|\mathbf g\|_{C^{\gamma}_b}, \|c\|_{L^\infty}) ,
		\\
		\label{ineq_partial_t_Holder_Quasi}
	\|\partial_t \mathbf u\|_{L^\infty(C^\gamma)}
		&\leq&
		\mathbf N_{\eqref{ineq_partial_t_Holder_linear}} (T, \mathscr M_{\mathbb A}(\mathbf N_{\eqref{ineq_Holder_quasi}} (T, \|\mathbf f \|_{L^\infty }, \|\mathbf g\|_{C^1_b},  \|c\|_{L^\infty})), \|\mathbf f \|_{L^\infty (C^\gamma_b)}, \|\mathbf g\|_{C^{2+\gamma}_b}, \|c\|_{L^\infty})
		\nonumber \\
			&=:& \mathbf N_{\eqref{ineq_partial_t_Holder_Quasi}}(T,  \|\mathbf f \|_{L^\infty (C^\gamma_b)}, \|\mathbf g\|_{C^{2+\gamma}_b}, \|c\|_{L^\infty}) .
	\end{eqnarray}
The \textit{a priori} Schauder estimates are then established.
	
	\subsection{Compactness of the parabolic operator}
	\label{sec_compact}
	
	To show that the operator is compact, we aim to obtain the sequential characteristic:
	for each bounded sequence $(\mathbf b_m)_{m \in \N}$ in $L^\infty([0,T]; C_b ^1(\R^d,\R^r))$ there is a subsequence of $\mathscr H_L(\mathbf f, \mathbf g)[\mathbf b_m]$ which converges in  $L^\infty([0,T]; C_b ^1(\R^d,\R^ r))$, i.e. $\mathscr H_L(\mathbf f, \mathbf g)[\cdot ]$ is relatively compact.

	\subsubsection{Truncation procedure}
	%
	
	The usual way to deduce convergence of a subsequence is to use Arzela-Ascoli theorem, however the starting space has to be compact.
	Hence, to apply this theorem we need to use a smooth cut-off $\theta_{y,R} \in \mathcal D$ in a ball $B_d(y,R) \subset \R^d$, the ball with $R$ as radius in $\R^d$ and $y\in \R^d$ as center, defined by
	\begin{equation}\label{def_theta}
	\theta_{y,R}(x)= \theta_y(\frac{x}{R}),
	\end{equation}
	where $\theta_y : \R^d \to [0,1]^d$ is a $C^\infty$ function such that 
	\begin{equation*}
	\theta_y(x)=\begin{cases}
	x, \ \text{ if } |x-y| < 1,	\\
	0, \ \text{ if } |x-y| > 2.
	\end{cases}
	\end{equation*}
	For $\tilde {\mathbf u}$ solution of the linear equation \eqref{KOLMO_suite}, we consider now the function defined, for any $(t,x ) \in [0,T] \times \R^d$, by
	\begin{equation}\label{def_u_R}
	\tilde {\mathbf u}_{y,R}(t,x):= \tilde {\mathbf u} (t,\theta_{y,R}(x)),
	\end{equation}
	with the important particular case
	\begin{equation}\label{u_Rxx}
	\tilde {\mathbf u}_{x,R}(t,x)= \tilde {\mathbf u} (t,x).
	\end{equation}
	But to directly get the suitable convergence when $R \to + \infty$, we need some integrability properties of $\tilde {\mathbf u}$,
	to do first let us consider the weak formulation of the parabolic equation \eqref{KOLMO} which allows to consider a truncated solution as in \eqref{def_u_R}. We pass to the strong solution in Section \ref{sec_weak_strong}.

	\subsubsection{Weak solution}
	\label{sec_weak}
	
	Let us take a smooth function $\varphi_R$ supported on $B_d(0,R)$.
	We can write a weak formulation of the solution $\mathbf u= \mathscr H_L (\mathbf f, \mathbf g)[\mathbf b]$, i.e.
	for any $(t,x) \in [0,T]\times \R^d$:
	\begin{eqnarray*}
		\int_{\R^d} \varphi_R(x-y)\big [ \partial_t \tilde {\mathbf u}(t,y) 
	+\mathbb B(t,y) \cdot \mathbf D \tilde {\mathbf u}(t,y)+\mathbb D(t,y)- \mathbf  D^2 \tilde {\mathbf u}(t,y):   a(t) \big ] dy
		&=& \int_{\R^d} \varphi_R(x-y) \mathbf f(t,y) dy ,
\nonumber \\
		\int_{\R^d} \varphi_R(x-y) \tilde {\mathbf u}(0,y) dy & =&	\int_{\R^d} \varphi_R(x-y) \mathbf g(y) dy ,
	\end{eqnarray*}
recalling that $\mathbb B$ and $\mathbb D$ are defined in \eqref{def_mathbb_B}.

	Thanks to the definition of the support of $\varphi_R$, the first line of the above equation equivalently writes
	\begin{eqnarray}
	\label{eq_weak_Kolmo}
	&&\int_{\R^d} \varphi_R(x-y)\big [ \partial_t \tilde {\mathbf u}_{x,R}(t,y) 
	+\mathbb B_{x,R}(t,y) \cdot \mathbf D \tilde {\mathbf u}_{x,R}(t,y)+\mathbb D_{x,R}(t,y)- \mathbf  D^2 \tilde {\mathbf u}_{x,R}(t,y):   a(t) \big ] dy
	\nonumber \\
	&=& \int_{\R^d} \varphi_R(x-y) \mathbf f(t,y) dy ,
	\end{eqnarray}
	with $\mathbb B_{x,R}(t,y)=\mathbb B(t,\theta_{x,R}(y))=\mathbb A(\mathbf b)(t,\theta_{x,R}(y))$ and $\mathbb D_{x,R}(t,y)=\mathbb D(t,\theta_{x,R}(y))=\mathbb C(\mathbf b)(t,\theta_{x,R}(y))$.
\\

	Hence, 
	for each bounded sequence $(\mathbf b_m)_{m \in \N}$ in $L^\infty([0,T]; C_b ^1(\R^d,\R^r))$, 	from Schauder estimates stated in Theorem \ref{THEO_SCHAU} we know that  $(t,y)\mapsto \mathscr H_L(\mathbf f, \mathbf g)[\mathbf b_m](t,\theta_{x,R}(y))=:\mathscr H_L(\mathbf f, \mathbf g)[\mathbf b_m]_{x,R}(t,y)$ is also bounded in  $L^\infty([0,T]; C_b ^1(B_d(0,R),\R^r))$; and, 
	thanks to the Arzelà–Ascoli theorem, 
	there is a subsequence of $\mathscr H_L(\mathbf f, \mathbf g)[\mathbf b_m]_{x,R}$  which converges in the Banach space $L^\infty([0,T]; C_b ^1(B_d(0,R),\R^ r))$.
	In other words, the operator $\mathscr H_L(\mathbf f, \mathbf g)[\mathbf b_m]_{x,R}$
 is relatively compact in $L^\infty([0,T]; C_b ^1(B_d(0,R),\R^ r))$.

	\subsubsection{Weak solution to strong equation}
	\label{sec_weak_Burgers}
	
	At this stage, we can apply Leray-Schauder theorem, i.e. Theorem \ref{theo_schaef} to the operator $\mathscr H_L(\mathbf f, \mathbf g)[\cdot]_{x,R}$ which is continuous and  compact in $L^\infty([0,T]; C_b ^1(B_d(0,R),\R^ r))$,  and any fixed-point is \textit{a priori} bounded, see Section \ref{sec_apriori_Burgers}; hence there is a smooth function $\mathbf u_{x,R} \in L^\infty([0,T]; C_b ^1(B_d(x,R),\R^ r))$, $x \in \R^d$, such that 
	\begin{equation*}
	\forall (t,y) \in [0,T] \times \R^d, \  \	\mathbf u_{x,R}(t,y) = \mathscr H_L(\mathbf f, \mathbf g)[\mathbf u_{x,R}]_{x,R}(t,y),
	\end{equation*}
	which is, in particular, a weak solution of quasi-linear equation writing for any $(t,x) \in [0,T] \times \R^d$:
		\begin{equation}
		\begin{cases}
	\label{eq_weak_Burgers}
	\int_{\R^d} \varphi_R(x-y)\big [ \partial_t \mathbf u_{x,R}(t,y) 
+\mathbb A(\mathbf u_{x,R})_{x,R}(t,y) \cdot \mathbf D \mathbf u_{x,R}(t,y)+\mathbb C(\mathbf u_{x,R})_{x,R}(t,y)- \mathbf  D^2 \mathbf u_{x,R}(t,y):   a(t) \big ] dy
\\
= \int_{\R^d} \varphi_R(x-y) \mathbf f(t,y) dy ,
\\
\int_{\R^d} \varphi_R(x-y) \mathbf u_{x,R}(0,y) dy =	\int_{\R^d} \varphi_R(x-y) \mathbf g(y) dy .
\end{cases}
\end{equation}
	Moreover, from the \textit{a priori} Schauder estimates performed in Section \ref{sec_apriori_Burgers}, we also get that $	\mathbf u_{x,R} \in L^\infty([0,T]; C_b ^{2+\gamma}(B_d(0,R),\R^ r)) \cap C^1_b([0,T]; C_b ^\gamma(B_d(0,R),\R^ r)) $.
	
	\subsubsection{From weak to strong solution}
	\label{sec_weak_strong}
	
	Thanks to the regularity of the above solution $\mathbf u_R$ of the weak quasi-linear equation \eqref{eq_weak_Burgers} we can expect to solve this equation point-wisely (in a strong form).
	To do so, let us introduce a smooth Dirac sequence $(\psi_m)_{m \geq 0}$ with compact support $B_d(0,R)$ such that $\int \psi_m=1$; we can choose for instance
	 the Landau example which is, in dimension $1$, defined by $\psi_m(x)= \frac{(2m+1)!}{2^{2m+1}(m!)^2}(1-\frac {x^2}{R^2})^m \mathds{1}_{x \in [-R,R]}$.
	\\
	
	The idea now is to show that we can pass to the limit when $m \to + \infty$ in the weak formulation \eqref{eq_weak_Burgers}. This is in fact possible thanks to the already known regularity of $\mathbf u_{x,R}$ stated in Theorem \ref{Theo_Quasi}.
	\\
	
Let us define for any $(t,x) \in [0,T] \times \R^d$:
	\begin{equation}\label{def_ux}
		\mathbf u(t,x) := \mathbf u_{x,R}(t,x),
	\end{equation}
	in the spirit of \eqref{u_Rxx} but for the weak solution of the quasi-linear equation defined only in a compact set. This function $	\mathbf u(t,x)$ is a good candidate to be the strong solution of the quasi-linear equation \eqref{Quasi_equation_v1}.
		\\
	
	Now, we show that replacing $\varphi_R$ by $\psi_m$ in \eqref{eq_weak_Burgers} and letting $m$ going to $+ \infty$ yields point-wisely to the strong formulation of the quasi-linear equation. Namely, we aim to prove that
	\begin{eqnarray*}
		&&\lim_{m \to + \infty}	\int_{\R^d} \psi_m (x-y)\big [ \partial_t \mathbf u_{x,R}(t,y) 
		+\mathbb A(\mathbf u_{x,R})_{x,R}(t,y) \cdot \mathbf D \mathbf u_{x,R}(t,y)+\mathbb C(\mathbf u_{x,R})_{x,R}(t,y)
		\\
	&&- \mathbf  D^2 \mathbf u_{x,R}(t,y):   a(t)	- \mathbf f(t,y) \big ] dy 
		\nonumber \\
		&=&\partial_t \mathbf u(t,x) 
		+\mathbb A(\mathbf u)(t,x) \cdot \mathbf D \mathbf u(t,x)+\mathbb C(\mathbf u)(t,x)- \mathbf  D^2 \mathbf u(t,x):   a(t) 
		- \mathbf f(t,x) .
	\end{eqnarray*}
	
	We write for any $x \in \R^d$, thanks to the well-known regularisation by convolution controls 
	and Schauder estimates the following results.
	\\
	
	$\bullet$ The time derivative part:
	\begin{eqnarray}\label{conv_weak_strong_Burgers_partial_t}
\Big |	\int_{\R^d} \psi_m(x-y) \partial_t \mathbf u_{x,R}(t,y) dy  - \partial_t \mathbf u(t,x) \Big |
	&=&
	\Big |\int_{\R^d} \psi_m(x-y)\big [  \partial_t \mathbf u_{x,R}(t,y) - \partial_t \mathbf u_{x,R}(t,x)  \big ] dy \Big |
	\nonumber \\
	 &\leq & C m^{-\gamma} \| \partial_t \mathbf u_{x,R}\|_{L^\infty(C^\gamma)}
	 	\nonumber \\
	 &\leq & C m^{-\gamma} 
	 	\mathbf N_{\eqref{ineq_partial_t_Holder_Quasi}}( T , \|\mathbf f\|_{L^\infty(C_b^\gamma)}, \|\mathbf g\|_{L^\infty(C_b^{2+\gamma})}  ) 
	 \nonumber \\
	 &\underset{m \to + \infty} \longrightarrow& 0.
	\end{eqnarray}
	
		$\bullet$ The first non-linear part, which is dealt thanks to assumptions \A{P${}_{\mathbb A}$} and \A{F}:
	\begin{eqnarray}\label{conv_weak_strong_Burgers_uDu}
	&& \Big |	\int_{\R^d} \psi_m(x-y) \mathbb A(\mathbf u_{x,R})_{x,R}(t,y) \cdot \mathbf D \mathbf u_{x,R}(t,y) dy  -\mathbb A(\mathbf u)(t,x) \cdot \mathbf D \mathbf u(t,x) \Big |
	\nonumber \\
&\leq & C m^{-\gamma} \big [ \mathscr M_{\mathbb A}(\| \mathbf u\|_{L^\infty(C^\gamma)})\| \mathbf D \mathbf u_{x,R}\|_{L^\infty}
+ \mathscr M_{\mathbb A}(\| \mathbf u\|_{L^\infty})\| \mathbf D \mathbf u_{x,R}\|_{L^\infty(C^\gamma)} \big ]
	\nonumber \\
	&\leq & C m^{-\gamma} 
	\mathbf N( T , \|\mathbf f\|_{L^\infty(C_b^\gamma)}, \|\mathbf g\|_{L^\infty(C_b^{2+\gamma})}  ) 
	\nonumber \\
	&\underset{m \to + \infty} \longrightarrow& 0,
	\end{eqnarray}
	where $\mathbf N( T , \|\mathbf f\|_{L^\infty(C_b^\gamma)}, \|\mathbf g\|_{L^\infty(C_b^{2+\gamma})}  ) 
	$ is a constant depending on $d$, $r$, $\gamma$, $\nu$, $T$, $ \|\mathbf f\|_{L^\infty(C_b^\gamma)}$, and  $\|\mathbf g\|_{L^\infty(C_b^{2+\gamma})}$ given by interpolation and by Schauder estimates.
	\\
	
	$\bullet$ The second non-linear part, thanks to assumption \A{F}:
	\begin{eqnarray}\label{conv_weak_strong_Burgers_D}
	\Big |	\int_{\R^d} \psi_m(x-y) \mathbb C(\mathbf u_{x,R})_{x,R}(t,y) dy  -\mathbb C(\mathbf u)(t,x) \Big |
	&\leq & C m^{-\gamma} \mathscr M_{\mathbb C}(\| \mathbf u\|_{L^\infty(C^\gamma)})
	\nonumber \\
	&\leq & C m^{-\gamma} \mathscr M_{\mathbb C}\Big (
		\mathbf N_{\eqref{ineq_Holder_quasi}} (T, \|\mathbf f \|_{L^\infty }, \|\mathbf g\|_{C^1_b},  \|c\|_{L^\infty}) \Big )
	\nonumber \\
	&\underset{m \to + \infty} \longrightarrow& 0.
	\end{eqnarray}

	$\bullet$ The second order part:
	\begin{eqnarray}\label{conv_weak_strong_Burgers_nu_Delta_u}
&&	\Big |	\int_{\R^d} \psi_m(x-y) \mathbf  D^2 \mathbf u_{x,R}(t,y):   a(t)  dy  - \mathbf  D^2 \mathbf u(t,x):   a(t)  \Big |
\nonumber \\
	&\leq & C \|a\|_{L^\infty}
	m^{-\gamma} \| \mathbf D^2 \mathbf u_{x,R}\|_{L^\infty(C^\gamma)}
	\nonumber \\
	&\leq & C  \|a\|_{L^\infty} m^{-\gamma} 
	\mathbf N_{\eqref{ineq_D2_Holder_Quasi}}(T , \|\mathbf f\|_{L^\infty(C_b^\gamma)}, \|\mathbf g\|_{L^\infty(C_b^{2+\gamma})}  ) 
	\nonumber \\
	&\underset{m \to + \infty} \longrightarrow& 0.
	\end{eqnarray}
	The convergence of the other contributions in \eqref{eq_weak_Burgers} are even more direct. Hence, from \eqref{conv_weak_strong_Burgers_partial_t}-
	\eqref{conv_weak_strong_Burgers_nu_Delta_u}, we obtain that $\mathbf u$ defined in \eqref{def_ux} is a strong solution to quasi-linear equation \eqref{Quasi_equation_v1} which satisfies Schauder estimates.
	
	\subsection{Uniqueness of the solution to quasi-linear equation}
	\label{sec_uniqu_Burger}
	
We establish uniqueness of the solution of the quasi-linear equation \eqref{Quasi_equation_v1} for any finite $T>0$. 
\\

Let us consider two smooth solutions $\mathbf u_1$ and $\mathbf u_2$ of \eqref{Quasi_equation_v1} lying in $L^\infty([0,T]; C_b^{2+\gamma}(\R^r\R^r))$ such that there is a positive mapping $\mathbf N_{\mathbf f, \mathbf g}(T)$ depending on $\|\mathbf f \|_{L^\infty(C^\gamma_b)}$ and $\|\mathbf g \|_{C^{2+\gamma}_b}$ (also on $r$, $d$ and $\gamma$) satisfying:
\begin{equation}\label{Uniqueness_Buergers_Schauder_u1_u2}
	\|\mathbf u_1 \|_{L^\infty(C^{2+\gamma}_b)}  + \|\mathbf u_2 \|_{L^\infty(C^{2+\gamma}_b)} \leq \mathbf  N_{\mathbf f ,\mathbf g }(T).
\end{equation}

We also define $\mathbf U= \mathbf u_1-\mathbf u_2$, which is solution, for any $(t,x) \in (0,T] \times \R^d$, of 
	 	\begin{equation}
\label{Burgers_equation_U}
\begin{cases}
\partial_t\mathbf U(t,x)+[\mathbb A (\mathbf u_1)(t,x)-\mathbb A (\mathbf u_2)(t,x)]   \cdot \mathbf D \mathbf u_1(t,x)
+\mathbb A [\mathbf u_2](t,x)   \cdot \mathbf D \mathbf U(t,x)
+[\mathbb C (\mathbf u_1)(t,x)-\mathbb C (\mathbf u_2)(t,x)]   
\\= \mathbf D^2  \mathbf U(t,x) : a(t)
,
\\
\mathbf U(0,x)=\mathbf 0.\\
\end{cases}
\end{equation}
	From Duhamel's principle, we readily derive that for any $(t,x) \in [0,T] \times \R^d$:
	\begin{equation}\label{Duahmel_U_burgers}
		\mathbf U(t,x) = \underbrace{\tilde P \mathbf U(0,x)}_{=\mathbf 0} - \tilde G  [\mathbb A (\mathbf u_1)-\mathbb A (\mathbf u_2)]  \cdot \mathbf D \mathbf u_1(t,x)
		- \tilde G \mathbb A (\mathbf u_2)  \cdot \mathbf D \mathbf U(t,x)
		- \tilde G [\mathbb C (\mathbf u_1)-\mathbb C (\mathbf u_2)] (t,x)  .
	\end{equation}
	We directly get:
	\begin{eqnarray}\label{ineq_Duhamel_U_Burgers_1}
		\| \tilde G  [\mathbb A (\mathbf u_1)-\mathbb A (\mathbf u_2)] \cdot \mathbf D \mathbf u_1(t,\cdot)\|_{L^\infty}
		&\leq& 
		C \|\mathbf D \mathbf u_1\|_{L^\infty} \int_0^t \| \mathbf U(s,\cdot)\|_{L^\infty}  \tilde {\mathscr M}_{\mathbb A}(\|\mathbf u_1\|_{L^\infty},\|\mathbf u_2\|_{L^\infty} )
	\nonumber \\
		&\leq& C \mathbf  N_{\mathbf f ,\mathbf g }(T)  \int_0^t \| \mathbf U(s,\cdot)\|_{L^\infty}  ds ,
	\end{eqnarray}
	 from assumption \A{F}.

	Also we write by integration by parts:
	\begin{eqnarray*}
	\|\tilde G \mathbb A (\mathbf u_2)  \cdot \mathbf D \mathbf U(t,\cdot)\|_{L^\infty}
&= & \sup_{t \in [0,T], \, x \in \R^d} \Big | 
\int_0^t \int_{\R^d}\mathbf D \cdot \big (\tilde p(s,t, x,y)  \mathbb A (\mathbf u_2) (s,y) \big ) \mathbf U (s,y) dy \, ds
\Big |
\nonumber \\
&\leq & C \int_0^t \big (\|\mathbf D \mathbb A (\mathbf u_2)\|_{L^\infty}+(t-s)^{-\frac{1}{2}} \|\mathbb A (\mathbf u_2)\|_{L^\infty}\big ) \|\mathbf U (s,\cdot)\|_{L^\infty} ds
.
	\end{eqnarray*}
	and by assumptions \A{P${}_{\mathbb A}$}, for $\gamma=1$, 
	\begin{eqnarray}\label{ineq_Duhamel_U_Burgers_2}
	\|\tilde G \mathbb A (\mathbf u_2)  \cdot \mathbf D \mathbf U(t,\cdot)\|_{L^\infty}
	&\leq & C \int_0^t  \big (\mathscr M_{\mathbb A}(\|\mathbf D \mathbf u_2\|_{L^\infty})+(t-s)^{-\frac{1}{2}} \mathscr M_{\mathbb A}(\|\mathbf u_2\|_{L^\infty})\big ) \|\mathbf U (s,\cdot)\|_{L^\infty} ds
	\nonumber \\
	&\leq &
	C \mathbf N_{\mathbf f ,\mathbf g }(T) \int_0^t  \|\mathbf U (s,\cdot)\|_{L^\infty} ds
	.
	\end{eqnarray}	
	We finally obtain by assumption \A{F}:
		\begin{eqnarray}\label{ineq_Duhamel_U_Burgers_3}
\| \tilde G  [\mathbb C (\mathbf u_1)(t,\cdot)-\mathbb C (\mathbf u_2)(t,\cdot)]\|_{L^\infty}
	&\leq& 
	 \int_0^t \| \mathbf U(s,\cdot)\|_{L^\infty}  \tilde {\mathscr M}_{\mathbb C}(\|\mathbf u_1\|_{L^\infty},\|\mathbf u_2\|_{L^\infty} ) ds
	\nonumber \\
	&\leq&
	\tilde {\mathscr M}_{\mathbb C}\Big ( \mathbf N_{\mathbf f ,\mathbf g }(T) 
	,	 \mathbf N_{\mathbf f ,\mathbf g }(T) 
	 \Big )
	\int_0^t \| \mathbf U(s,\cdot)\|_{L^\infty}  .
	\nonumber \\
	\end{eqnarray}
	Combining \eqref{Duahmel_U_burgers} with \eqref{ineq_Duhamel_U_Burgers_1}, \eqref{ineq_Duhamel_U_Burgers_2} and \eqref{ineq_Duhamel_U_Burgers_3} readily yields:
		\begin{equation}\label{ineq_Duhamel_U_Burgers_T}
	\| \mathbf U(t,\cdot)\|_{L^\infty}
\leq 
C \mathbf N_{\mathbf f ,\mathbf g }(T) \int_0^t  \|\mathbf U (s,\cdot)\|_{L^\infty} ds
.
	\end{equation}
	We deduce from Gr\"onwall's lemma that $	\| \mathbf U\|_{L^\infty}=0$, which,  in particular, means that $\mathbf u_1=\mathbf u_2$, the smooth solution of quasi-linear equation \eqref{Quasi_equation_v1} is then unique.
	
\mysection{A ``semi" Navier-Stokes equation}
\label{sec_semi_NS}

In this part, we set the dimension\footnote{This dimension assumption is crucial for the harmonic analysis required in the Poisson equation implied by the Leray-Hopf projector.} $d=r=3$.
We consider the following semi-linear equation, that we call ``semi" Navier-Stokes equation:
		\begin{equation}
	\label{Semi_Navier_Stokes_equation_v1}
	\begin{cases}
	\partial_t\mathbf u(t,x)+\mathbb P[\mathbf u](t,x)   \cdot \mathbf D \mathbf u(t,x)
	= \nu \Delta\mathbf u(t,x)+ \mathbf f(t,x) 
	,\ t\in (0,T],\\
	\mathbf u(0,x)=\mathbf g(x),
	\end{cases}
	\end{equation}
	where $\nu>0$ is called \textit{viscosity} in fluid mechanic context, and  $\P$ is the Leray-Hopf projector defined in \eqref{def_P}.
	
	We point out that in the non-linear contribution, the Leray projector applies on $\mathbf u$ and not on $\mathbf u\cdot \mathbf D \mathbf u$ required for the usual Navier Stokes equation, see \cite{lema:16}.
	If $\mathbf u$ is divergence free, in such a case the Cauchy problem \eqref{Semi_Navier_Stokes_equation_v1} would match with the incompressible Navier Stokes equation, but this hypothesis has no reason to be generally true...
	\\
	
For the sake of simplicity, in this section, we do not specify the exact norm  in the upper-bounds of the type $\mathbf N_{(\cdot)}(\cdot)$, as the upper-bounds require several extra norms (in Lebesgue space and some queuing controls). We write instead upper-bounds as $\mathbf N_{(\cdot)}(T,\mathbf f, \mathbf g)$, where the index still refers to the associated identity.
This is all the more relevant here, as we do not have exactly a usual parabolic bootstrap. Indeed in the Schauder estimates below, some $L^2$ norms are in the r.h.s. which is due to the required energy control.
	
		\begin{THM}[Schauder estimates for semi Navier-Stokes equation]\label{Theo_Semi}
		For  $\gamma \in (0,1)$ be given.
		For all source functions $\mathbf f \in L^\infty([0,T];  C_b^{\gamma}(\R^3,\R^3))\cap L^2([0,T]; L^2(\R^3,\R^3))$ and $\mathbf g \in C^{2+\gamma}_b(\R^3,\R^3)\cap L^2(\R^3,\R^3)$ satisfying
		\begin{equation*}
		\sup_{|\alpha|\leq 1}	\|\mathbf D^{\alpha} \mathbf  f\|_{L^\infty,\beta } + 	\sup_{|\alpha'|\leq 2}	\|\mathbf D^{\alpha '} \mathbf  g\|_{\beta } <+ \infty,
		\end{equation*}
		there is  a unique strong solution $\mathbf u \in  L^\infty([0,T];C_b^{2+\gamma}(\R^{3},\R^3)) \cap C^1_b([0,T];C_b^{\gamma}(\R^{d},\R^3)) $ of \eqref{Semi_Navier_Stokes_equation_v1} satisfying the energy estimates
			\begin{eqnarray}\label{ineq_semi_NS_Leray_theo}
		\sup_{t \in[0,T]}\|\mathbf u (t,\cdot)\|_{L^2}^2
		&\leq &  \sqrt 2 \|\mathbf g\|_{L^2}^2+ 2\|\mathbf f\|_{L^2 L^2} ^2,
		\nonumber \\
		\int_0^T  \|\mathbf D  \mathbf u(s,\cdot) \|_{L^2}^2 ds 
	&\leq& \nu^{-1}\Big ( \|\mathbf g\|_{L^2}^2  
+
\int_0^t  \|\mathbf f(s,\cdot)\|_{L^2}  e^{\frac{s}{2} }\big (\|\mathbf g\|_{L^2}^2 +
\frac{1}{2}\int_0^s \| \mathbf f(\tilde s,\cdot) \|_{L^2} ^2 d\tilde s \big )  ds \Big )
,
		\end{eqnarray}
		the Schauder estimates, 
	and	also the queueing controls
		\begin{equation}\label{ineq_Du_beta_NS_theo}
				\|\partial_t \mathbf u(t,\cdot)\|_{\beta}+\|\mathbf D^2 \mathbf u(t,\cdot) \|_{\beta}+
				\|\mathbf D \mathbf u(t,\cdot) \|_{\beta}+	\| \mathbf u (t, \cdot) \|_{\beta} 
				\leq 
				  \mathbf N_{\eqref{ineq_partial_t_u_beta_semi_NS_final}, \eqref{ineq_D2u_beta_semi_NS_final}, \eqref{ineq_Du_beta_semi_NS_final},\eqref{ineq_u_beta_semi_NS_final} } (t, \mathbf f, \mathbf g).
		\end{equation}
%
		
	\end{THM}
	
	\begin{remark}
If $\P u $ is solution of an incompressible Navier-Stokes, then we can prove that this solution is unique and is ``physically reasonable", see \cite{lema:16};
	and thanks to the queuing controls, we can readily derive the smoothness of the projection of $\mathbf u$ on the divergence free function space.
Unfortunately, to prove this assumption seems to be highly not trivial, or even false in most cases.
	\end{remark}
	
	\begin{remark}
	There is no dependency on $\nu$ in the $L^2$ norm in \eqref{ineq_semi_NS_Leray_theo} and in the uniform control
	$		\|\mathbf u\|_{L^\infty} \leq T \|\mathbf f\|_{L^\infty}+\|\mathbf g\|_{L^\infty}$, see \eqref{ineq_Linfty_semi_NS} below, then we can expect some regular behaviour in a  turbulent regime phenomenon, when $\nu \to 0$. 
	\end{remark}

	We need to specifically study this equation as the operator $\mathbb P$ does not satisfy \textit{a priori} assumptions \A{P${}_{\mathbb A}$} nor \A{F}.
	\\
	
	To use Schaeffer theorem, we need to consider continuous compact operator in Banach space with bounded fixed points.
	The proof of continuity of the consider operator is very similar to the Navier-Stokes operator,  see  Section 16.5 dans \cite{lema:16}. 
	In order to obtain compactness, we use Arzelà–Ascoli theorem, but we need to consider a starting compact space as in the previous quasi-linear case. Thanks to the queuing controls of the norm $\|\cdot \|_\beta$, we can growth the compact size to infinity in a easier way than in Section \ref{sec_compact}.
	
	\subsection{Compactness}
	\label{sec_compact_NS}
	
	The Duhamel formulation of the  solution \eqref{Semi_Navier_Stokes_equation_v1} is
	\begin{equation}
	\mathbf u(t,x) = \tilde G \mathbf f (t,x) + \tilde P \mathbf g (t,x) + \tilde G \big(\P [\mathbf u] \cdot \mathbf D \mathbf u \big )(t,x).
	\end{equation}
	Hence, the operator to consider is for any $\psi \in C^\infty_0(\R^3,\R^3)$:
	\begin{equation}
	\mathscr H \psi  (t,x) : = \tilde G \mathbf f (t,x) + \tilde P \mathbf g (t,x) + \tilde G \big (\P [\psi   ] \cdot \mathbf D \psi  \big )(t,x).
	\end{equation}
	To establish compactness of this operator, we have to show that for any bounded sequence of functions $(\psi_m)_{m \in \N}$ of the Banach space $E$ there is a subsequence of $\mathscr H \psi_m$ converging in $E$, 
	i.e. $\mathscr H$ is relatively compact.
	\\
	
	Let us choose as Banach space $E$, the space of all functions $\varphi$ lying in $L^\infty([0,T]; C ^{2+\gamma}(\R^3,\R^3)) \cap C^{1}([0,T]; C ^{\gamma}(\R^3,\R^3))$  such that
	\begin{equation}
	\|\varphi\|_{L^\infty , \beta }+	\|\mathbf D  \varphi\|_{L^\infty, \beta } <+ \infty.
	\end{equation}
	By the controls stated in Theorem \ref{Theo_Semi} and computations performed further, we see that $\mathscr H \psi_m$ is uniformly continuous.
	Therefore, by Arzelà–Ascoli theorem there is a subsequence $(\mathscr H \psi_{m_k})_{k \geq 1}$ uniformly converging in all compacts of $[0,T] \times \R^3$ toward a limit $\mathscr H \psi$ lying in $E$.
	
	In particular, for any $R>0$, $\mathscr H \psi_{m_k}$ uniformly converges on $[0,T] \times B_3(0,R)$ towards $\mathscr H \psi$. Furthermore, for any $t \in  [0,T]$
	\begin{eqnarray*}
		&&	{\lim\sup}_{k \to + \infty} \|\mathscr H \psi_{m_k}(t,\cdot) -\mathscr H \psi(t,\cdot) \|_{L^\infty}
		\nonumber \\
		&= &
		{\lim\sup}_{k \to + \infty} \sup_{|x| >R} |\mathscr H \psi_{m_k}(t,x) -\mathscr H \psi(t,x) |
		\nonumber \\
		&\leq &		{\lim\sup}_{k \to + \infty} \sup_{|x| >R} (1+|x|)^{-\beta} \big ( \|\mathscr H \psi_{m_k}(t,\cdot)\|_\beta + \|\mathscr H \psi(t,\cdot) \|_\beta \big )
		\nonumber \\
		&\leq & \frac{C_\beta}{R}\sup_{m \in \N}\|\psi_m\|_{E} \underset{R \to + \infty}{\longrightarrow}0,
	\end{eqnarray*}
	where $\|\psi_m\|_{E}$ depends on the associated norms to $E$.
	In other words,  $(\mathscr H \psi_{m_k})_{k \geq 1}$  uniformly converges on the whole space $[0,T] \times \R^3$ towards $\mathscr H \psi$. 

	\subsubsection{Energy estimates}
	
	Considering $\mathbf u$ a solution to the ``semi" Navier-Stokes equation \eqref{Semi_Navier_Stokes_equation_v1}, and taking the scalar product with $\mathbf u$ of the solution yields:
	\begin{equation*}
	\partial_t \mathbf u \cdot \mathbf u
	+ \mathbb P[\mathbf u] \cdot \mathbf D \mathbf u \cdot \mathbf u
	= \nu  \mathbf \Delta \mathbf u \cdot \mathbf u + \mathbf f \cdot \mathbf u,
	\end{equation*}
	this is equivalent to
	\begin{equation*}
	\partial_t |\mathbf u |^2
	+ \mathbb P[\mathbf u] \cdot \mathbf D |\mathbf u |^2
	= \nu  \mathbf \Delta \mathbf u \cdot \mathbf u  + \mathbf f \cdot \mathbf u.
	\end{equation*}
	We can then integrate in time and space:
	\begin{equation*}
		\|\mathbf u(t,\cdot) \|_{L^2}^2-\|\mathbf g\|_{L^2}^2
		+ \int_0^t \int_{\R^3}\mathbb P[\mathbf u] \cdot \mathbf D |\mathbf u |^2 (s,y) dy \, ds
		\nonumber \\
		= \int_0^t \int_{\R^3} \nu  \mathbf \Delta \mathbf u \cdot \mathbf u (s,y) dy \, ds 
		+ \int_0^t \int_{\R^3}  \mathbf f \cdot \mathbf u(s,y) dy \, ds.
	\end{equation*}
	By integration by parts, we derive that:
	\begin{equation*}
		\|\mathbf u(t,\cdot) \|_{L^2}^2-\|\mathbf g\|_{L^2}^2
		- \int_0^t \int_{\R^3}(\mathbf D \cdot \mathbb P[\mathbf u] ) |\mathbf u |^2 (s,y) dy \, ds
		\nonumber \\
		= -\int_0^t \int_{\R^3} \nu  |\mathbf D  \mathbf u |^2(s,y) dy \, ds 
		+ \int_0^t \int_{\R^3}  \mathbf f \cdot \mathbf u(s,y) dy \, ds.
	\end{equation*}
	From $\mathbf D \cdot \mathbb P[\mathbf u]=0$, we derive the following crucial formula
	\begin{equation}\label{eq_NS_Leray}
	\|\mathbf u (t,\cdot)\|_{L^2}^2+\int_0^t \int_{\R^3} \nu  |\mathbf D  \mathbf u |^2(s,y) dy \, ds 
	=\|\mathbf g\|_{L^2}^2 + \int_0^t \int_{\R^3}  \mathbf f \cdot \mathbf u(s,y) dy \, ds.
	\end{equation}
Next, by the Cauchy-Schwarz inequality
\begin{equation*}
\| \mathbf u (t,\cdot)\|_{L^2}^2
+\int_0^t \int_{\R^3} \nu  |\mathbf D \mathbf u |^2(s,y) dy \, ds 
\leq \|\mathbf g\|_{L^2}^2 +
\int_0^t \| \mathbf  f(s,\cdot) \|_{L^2} 
\| \mathbf u(s,\cdot) \|_{L^2}  ds,
\end{equation*}
and by Young's inequality 
\begin{equation*}
\| \mathbf u (t,\cdot)\|_{L^2}^2
\leq \|\mathbf g\|_{L^2}^2 +
\frac{1}{2}\int_0^t \| \mathbf  f(s,\cdot) \|_{L^2} ^2 ds 
+ \frac{1}{2}\int_0^t  \| \mathbf  u(s,\cdot) \|_{L^2}^2  ds.
\end{equation*}
We deduce from  Gr\"onwall's lemma:
\begin{equation*}
\| \mathbf u (t,\cdot)\|_{L^2}^2
\leq  e^{\frac{t}{2} }\Big (\|\mathbf g\|_{L^2}^2 +
\frac{1}{2}\int_0^t \| \mathbf f(s,\cdot) \|_{L^2} ^2 ds \Big ),
\end{equation*}
and so
\begin{eqnarray*}
	\int_0^t  \| \mathbf D \mathbf    u(s,\cdot) \|_{L^2}^2 ds 
	&\leq& \nu^{-1}\Big ( \|\mathbf g\|_{L^2}^2  
	+
	\int_0^t  \|\mathbf f(s,\cdot)\|_{L^2}  e^{\frac{s}{2} }\big (\|\mathbf g\|_{L^2}^2 +
	\frac{1}{2}\int_0^s \| \mathbf f(\tilde s,\cdot) \|_{L^2} ^2 d\tilde s \big )  ds \Big )
	.
\end{eqnarray*}
	
	\subsubsection{$L^\infty$ controls}
	\label{sec_Lp_semi_NS}
	
	We directly derive, as previously, from Feynman-Kac representation, see Section \ref{sec_Feynman_Kac}, because 
	the corresponding \textit{drift} $\mathbb P[\mathbf u] \in L^\infty([0,T];C_b^1(\R^3,\R^3))$, see \eqref{ineq_Du_beta_semi_NS_almost_final} further, which is smooth enough to get a probabilistic representation and the uniform control of the type \ref{unif_u_Feynman_Kac}; namely:
	\begin{equation}\label{ineq_Linfty_semi_NS}
		\|\mathbf u\|_{L^\infty} \leq T \|\mathbf f\|_{L^\infty}+\|\mathbf g\|_{L^\infty}=: \mathbf N_{\eqref{ineq_Linfty_semi_NS}}(T, \|\mathbf f\|_{L^\infty},\|\mathbf g\|_{L^\infty}).
	\end{equation} 
	
%


\subsubsection{A first control of the gradient}
	\label{sec_grad_semi_NS}
	
	From Duhamel formula, we get for any $t \in [0,T]$:
	\begin{equation*}
		\|\mathbf D \mathbf u(t,\cdot) \|_{L^\infty}\leq 
	C \nu^{-\frac 12}	t^{\frac{1}{2}} \|\mathbf f\|_{L^\infty}+\|\mathbf D  \mathbf g \|_{L^\infty} +\sup_{x \in \R^3} \Big |\int_0^t \int_{\R^3} \mathbf D \tilde p (s,t,x,y) \mathbb P[\mathbf u]\cdot \mathbf D \mathbf u(s,y) dy \, ds \Big |.
	\end{equation*}
	By H\"older inequality, we have for any $1 \leq p,q \leq + \infty$:
		\begin{equation}\label{ineq_Du_infty_semi_NS1_prelim}
	\|\mathbf D \mathbf u(t,\cdot) \|_{L^\infty}\leq 
C \nu^{-\frac 12}	t^{\frac{1}{2}} \|\mathbf f\|_{L^\infty}+\|\mathbf D  \mathbf g \|_{L^\infty} +\sup_{x \in \R^3} \int_0^t \|\mathbf D \tilde p (s,t,x,\cdot)\|_{L^q} \| \mathbb P[\mathbf u]\cdot \mathbf D \mathbf u(s,\cdot) \|_{L^p} ds .
	\end{equation}
By similar computation as \eqref{ineq_h_nu_Lp}, from \eqref{ineq_absorb}, we write
	\begin{eqnarray*}
	\|\mathbf D \tilde p (s,t,x,\cdot)\|_{L^q}
	&=&	
	\Big (\int_{\R^3} |\mathbf D \tilde p (s,t,x,y)|^q dy \Big )^{\frac{1}{q}}
	\nonumber \\
	&\leq & C [\nu (t-s)]^{-\frac{1}{2}} \Big ( \int_{\R^3} |\bar p (s,t,x,\cdot)|^q dy\Big )^{\frac{1}{q}}
		\nonumber \\
		&=&  C q^{-\frac{3}{2q}} [\nu (t-s)]^{-\frac{1}{2}+\frac{\frac{-3q}{2}+\frac{3}{2}}{q}}
		\nonumber \\
		&=&C q^{-\frac{3}{2q}}  [\nu (t-s)]^{-2+\frac{3}{2q}},
	\end{eqnarray*}
which is an integrable singularity, if -2+ $\frac{3}{2q}>-1$.
In other words, we suppose 
\begin{equation}\label{condi_Holder_inq_Du}
	q<\frac{3}{2}, \, p>\frac{\frac{3}{2}}{\frac{3}{2}-1}=3. 
\end{equation}
Next, it is clear that
\begin{equation*}
	\| \mathbb P[\mathbf u](s,\cdot)\cdot \mathbf D \mathbf u(s,\cdot) \|_{L^p}
	\leq  
	\| \mathbf D \mathbf u(s,\cdot) \|_{L^\infty}
	\| \mathbb P[\mathbf u](s,\cdot)\|_{L^p}.
\end{equation*}
To control the $L^p$ norm of the Leray projector, we need a Calderòn-Zygmund inequality, see e.g. Theorem 9.9 in \cite{gilb:trud:83}.
\begin{lemma}\label{lemme_ineq_Leray_operator_Lp}
	For any $1<p<+ \infty$,  $\varphi \in L^p(\R^3, \R^3)  $, there is a constant $\mathscr C_ p=\mathscr C_ p(p)>0$ such that
	\begin{equation}\label{ineq_Leray_projector}
	\|\mathbb P \varphi \|_{L^p}
	\leq  \mathscr C_ p (d+1) \|\varphi\|_{L^p}.
	\end{equation}
\end{lemma}
From this result, we deduce
\begin{eqnarray*}
	\| \mathbb P[\mathbf u](s,\cdot)\cdot \mathbf D \mathbf u(s,\cdot) \|_{L^p}
	&\leq &
4	\mathscr C_ p
	\| \mathbf D \mathbf u(s,\cdot) \|_{L^\infty}
	\| \mathbf u(s,\cdot)\|_{L^p}
	\nonumber \\
	&\leq &
4	\mathscr C_ p
	\| \mathbf D \mathbf u(s,\cdot) \|_{L^\infty}
	\| \mathbf u(s,\cdot)\|_{L^\infty}^{\frac{p-2}p}\| \mathbf u\|_{L^2}^{\frac{2}{p}},
\end{eqnarray*}
by interpolation inequality.
Also from \eqref{ineq_Linfty_semi_NS} and \eqref{ineq_semi_NS_Leray_theo}
\begin{eqnarray*}
	\| \mathbb P[\mathbf u](s,\cdot)\cdot \mathbf D \mathbf u(s,\cdot) \|_{L^p}
		&\leq &
	4	\mathscr C_ p
	\| \mathbf D \mathbf u(s,\cdot) \|_{L^\infty}
	\big (s \|\mathbf f\|_{L^\infty}+\|\mathbf g\|_{L^\infty} \big )^{\frac{p-2}p}
\big ( \sqrt 2 \|\mathbf g\|_{L^2}^2+ 2\|\mathbf f\|_{L^2 L^2}^{2} \big )^{\frac{2}{p}},
\end{eqnarray*}
where we denote $\|\mathbf f\|_{L^2 L^2}^{2}=\int_0^T \| \mathbf f( s,\cdot) \|_{L^2} ^2 d s$.

Hence, from \eqref{ineq_Du_infty_semi_NS1_prelim}
		\begin{eqnarray}\label{ineq_Du_infty_semi_NS2}
&&
\|\mathbf D \mathbf u(t,\cdot) \|_{L^\infty}
\nonumber \\
&\leq& C
[\nu^{-1} t]^{\frac{1}{2}} \|\mathbf f\|_{L^\infty}+\|\mathbf D  \mathbf g \|_{L^\infty}
\nonumber \\
&& + C q^{-\frac{3}{2q}}  \mathscr C_ p \int_0^t [\nu (t-s)]^{-2+\frac{3}{2q}}	\| \mathbf D \mathbf u(s,\cdot) \|_{L^\infty}
\big (s \|\mathbf f\|_{L^\infty}+\|\mathbf g\|_{L^\infty} \big )^{\frac{p-2}p}
\big ( \sqrt 2 \|\mathbf g\|_{L^2}^2+ 2\|\mathbf f\|_{L^2 L^2}^{2}\big ) ^{\frac{2}{p}} ds \nonumber.
\end{eqnarray}
By Gr\"onwall Lemma, we deduce:
		\begin{eqnarray}\label{ineq_Du_infty_semi_NS_almost_final}
&&\|\mathbf D \mathbf u(t,\cdot) \|_{L^\infty}
\nonumber \\
&\leq& 
\Big (C
[\nu^{-1} t]^{\frac{1}{2}} \|\mathbf f\|_{L^\infty}+\|\mathbf D  \mathbf g \|_{L^\infty}
\Big )
\nonumber \\
&&\times \exp \Big (C q^{-\frac{3}{2q}}  \mathscr C_ p \int_0^t [\nu (t-s)]^{-2+\frac{3}{2q}}
\big (s \|\mathbf f\|_{L^\infty}+\|\mathbf g\|_{L^\infty} \big )^{\frac{p-2}p}
\big ( \sqrt 2 \|\mathbf g\|_{L^2}^2+ 2\|\mathbf f\|_{L^2 L^2}^{2}\big ) ^{\frac{2}{p}} ds \Big )
\nonumber \\
&\leq& 
\Big (C 
[\nu^{-1} t]^{\frac{1}{2}} \|\mathbf f\|_{L^\infty}+\|\mathbf D  \mathbf g \|_{L^\infty}
\Big )
\nonumber \\
&&\times \exp \Big (C q^{-\frac{3}{2q}}  \mathscr C_ p \big (t \|\mathbf f\|_{L^\infty}+\|\mathbf g\|_{L^\infty} \big )^{\frac{p-2}p}
\big ( \sqrt 2 \|\mathbf g\|_{L^2}^2+ 2\|\mathbf f\|_{L^2 L^2}^{2}\big ) ^{\frac{2}{p}} \int_0^t (t-s)^{-2+\frac{3}{2q}}
 ds \Big ) 
 \nonumber \\
 &\leq& 
 \Big (C
 [\nu^{-1} t]^{\frac{1}{2}} \|\mathbf f\|_{L^\infty}+\|\mathbf D  \mathbf g \|_{L^\infty}
 \Big ) 
 \nonumber \\
 && \times \exp \Big ( \mathscr C_ p \frac{2q^{1-\frac{3}{2q}}  C}{3-2q} \big (t \|\mathbf f\|_{L^\infty}+\|\mathbf g\|_{L^\infty} \big )^{\frac{p-2}p}
 \big ( \sqrt 2 \|\mathbf g\|_{L^2}^2+ 2\|\mathbf f\|_{L^2 L^2}^{2}\big ) ^{\frac{2}{p}} \nu^{-2+\frac{3}{2q}}t^{\frac{3-2q}{2q}}
  \Big ).
  \nonumber \\
\end{eqnarray}
Hence, we derive
		\begin{eqnarray}\label{ineq_Du_infty_semi_NS_final}
&&\|\mathbf D \mathbf u(t,\cdot) \|_{L^\infty}
 \nonumber \\
&\leq& 
\Big (C
[\nu^{-1} t]^{\frac{1}{2}} \|\mathbf f\|_{L^\infty}+\|\mathbf D  \mathbf g \|_{L^\infty}
\Big )
\nonumber \\
&& \inf_{1\leq q < \frac{3}{2}, \ p^{-1}+q^{-1}=1}\exp \Big ( \mathscr C_ p \frac{2q^{1-\frac{3}{2q}} C}{3-2q} \big (t \|\mathbf f\|_{L^\infty}+\|\mathbf g\|_{L^\infty} \big )^{\frac{p-2}p}
\big ( \sqrt 2 \|\mathbf g\|_{L^2}^2+ 2\|\mathbf f\|_{L^2 L^2}^{2}\big ) ^{\frac{2}{p}} \nu ^{-2+\frac{3}{2q}} t^{\frac{3-2q}{2q}}
\Big )
\nonumber \\
&=:&
\mathbf N_{\eqref{ineq_Du_infty_semi_NS_final}} (t, \mathbf f, \mathbf g)
. \nonumber
\end{eqnarray}

\subsubsection{A second control of the gradient}
\label{sec_decreasing_grad_semi_NS}

From Duhamel formula and queue controls stated in \eqref{ineq_h_nu_beta}, we get for any $(t,x) \in [0,T] \times \R^3$:
\begin{equation*}
|\mathbf D \mathbf u(t,x)| \leq 
C
(1+[\nu t]^{\frac{\beta}{2}})
(1+|x|)^{-\beta} \Big ([\nu^{-1}t]^{\frac{1}{2}} \|\mathbf f\|_{\beta }+\|\mathbf D \mathbf g \|_{\beta} \Big )+ \Big |\int_0^t \int_{\R^3} \mathbf D \tilde p (s,t,x,y) \otimes \mathbb P[\mathbf u]\cdot \mathbf D \mathbf u(s,y) dy \, ds \Big |.
\end{equation*}
By H\"older inequality, we have for any $1 \leq p,q \leq + \infty$:
\begin{equation}\label{ineq_Du_infty_semi_NS1}
\|\mathbf D \mathbf u(t,\cdot) \|_{L^\infty}\leq 
C(1+[\nu t]^{\frac{\beta }{2}})
(1+|x|)^{-\beta} \Big ([\nu^{-1} t]^{\frac{1}{2}} \|\mathbf f\|_{\beta }+\|\mathbf D \mathbf g \|_{\beta} \Big )+ \int_0^t \|\mathbf D \tilde p (s,t,x,\cdot)\otimes \mathbf D \mathbf u(s,\cdot)\|_{L^q} \| \mathbb P\mathbf u \|_{L^p} ds .
\end{equation}
From \eqref{ineq_absorb}, we have
\begin{eqnarray}\label{ineq_Dp_Du_Lp}
	\|\mathbf D \tilde p (s,t,x,\cdot) \mathbf D \mathbf u(s,\cdot) \|_{L^q}
	&=&	\Big (\int_{\R^3} |\mathbf D \tilde p (s,t,x,y) \mathbf D \mathbf u(s,y)|^q dy \Big )^{\frac{1}{q}}
	\nonumber \\
	&\leq &  
C q^{-\frac{3}{2q}}  (1+[\nu t]^{\frac{\beta }{2}}) |\mathbf D \mathbf u\|_{\beta}
	 \big ( \nu (t-s) \big )^{-2+\frac{3}{2q}} 	 (1+|x|)^{- \beta },
\end{eqnarray}
from identity \eqref{ineq_h_nu_beta}, the above time singularity is an integrable singularity, if -2+ $\frac{3}{2q}>-1$.
In other words, we suppose 
\begin{equation}\label{condi_Holder_inq_Du}
q<\frac{3}{2}, \, p>\frac{\frac{3}{2}}{\frac{3}{2}-1}=3. 
\end{equation}
From Lemma \ref{lemme_ineq_Leray_operator_Lp},
\begin{equation*}
	\| \mathbb P[\mathbf u](s,\cdot) \|_{L^p}
	\leq 
	6\mathscr C_ p  \| \mathbf u(s,\cdot)\|_{L^p}
	\leq 
6 	\mathscr C_ p \| \mathbf u(s,\cdot)\|_{L^\infty}^{\frac{p-2}p}\| \mathbf u(s,\cdot)\|_{L^2}^{\frac{2}{p}},
\end{equation*}
by interpolation inequality.

Furthermore, we obtain thanks to the uniform control \eqref{ineq_Linfty_semi_NS} and to the energy control \eqref{ineq_semi_NS_Leray_theo}
\begin{equation*}
	\| \mathbb P[\mathbf u](s,\cdot) \|_{L^p}
	\leq 6 \mathscr C_ p
	\big (s \|\mathbf f\|_{L^\infty}+\|\mathbf g\|_{L^\infty} \big )^{\frac{p-2}p}
	\big ( \sqrt 2 \|\mathbf g\|_{L^2}^2+ 2\|\mathbf f\|_{L^2 L^2}^{2} \big )^{\frac{2}{p}},
\end{equation*}
Hence, from \eqref{ineq_Du_infty_semi_NS1} 	and \eqref{ineq_Dp_Du_Lp}
\begin{eqnarray}\label{ineq_Du_beta_semi_NS2}
&&
\|\mathbf D \mathbf u(t,\cdot) \|_{\beta}
\nonumber \\
&\leq& C (1+[\nu t]^{\frac{\beta }{2}})
\big ([\nu^{-1} t]^{\frac{1}{2}} \|\mathbf f\|_{\beta }+\|\mathbf D \mathbf g \|_{\beta} \big )
\nonumber \\
&&+ C q^{-\frac{3}{2q}}  \mathscr C_ p (1+[\nu t]^{\frac{\beta }{2}})\int_0^t (t-s)^{-2+\frac{3}{2q}}	\| \mathbf D \mathbf u(s,\cdot) \|_{\beta}
\big (s \|\mathbf f\|_{L^\infty}+\|\mathbf g\|_{L^\infty} \big )^{\frac{p-2}p}
\big ( \sqrt 2 \|\mathbf g\|_{L^2}^2+ 2\|\mathbf f\|_{L^2 L^2}^{2}\big ) ^{\frac{2}{p}} ds \nonumber.
\end{eqnarray}
By Gr\"onwall Lemma, we deduce:
\begin{eqnarray}\label{ineq_Du_beta_semi_NS_almost_final}
&&\|\mathbf D \mathbf u(t,\cdot) \|_{\beta}
\nonumber \\
&\leq& 
C (1+[\nu t]^{\frac{\beta }{2}}) \big ([\nu^{-1} t]^{\frac{1}{2}} \|\mathbf f\|_{\beta }+\|\mathbf D \mathbf g \|_{\beta} \big )
\nonumber \\
&&\times \exp \Big (C q^{-\frac{3}{2q}}  \mathscr C_ p (1+[\nu t]^{\frac{\beta }{2}}) \int_0^t [\nu (t-s)]^{-2+\frac{3}{2q}}
\big (s \|\mathbf f\|_{L^\infty}+\|\mathbf g\|_{L^\infty} \big )^{\frac{p-2}p}
\big ( \sqrt 2 \|\mathbf g\|_{L^2}^2+ 2\|\mathbf f\|_{L^2 L^2}^{2}\big ) ^{\frac{2}{p}} ds \Big )
\nonumber \\
&\leq& 
C(1+[\nu t]^{\frac{\beta }{2}}) \big ([\nu^{-1} t]^{\frac{1}{2}} \|\mathbf f\|_{\beta }+\|\mathbf D \mathbf g \|_{\beta} \big )
\nonumber \\
&&\times \exp \Big (Cq^{-\frac{3}{2q}} \mathscr C_ p (1+[\nu t]^{\frac{\beta }{2}}) \big (t \|\mathbf f\|_{L^\infty}+\|\mathbf g\|_{L^\infty} \big )^{\frac{p-2}p}
\big ( \sqrt 2 \|\mathbf g\|_{L^2}^2+ 2\|\mathbf f\|_{L^2 L^2}^{2}\big ) ^{\frac{2}{p}} \int_0^t [\nu (t-s)]^{-2+\frac{3}{2q}}
ds \Big ) 
\nonumber \\
&\leq& 
C (1+[\nu t]^{\frac{\beta }{2}}) \big ([\nu^{-1} t]^{\frac{1}{2}} \|\mathbf f\|_{\beta }+\|\mathbf D \mathbf g \|_{\beta} \big )
\\
&&\times \exp \Big ( \mathscr C_ p (1+[\nu t]^{\frac{\beta }{2}}) \frac{2q^{1-\frac {3}{2q}} C}{3-2q} \big (t \|\mathbf f\|_{L^\infty}+\|\mathbf g\|_{L^\infty} \big )^{\frac{p-2}p}
\big ( \sqrt 2 \|\mathbf g\|_{L^2}^2+ 2\|\mathbf f\|_{L^2 L^2}^{2}\big ) ^{\frac{2}{p}} \nu^{-2+\frac{3}{2q}} t^{\frac{3-2q}{2q}}
\Big ). \nonumber 
\end{eqnarray}
Hence, we derive
\begin{eqnarray}\label{ineq_Du_beta_semi_NS_final}
&&\|\mathbf D \mathbf u(t,\cdot) \|_{\beta}
\nonumber \\
&\leq& 
C (1+[\nu t]^{\frac{\beta }{2}}) \big ([\nu t]^{\frac{1}{2}} \|\mathbf f\|_{\beta }+\|\mathbf D \mathbf g \|_{\beta} \big )
\inf_{1\leq q < \frac{3}{2}, \ p^{-1}+q^{-1}=1}
\nonumber \\
&&
\Big \{\exp \Big ( \mathscr C_ p (1+[\nu t]^{\frac{\beta }{2}}) \frac{2q^{1-\frac{3}{2q}} C}{3-2q} \big (t \|\mathbf f\|_{L^\infty}+\|\mathbf g\|_{L^\infty} \big )^{\frac{p-2}p}
\big ( \sqrt 2 \|\mathbf g\|_{L^2}^2+ 2\|\mathbf f\|_{L^2 L^2}^{2}\big ) ^{\frac{2}{p}} \nu^{-2+\frac{3}{2q}} t^{\frac{3-2q}{2q}}
\Big ) \Big \}
\nonumber \\
&=:&
\mathbf N_{\eqref{ineq_Du_beta_semi_NS_final}} (t, \mathbf f, \mathbf g)
. 
\end{eqnarray}
From a precise analysis of the Poisson equation, it is well known that:
\begin{equation}\label{Poisson_queue}
\|\P \mathbf D \mathbf u \|_{L^\infty,\beta-2} \leq C \|\mathbf D \mathbf u \|_{\beta},
\end{equation}
see e.g. \cite{lema:16}.
We can now give a first point-wise estimate of the Leray-Hopf projector of the gradient of the solution, i.e.
\begin{equation}\label{ineq_beta_P_u_semi_NS}
	\|\P \mathbf D \mathbf u \|_{L^\infty,\beta-2} 
	\leq C \mathbf N_{\eqref{ineq_Du_beta_semi_NS_final}} (T, \mathbf f, \mathbf g).
\end{equation}

\subsubsection{A second point-wise control of the velocity}
\label{sec_decreasing_semi_NS}

From Duhamel formula, we get for any $(t,x) \in [0,T] \times \R^3$:
\begin{equation*}
| \mathbf u(t,x)| \leq 
C(1+[\nu t]^{\frac{\beta }{2}}) (1+|x|)^{-\beta} \Big (t \|\mathbf f\|_{\beta }+\|\mathbf g \|_{\beta} \Big )+\sup_{x \in \R^3} \Big |\int_0^t \int_{\R^3} \tilde p (s,t,x,y) \mathbb P[\mathbf u]\cdot \mathbf D \mathbf u(s,y) dy \, ds \Big |.
\end{equation*}
By H\"older inequality, we have for any $1 \leq p,q \leq + \infty$:
\begin{equation}\label{ineq_u_beta_semi_NS1}
| \mathbf u(t,x) |
\leq 
C (1+[\nu t]^{\frac{\beta }{2}}) (1+|x|)^{-\beta} \Big (t \|\mathbf f\|_{\beta }+\|\mathbf D \mathbf g \|_{\beta} \Big )+\sup_{x \in \R^3} \int_0^t \|\tilde p (s,t,x,\cdot)\mathbf D \mathbf u(s,\cdot)\|_{L^q} \| \mathbb P\mathbf u \|_{L^p} ds .
\end{equation}
From \eqref{ineq_absorb}, we have
\begin{eqnarray*}
	\| \tilde p (s,t,x,\cdot) \mathbf D \mathbf u(s,\cdot) \|_{L^q}
	&=&	\Big (\int_{\R^3} | \tilde p (s,t,x,y) \mathbf D \mathbf u(s,y)|^q dy \Big )^{\frac{1}{q}}
	\nonumber \\
	&\leq &  Cq^{-\frac{3}{2q}} (1+[\nu t]^{\frac{\beta }{2}})  \|\mathbf D \mathbf u\|_{\beta}
	\big ( \nu (t-s) \big )^{-\frac{3(q-1)}{2q}}
	(1+|x|)^{- \beta }
	\nonumber \\
	&\leq & C q^{-\frac{3}{2q}}  (1+[\nu t]^{\frac{\beta }{2}}) \mathbf N_{\eqref{ineq_Du_infty_semi_NS_final}} (t, \mathbf f, \mathbf g)
	\big ( \nu (t-s) \big )^{-\frac{3(q-1)}{2q}}	 (1+|x|)^{- \beta },
\end{eqnarray*}
from  \ref{ineq_h_nu_beta}, the above time singularity is an integrable singularity, if $\frac{3(q-1)}{2q}>1$.
In other words, we suppose 
\begin{equation}\label{condi_Holder_inq_u}
q<3, \, p>\frac{3}{3-1}=\frac{3}{2}. 
\end{equation}

Next, by the Calder\`on-Zygmund control stated in Lemma \ref{lemme_ineq_Leray_operator_Lp},
\begin{equation*}
\| \mathbb P[\mathbf u] (s,\cdot)\|_{L^p}
\leq 4
\mathscr C_p
\| \mathbf u(s,\cdot)\|_{L^p}
\leq 4
\mathscr C_p
\| \mathbf u(s,\cdot)\|_{L^\infty}^{\frac{p-2}p}\| \mathbf u(s,\cdot)\|_{L^2}^{\frac{2}{p}},
\end{equation*}
by interpolation inequality.

Also from \eqref{ineq_Linfty_semi_NS} and \eqref{ineq_semi_NS_Leray_theo}
\begin{equation*}
	\| \mathbb P[\mathbf u](s,\cdot) \|_{L^p}
	\leq 4
	\mathscr C_p
	\big (s \|\mathbf f\|_{L^\infty}+\|\mathbf g\|_{L^\infty} \big )^{\frac{p-2}p}
	\big ( \sqrt 2 \|\mathbf g\|_{L^2}^2+ 2\|\mathbf f\|_{L^2 L^2}^{2} \big )^{\frac{2}{p}}.
\end{equation*}
Hence, from \eqref{ineq_u_beta_semi_NS1}
\begin{eqnarray}\label{ineq_u_beta_semi_NS_almost_final}
&&
\|\mathbf u(t,\cdot) \|_{\beta}
\nonumber \\
&\leq& C (1+[\nu t]^{\frac{\beta }{2}}) \Big (
\big ([\nu^{-1}t]^{\frac{1}{2}} \|\mathbf f\|_{\beta }+\|\mathbf D \mathbf g \|_{\beta} \big )
\nonumber \\
&&+ C q^{-\frac{3}{2q}}  \mathscr C_p \int_0^t  \mathbf N_{\eqref{ineq_Du_beta_semi_NS_final}} (s, \mathbf f, \mathbf g)
\big ( \nu (t-s) \big )^{-\frac{3(q-1)}{2q}}	
\big (s \|\mathbf f\|_{L^\infty}+\|\mathbf g\|_{L^\infty} \big )^{\frac{p-2}p}
\big ( \sqrt 2 \|\mathbf g\|_{L^2}^2+ 2\|\mathbf f\|_{L^2 L^2}^{2}\big ) ^{\frac{2}{p}} ds \Big )
\nonumber \\
&\leq& C (1+[\nu t]^{\frac{\beta }{2}}) \Big (
\big ([\nu^{-1} t]^{\frac{1}{2}} \|\mathbf f\|_{\beta }+\|\mathbf D \mathbf g \|_{\beta} \big )
\\
&&+ \mathscr C_p \frac{C q^{1-\frac{3}{2q}}}{3-q}  
\nu ^{-\frac{3(q-1)}{2q}} t^{\frac{3-q}{2q}}	
 \mathbf N_{\eqref{ineq_Du_beta_semi_NS_final}} (t, \mathbf f, \mathbf g)
\big (t \|\mathbf f\|_{L^\infty}+\|\mathbf g\|_{L^\infty} \big )^{\frac{p-2}p}
\big ( \sqrt 2 \|\mathbf g\|_{L^2}^2+ 2\|\mathbf f\|_{L^2 L^2}^{2}\big ) ^{\frac{2}{p}} \Big )
\nonumber
.
\end{eqnarray}
So
\begin{eqnarray}\label{ineq_u_beta_semi_NS_final}
&&\|\mathbf u(t,\cdot) \|_{\beta}
\nonumber \\
&\leq& C
(1+[\nu t]^{\frac{\beta }{2}})
\Big ([\nu ^{-1} t]^{\frac{1}{2}} \|\mathbf f\|_{\beta }+\|\mathbf D \mathbf g \|_{\beta} 
\nonumber \\`
&&
+ \! \! \!  \inf_{1\leq q < 3, \ p^{-1}+q^{-1}=1} \mathscr C_p \frac{q^{1-\frac{3}{2q}}}{3-q}  
\nu ^{-\frac{3(q-1)}{2q}} t^{\frac{3-q}{2q}}	
\mathbf N_{\eqref{ineq_Du_beta_semi_NS_final}} (t, \mathbf f, \mathbf g)
\big (t \|\mathbf f\|_{L^\infty}+\|\mathbf g\|_{L^\infty} \big )^{\frac{p-2}p}
\big ( \sqrt 2 \|\mathbf g\|_{L^2}^2+ 2\|\mathbf f\|_{L^2 L^2}^{2}\big ) ^{\frac{2}{p}} 
\Big )
\nonumber \\
&=:&
\mathbf N_{\eqref{ineq_u_beta_semi_NS_final}} (t, \mathbf f, \mathbf g).
\end{eqnarray}

Let us also precise that, like in \eqref{Poisson_queue}, we have the crucial point-wise estimate of the Leray-Hopf projector of the solution, i.e.
\begin{equation}\label{ineq_beta_P_u_semi_NS}
\|\P \mathbf u \|_{L^\infty,\beta-2} \leq C \| \mathbf u \|_{L^\infty,\beta} \leq C \mathbf N_{\eqref{ineq_u_beta_semi_NS_final}} (T, \mathbf f, \mathbf g).
\end{equation}

\subsubsection{A first control of the Hessian}
\label{sec_D2_u_infty_semi_NS}

Still by Duhamel formula, for any $t \in [0,T]$:
\begin{eqnarray*}
\|\mathbf D^2 \mathbf u(t,\cdot) \|_{L^\infty}
&\leq& 
C \nu^{-1+ \frac{\gamma}{2}} t^{\frac{\gamma}{2}} \|\mathbf f\|_{L^\infty(C^\gamma)}
\nonumber
+\|\mathbf D^2  \mathbf g \|_{L^\infty} 
\nonumber \\
&&+\sup_{x \in \R^3} \Big |\int_0^t \int_{\R^3} \mathbf D \tilde p (s,t,x,y) \mathbf D  \big (\mathbb P[\mathbf u]\cdot \mathbf D \mathbf u \big )(s,y) dy \, ds \Big |
\nonumber \\
&\leq &
C \nu^{-1+ \frac{\gamma}{2}}t^{\frac{\gamma}{2}} \|\mathbf f\|_{L^\infty(C^\gamma)}+\|\mathbf D^2  \mathbf g \|_{L^\infty} 
+\sup_{x \in \R^3} \Big |\int_0^t \int_{\R^3} \mathbf D \tilde p (s,t,x,y) \mathbb P[\mathbf D \mathbf u]\cdot \mathbf D \mathbf u (s,y) dy \, ds \Big |
\nonumber \\
&&+\sup_{x \in \R^3} \Big |\int_0^t \int_{\R^3} \mathbf D \tilde p (s,t,x,y) \mathbb P[\mathbf u]\cdot \mathbf D^2 \mathbf u (s,y) dy \, ds \Big |
.
\end{eqnarray*}
We can use the previous point-wise controls to obtain
\begin{eqnarray*}
	\|\mathbf D^2 \mathbf u(t,\cdot) \|_{L^\infty}
	&\leq &
C	\nu^{-1+ \frac{\gamma}{2}} t^{\frac{\gamma}{2}} \|\mathbf f\|_{L^\infty(C^\gamma)}+\|\mathbf D^2  \mathbf g \|_{L^\infty} 
	+ C
	\int_0^t [\nu (t-s)]^{-\frac{1}{2}}  \|\mathbf D \mathbf u(s,\cdot)\|_\beta \|\mathbf D \mathbf u (s,\cdot) \|_{L^\infty} ds
	\nonumber \\
	&&
	+ C \int_0^t [\nu (t-s)]^{-\frac{1}{2}}  \|\mathbf u(s,\cdot)\|_\beta \|\mathbf D^2 \mathbf u (s,\cdot) \|_{L^\infty} ds
	\nonumber \\
		&\leq &
	C \nu^{-1+ \frac{\gamma}{2}}t^{\frac{\gamma}{2}} \|\mathbf f\|_{L^\infty(C^\gamma)}+\|\mathbf D^2  \mathbf g \|_{L^\infty} 
	+ C \nu^{-\frac{1}{2}}t^{\frac{1}{2}} \mathbf N_{\eqref{ineq_Du_beta_semi_NS_final}} ^2 (t, \mathbf f, \mathbf g)
	\nonumber \\
	&&
	+ C \int_0^t [\nu (t-s)]^{-\frac{1}{2}}  \mathbf N_{\eqref{ineq_u_beta_semi_NS_final}} (s, \mathbf f, \mathbf g) \|\mathbf D^2 \mathbf u (s,\cdot) \|_{L^\infty} ds
 .
\end{eqnarray*}
We finally get by Gr\"onwall's lemma:
\begin{eqnarray}\label{ineq_D2u_infty_semi_NS_final}
	&&\|\mathbf D^2 \mathbf u(t,\cdot) \|_{L^\infty}
	\nonumber \\
	&\leq &
	\Big (C \nu^{-1+ \frac{\gamma}{2}}t^{\frac{\gamma}{2}} \|\mathbf f\|_{L^\infty(C^\gamma)}+\|\mathbf D^2  \mathbf g \|_{L^\infty} 
	+ C [\nu^{-1}t]^{\frac{1}{2}} \mathbf N_{\eqref{ineq_Du_beta_semi_NS_final}} ^2 (t, \mathbf f, \mathbf g) \Big )
	\nonumber \\
	&&\times \exp \Big (
	C \int_0^t [\nu (t-s)]^{-\frac{1}{2}}  \mathbf N_{\eqref{ineq_u_beta_semi_NS_final}} (s, \mathbf f, \mathbf g) ds \Big )
	\nonumber \\
		&\leq &
	\Big (C \nu^{-1+ \frac{\gamma}{2}}t^{\frac{\gamma}{2}} \|\mathbf f\|_{L^\infty(C^\gamma)}+\|\mathbf D^2  \mathbf g \|_{L^\infty} 
	+ C [\nu^{-1}t]^{\frac{1}{2}} \mathbf N_{\eqref{ineq_Du_beta_semi_NS_final}} ^2 (t, \mathbf f, \mathbf g) \Big )
	\exp \Big (
	C [\nu^{-1} t]^{\frac{1}{2}} 
	\mathbf  N_{\eqref{ineq_u_beta_semi_NS_final}} (t, \mathbf f, \mathbf g)  \Big )
	\nonumber \\
	&=:&
\mathbf 	N_{\eqref{ineq_D2u_infty_semi_NS_final}} (t, \mathbf f, \mathbf g)  
	.
\end{eqnarray}

\subsubsection{A second control of the Hessian}
\label{sec_D2_u_beta_semi_NS}

We also have, for any $t \in [0,T]$:
\begin{eqnarray*}
		&&
	|\mathbf D^2 \mathbf u(t,x)|
	\nonumber \\
		&\leq &
	C(1+[\nu t]^{\frac{\beta }{2}})(1+|x|)^{-\beta} \big (t^{\frac{1}{2}} \|\mathbf D \mathbf  f\|_{L^\infty,\beta }
+\|\mathbf D^2  \mathbf g \|_{\beta } \big ) 	+\Big |\int_0^t \int_{\R^3} \mathbf D \tilde p (s,t,x,y) \mathbb P[\mathbf D \mathbf u]\cdot \mathbf D \mathbf u (s,y) dy \, ds \Big |
	\nonumber \\
	&&+\Big |\int_0^t \int_{\R^3} \mathbf D \tilde p (s,t,x,y) \otimes \mathbb P[\mathbf u]\cdot \mathbf D^2 \mathbf u (s,y) dy \, ds \Big |
	.
\end{eqnarray*}
We can use the previous point-wise controls to obtain
\begin{eqnarray*}
	&&\|\mathbf D^2 \mathbf u(t,\cdot) \|_{\beta}
	\nonumber \\
	&\leq &
C (1+[\nu t]^{\frac{\beta }{2}})\big ([\nu^{-1} t]^{\frac{1}{2}} \|\mathbf D \mathbf  f\|_{L^\infty(\beta )}
+\|\mathbf D^2  \mathbf g \|_{\beta } \big )
	+ C (1+[\nu t]^{\frac{\beta }{2}})
	\int_0^t [\nu (t-s)]^{-\frac{1}{2}}  \|\mathbf D \mathbf u(s,\cdot)\|_\beta \|\mathbf D \mathbf u (s,\cdot) \|_{\beta } ds
	\nonumber \\
	&&
	+C (1+[\nu t]^{\frac{\beta}{2}}) \int_0^t [\nu (t-s)]^{-\frac{1}{2}}  \|\mathbf u(s,\cdot)\|_\beta \|\mathbf D^2 \mathbf u (s,\cdot) \|_{\beta} ds
	\nonumber \\
	&\leq &
C(1+[\nu t]^{\frac{\beta}{2}}) \Big (([\nu^{-1} t]^{\frac{1}{2}} \|\mathbf D \mathbf f\|_{\beta }+\|\mathbf D^2  \mathbf g \|_{\beta } )
	+  [\nu ^{-1}t]^{\frac{1}{2}} \mathbf N_{\eqref{ineq_Du_beta_semi_NS_final}} ^2 (t, \mathbf f, \mathbf g)
	\nonumber \\
&&
	+ 
	 \int_0^t [\nu (t-s)]^{-\frac{1}{2}} \mathbf  N_{\eqref{ineq_u_beta_semi_NS_final}} (s, \mathbf f, \mathbf g) \|\mathbf D^2 \mathbf u (s,\cdot) \|_{L^\infty} ds \Big )
	.
\end{eqnarray*}
We finally get by Gr\"onwall's lemma:
\begin{eqnarray}\label{ineq_D2u_beta_semi_NS_final}
\|\mathbf D^2 \mathbf u(t,\cdot) \|_{\beta}
&\leq &
C(1+[\nu t]^{\frac{\beta }{2}}) \Big ([\nu^{-1}t]^{\frac{1}{2}} \|\mathbf D \mathbf  f\|_{\beta }+\|\mathbf D^2  \mathbf g \|_{\beta} 
+ C [\nu^{-1}t]^{\frac{1}{2}} \mathbf N_{\eqref{ineq_Du_beta_semi_NS_final}} ^2 (t, \mathbf f, \mathbf g) \Big )
\nonumber \\
&& \times \exp \Big (
C(1+[\nu t]^{\frac{\beta }{2}}) \int_0^t [\nu (t-s)]^{-\frac{1}{2}}  N_{\eqref{ineq_u_beta_semi_NS_final}} (s, \mathbf f, \mathbf g) ds \Big )
\nonumber \\
&\leq &
C(1+[\nu t]^{\frac{\beta }{2}}) \Big ([\nu^{-1}t]^{\frac{1}{2}} \|\mathbf D \mathbf f\|_{\beta }+\|\mathbf D^2  \mathbf g \|_{\beta} 
+ C t^{\frac{1}{2}} \mathbf N_{\eqref{ineq_Du_beta_semi_NS_final}} ^2 (t, \mathbf f, \mathbf g) \Big )
\nonumber \\
&& \times \exp \Big ( C(1+[\nu t]^{\frac{\beta }{2}})
[\nu^{-1}t]^{\frac{1}{2}} \mathbf  N_{\eqref{ineq_u_beta_semi_NS_final}} (t, \mathbf f, \mathbf g)  \Big )
\nonumber \\
&=:&
\mathbf N_{\eqref{ineq_D2u_beta_semi_NS_final}} (t, \mathbf f, \mathbf g)  
.
\end{eqnarray}

\subsubsection{Control of the H\"older modulus of Hessian}
\label{sec_D2_u_Holder_semi_NS}

Similarly, we obtain
\begin{eqnarray}\label{ineq_NS_D2_Holder1}
		[\mathbf D^2 \mathbf u(t,\cdot ) ]_\gamma
	&\leq& 
\nu^{-1+ \frac{\gamma}{2}}(1+\nu^{-\frac{1}{2}})  \|\mathbf f\|_{L^\infty(C^\gamma)}+\|\mathbf D^2  \mathbf g \|_{L^\infty(C^\gamma)} 
\nonumber \\
&&+ \Big [ \int_0^t \int_{\R^3} \mathbf D \tilde p (s,t,\cdot ,y) \mathbf D  \big (\mathbb P[\mathbf u]\cdot \mathbf D \mathbf u \big )(s,y) dy \, ds \Big ]_\gamma
	.
\end{eqnarray}
Let us separate the diagonal and the off-diagonal cases as in Appendix.
For any $(x,x') \in \R^3\times \R^3$, we first write 
\begin{eqnarray}\label{ineq_Holder_D2_u_semi_NS1}
	&& \Big  | \int_{t+|x-x'|^2}^t \int_{\R^3} \big [ \mathbf D^2  \tilde p (s,t,x,y)-\mathbf D^2  \tilde p (s,t,x',y) \big ] \mathbb P[\mathbf D \mathbf u]\cdot \mathbf D \mathbf u (s,y) dy \, ds 
	 \Big |
	\nonumber \\
	&\leq & |x-x'| \times  \|\P[\mathbf D \mathbf u]\cdot \mathbf D \mathbf u (s,y)  \|_{L^\infty(C^\gamma)} \int_{t+|x-x'|^2}^t [\nu(t-s)]^{\frac{\gamma-3}{2}} ds
	\nonumber \\
	&\leq &
	C \nu^{\frac{\gamma-3}{2}}|x-x'|^\gamma 
	 \|\P[ \mathbf u]\cdot \mathbf D \mathbf u (s,y)  \|_{L^\infty}^{1-\gamma}
	 \big (   \|\P[\mathbf D \mathbf u]\cdot \mathbf D \mathbf u (s,y)  \|_{L^\infty}
	 + \|\P[ \mathbf u]\cdot \mathbf D ^2\mathbf u (s,y)  \|_{L^\infty} \big )^\gamma
	  \nonumber \\
	  &\leq &
	  C \nu^{\frac{\gamma-3}{2}} |x-x'|^\gamma 
	\mathbf N_{\eqref{ineq_u_beta_semi_NS_final}}^{1-\gamma} (t, \mathbf f, \mathbf g) 
	\mathbf N_{\eqref{ineq_Du_infty_semi_NS_final}}^{1-\gamma} (t, \mathbf f, \mathbf g)
	\Big ( 	C \mathbf N_{\eqref{ineq_Du_beta_semi_NS_final}}^2 (t, \mathbf f, \mathbf g)
	+C \mathbf N_{\eqref{ineq_u_beta_semi_NS_final}} (t, \mathbf f, \mathbf g) \mathbf N_{\eqref{ineq_D2u_infty_semi_NS_final}} (t, \mathbf f, \mathbf g)  \Big )^\gamma
	\nonumber \\
	&=:& 
	|x-x'|^\gamma
\mathbf 	N_{\eqref{ineq_Holder_D2_u_semi_NS1}} (t, \mathbf f, \mathbf g)
	.
\end{eqnarray}
Next,
\begin{eqnarray}\label{ineq_Holder_D2_u_semi_NS2}
&& \Big  | \int_0^{t+|x-x'|^2} \int_{\R^3} \big [ \mathbf D^2  \tilde p (s,t,x,y)-\mathbf D^2  \tilde p (s,t,x',y) \big ] \mathbb P[\mathbf D \mathbf u]\cdot \mathbf D \mathbf u (s,y) dy \, ds 
\Big |
\nonumber \\
&\leq & \|\P[\mathbf D \mathbf u]\cdot \mathbf D \mathbf u (s,y)  \|_{L^\infty(C^\gamma)} \int_{t+|x-x'|^2}^t [\nu(t-s)]^{\frac{\gamma-2}{2}} ds
\nonumber \\
&\leq &
C \nu ^{\frac{\gamma-2}{2}} |x-x'|^\gamma 
\mathbf N_{\eqref{ineq_u_beta_semi_NS_final}}^{1-\gamma} (t, \mathbf f, \mathbf g) 
\mathbf N_{\eqref{ineq_Du_infty_semi_NS_final}}^{1-\gamma} (t, \mathbf f, \mathbf g)
\Big ( 	C \mathbf N_{\eqref{ineq_Du_beta_semi_NS_final}}^2 (t, \mathbf f, \mathbf g)
+C \mathbf N_{\eqref{ineq_u_beta_semi_NS_final}} (t, \mathbf f, \mathbf g)\mathbf N_{\eqref{ineq_D2u_infty_semi_NS_final}} (t, \mathbf f, \mathbf g)  \Big )^\gamma
\nonumber \\
&=:& 
|x-x'|^\gamma
\mathbf N_{\eqref{ineq_Holder_D2_u_semi_NS2}} (t, \mathbf f, \mathbf g)
.
\end{eqnarray}
Hence, gathering \eqref{ineq_NS_D2_Holder1}-\eqref{ineq_Holder_D2_u_semi_NS2}
\begin{eqnarray}\label{ineq_Holder_D2_u_semi_NS}
[\mathbf D^2 \mathbf u(t,\cdot ) ]_\gamma
	&\leq& 
C\nu^{-1+ \frac{\gamma}{2}}(1+\nu^{-\frac{1}{2}})   \|\mathbf f\|_{L^\infty(C^\gamma)}+C \|\mathbf D^2  \mathbf g \|_{L^\infty(C^\gamma)} + \mathbf N_{\eqref{ineq_Holder_D2_u_semi_NS1}} (t, \mathbf f, \mathbf g)
+ \mathbf N_{\eqref{ineq_Holder_D2_u_semi_NS2}} (t, \mathbf f, \mathbf g)
\nonumber \\
&=:& \mathbf N_{\eqref{ineq_Holder_D2_u_semi_NS}} (t, \mathbf f, \mathbf g).
\end{eqnarray}

\subsubsection{A control of the time derivative}
\label{sec_partial_t_u_beta_semi_NS}

For any $t \in [0,T]$, it is direct that:
\begin{eqnarray}\label{ineq_partial_u_unif}
	|\partial_t \mathbf u(t,x)|
	&\leq &
	C (1+[\nu t]^{\frac{\beta }{2}}) (1+|x|)^{-\beta} \big (\nu^{-\frac{1}{2}} t^{\frac{1}{2}} \|\mathbf D \mathbf  f\|_{L^\infty(\beta )}
	+\|\mathbf D^2  \mathbf g \|_{\beta } \big ) 	+\Big | \mathbb P[\mathbf u]\cdot \mathbf D \mathbf u (t,x)\Big |
	\nonumber \\
	&&+\Big |\int_0^t \int_{\R^3} \partial_t \tilde p (s,t,x,y) \mathbb P[\mathbf u]\cdot \mathbf D \mathbf u (s,y) dy \, ds \Big |
	.
\end{eqnarray}
As previously, we get:
\begin{eqnarray}\label{ineq_partial_t_u_beta_semi_NS_final}
	\|\partial_t \mathbf u(t,\cdot) \|_{\beta}
	&\leq &
C	(1+[\nu t]^{\frac{\beta }{2}}) \Big (
\big (\nu^{-\frac{1}{2}} t^{\frac{1}{2}} \|\mathbf D \mathbf  f\|_{L^\infty(\beta )}
+\|\mathbf D^2  \mathbf g \|_{\beta } \big )
+  \mathbf N_{\eqref{ineq_u_beta_semi_NS_final}} (t, \mathbf f, \mathbf g)
\mathbf N_{\eqref{ineq_Du_beta_semi_NS_final}} (t, \mathbf f, \mathbf g)
\nonumber \\	
&&	+   [\nu^{-1}t]^{\frac{1}{2}} \big ( \mathbf N_{\eqref{ineq_Du_beta_semi_NS_final}} ^2 (s, \mathbf f, \mathbf g)
	+   \mathbf  N_{\eqref{ineq_u_beta_semi_NS_final}} (s, \mathbf f, \mathbf g) \mathbf N_{\eqref{ineq_D2u_beta_semi_NS_final}} (t, \mathbf f, \mathbf g) \big ) \Big )
\nonumber \\
&=:&
\mathbf N_{\eqref{ineq_partial_t_u_beta_semi_NS_final}} (t, \mathbf f, \mathbf g)	.
\end{eqnarray}

\subsubsection{Spatial H\"older modulus of the time derivative}
\label{sec_partial_t_u_Holder_semi_NS}

The last estimate readily derives from Section \ref{sec_D2_u_Holder_semi_NS}:
\begin{equation}\label{ineq_Holder_partial_t_u_semi_NS}
[\partial_t \mathbf u(t,\cdot ) ]_\gamma
\leq 
\|\mathbf f\|_{L^\infty(C_b^\gamma)}+\|\mathbf D^2  \mathbf g \|_{L^\infty(C^\gamma)} + \mathbf N_{\eqref{ineq_Holder_D2_u_semi_NS1}} (t, \mathbf f, \mathbf g)
+\mathbf N_{\eqref{ineq_Holder_D2_u_semi_NS2}} (t, \mathbf f, \mathbf g)
=: \mathbf N_{\eqref{ineq_Holder_partial_t_u_semi_NS}} (t, \mathbf f, \mathbf g).
\end{equation}

The Schauder estimates are then established.


\mysection{Non-linear equations}
\label{sec_non_line}

For $0<r\leq d$, we consider the non-linear equation defined for a given $T>0$ (arbitrary big) by
\begin{equation}
\label{Non_linear_equation_v1}
\begin{cases}
\partial_t\mathbf u(t,x)+  \mathbf P (  \mathbf u, \mathbf D \mathbf u(t,x) ) + \mathbf c(t) \otimes \mathbf u(t,x) 
= \mathbf  D^2 \mathbf u(t,x):   a(t)+ \mathbf f(t,x) 
,\ t\in (0,T],\\
\mathbf u(0,x)=\mathbf g(x),
\end{cases}
\end{equation}
where $\mathbf P$ is a locally bounded function.
We point out that here the first input  of $\mathbf P$ does not depend on the current point $x$ and the second input the non-linearity is supposed to be local.
\\

For the sake of simplicity, we do not consider a general non-linearity $\mathbb C(\mathbf u)$ as in quasi-linear equation \eqref{Quasi_equation_v1}.
 For such a non-linearity, we would consider a hypothesis of the kind $\|\mathbf D \mathbb C(\mathbf u)\|_{L^\infty}\leq \|\mathbf c\|_{L^\infty}\|\mathbf D \mathbf u\|_{L^\infty}$.
\\

	\textbf{\large{Assumptions}}
\begin{trivlist}
	\item[\A{P${}_{\mathbf P}$}]
	There is a non-negative real function $\mathscr M_{\mathbf P}: \R_+\longrightarrow \R_+$ locally bounded such that, for all $\mathbf b \in C^\infty_0([0,T] \times \R^d, \R^r)$, $\mathbf c \in C^\infty_0([0,T] \times \R^d, \R^r)$ and $\gamma \in (0,1]$), 
	\begin{eqnarray}\label{hypo_A}
	\|(t,x) \mapsto \mathbf P(\mathbf c, \mathbf b(t,x)))\|_{L^\infty} &\leq &\mathscr M_{\mathbf P}(\|\mathbf b\|_{L^\infty}) (1+ \|\mathbf c\|_{L^\infty} ),
	\nonumber \\
	\|(t,x)\mapsto \mathbf P(\mathbf c,\mathbf b(t,x))\|_{L^\infty(C_b^\gamma)} &\leq & \mathscr M_{\mathbf P} ( \|\mathbf b\|_{L^\infty(C_b^\gamma)})(1+ \|\mathbf c\|_{L^\infty} ),
	\end{eqnarray}
	with, for $\gamma=1$, 	$\|(t,x) \mapsto \mathbf D \mathbf P(\mathbf c,\mathbf b(t,x))\|_{L^\infty} \leq  \mathscr M_{\mathbf P} ( \|\mathbf D \mathbf b\|_{L^\infty})(1+ \| \mathbf c(t,\cdot )\|_{C^1_b} )$, with $\mathbf D \mathbf P(\mathbf c,\mathbf b(t,x))$ the Gateau  
	derivative of $b \rightarrow  \mathbf P(\mathbf c,\mathbf b(t,x))$ for a given $\mathbf c$.
	\item[\A{F${}_{\mathbf P}$}]
	There is a non-negative  function $\tilde{\mathscr M}_{\mathbf P}: \R_+^4\longrightarrow \R_+$ locally bounded such that, for all $\mathbf b_1, \mathbf b_2 \in C^\infty_0([0,T] \times \R^d, \R^r)$ and $\mathbf c_1, \mathbf c_2 \in C^\infty_0([0,T] \times \R^d, \R^r)$, 
	\begin{equation}\label{hypo_A_func}
	\|\mathbf P(\mathbf c_1, \mathbf b_1)-\mathbf P(\mathbf c_2,\mathbf b_2)\|_{L^\infty} \leq \|\mathbf b_1-\mathbf b_2\|_{L^\infty} \tilde {\mathscr M}_{\mathbf P}(\|\mathbf b_1\|_{L^\infty}, \|\mathbf b_2\|_{L^\infty},\|\mathbf c_1\|_{L^\infty}, \|\mathbf c_2\|_{L^\infty}).
	\end{equation}
\end{trivlist}

\begin{THM}\label{THEO_SCHAU_NON}
	We suppose \A{E}, \A{P${}_{\mathbf P}$} and \A{F${}_{\mathbf P}$}.
For  $\gamma \in (0,1)$ be given.
For all $\mathbf f \in L^\infty([0,T];  C_b^{1}(\R^d,\R^r))$, $\mathbf g \in C^{2+\gamma}_b(\R^d,\R^r)$ and $\mathbf c \in L^\infty([0,T], \R^r)$, 
there is  a unique strong solution $\mathbf u \in  L^\infty([0,T];C_b^{2+\gamma}(\R^{d},\R^r)) \cap C^1_b([0,T];C_b^{k+\gamma}(\R^{d},\R^r)) $ of \eqref{Non_linear_equation_v1}. 

\end{THM}
\begin{proof}[Proof of Theorem \ref{THEO_SCHAU_NON} ]
	Continuity and compactness of the associated operator is very similar to the quasi-linear case. In the current proof, we detail the \textit{a priori} controls.
	Nevertheless, in the current non-linear case, we first need to upper-bound the gradient.
	\\
	
	\textbf{Control of $\|\mathbf D \mathbf u (t,\cdot)\|_{L^\infty}$}
	\\
	
	Let us denote $\mathbf v= \mathbf D \mathbf u$, thanks to chain rules, we get 
		\begin{equation*}
	\mathbf D \Big ( \mathbf P (  \mathbf u, \mathbf u(t,x) ) \Big ) =
	\mathbf D \mathbf P (  \mathbf u, \mathbf D \mathbf u(t,x) )\cdot \mathbf D \mathbf v(t,x),
	\end{equation*}
	where $\mathbf D \mathbf P $ stands for the Gateau 
	derivative of the operator  w.r.t. the second entry.
	By differentiating the Cauchy problem \eqref{Non_linear_equation_v1},  we then derive for any $x \in \R^d$,
	\begin{equation}
	\label{Derivative_Non_linear_equation_v1}
	\begin{cases}
	\partial_t\mathbf v(t,x)+  \mathbf D \mathbf P (  \mathbf u, \mathbf D \mathbf u(t,x) )  \cdot \mathbf D \mathbf v(t,x)
	 + \mathbf c(t) \otimes \mathbf v(t,x) 
	= \mathbf  D^2 \mathbf v(t,x):   a(t)+ \mathbf D \mathbf  f(t,x) 
	,\ t\in (0,T],\\
	\mathbf v(0,x)=\mathbf D \mathbf  g(x).
	\end{cases}
	\end{equation}
	From the Feynman-Kac representation, see Section \ref{sec_Feynman_Kac}, we directly derive that
	\begin{equation*}
		\|\mathbf v (t,\cdot) \|_{L^\infty}
		= \|\mathbf D \mathbf u (t,\cdot)\|_{L^\infty}
		\leq t \|\mathbf D \mathbf f\|_{L^\infty}+ \|\mathbf D \mathbf g\|_{L^\infty}+ \int_{0}^t \mathbf |\mathbf c(s)| 	\|\mathbf v (s,\cdot) \|_{L^\infty} ds.
	\end{equation*}
	Let us point out that we cannot easily perform a control of an integration by part to consider less regularity of the source functions. Indeed, the probability density relies on the stochastic process associated with the Kolmogorov equation \eqref{Derivative_Non_linear_equation_v1} whose the gradient behaviour may depend 
	on the corresponding drift (which also depends on $\mathbf u$), see for instance \cite{aron:59} and \cite{frie:64}.
	\\
	
	Next, the Gr\"onwall's lemma yields
	\begin{eqnarray}\label{ineq_Du_non_linear}
|\mathbf D \mathbf u (t,\cdot)\|_{L^\infty} &=&	\|\mathbf v (t,\cdot) \|_{L^\infty}
	 \nonumber \\
	&\leq& \big (t \|\mathbf D \mathbf f\|_{L^\infty}+ \|\mathbf D \mathbf g\|_{L^\infty}\big ) \exp \big (\int_{0}^t |\mathbf c(s)|  ds \big ) 
	\nonumber \\
	&:=& \mathbf N_{\eqref{ineq_Du_non_linear}}(t,\|\mathbf D \mathbf f\|_{L^\infty}, \|\mathbf D \mathbf g\|_{L^\infty}, \|\mathbf c\|_{L^\infty}).
	\end{eqnarray}
	
	\textbf{Control of $\| \mathbf u (t,\cdot)\|_{L^\infty}$}
	\\
	
	Coming back to the initial Cauchy problem \eqref{Non_linear_equation_v1}, we derive from Duhamel formula
		\begin{equation*}
	\|\mathbf u (t,\cdot) \|_{L^\infty}
	\leq t \|\mathbf f\|_{L^\infty}+ \|\mathbf g\|_{L^\infty}+ \int_0^t \|\mathbf P(\mathbf u(s),\mathbf D \mathbf u(s,\cdot ))\|_{L^\infty}ds 
	+ \int_{0}^t |\mathbf c(s)| \|\mathbf u (s,\cdot)\|_{L^\infty} ds .
	\end{equation*}
	By assumption \A{P${}_{\mathbf P}$}, we get
		\begin{equation*}
	\|\mathbf u (t,\cdot) \|_{L^\infty}
	\leq t \|\mathbf f\|_{L^\infty}+ \|\mathbf g\|_{L^\infty}+ \int_0^t \mathscr M_{\mathbf P}(\|\mathbf D \mathbf u(s,\cdot )\|_{L^\infty})(1+\|\mathbf u(s,\cdot )\|_{L^\infty})ds 
	+ \int_{0}^t |\mathbf c(s)| \|\mathbf u (s,\cdot)\|_{L^\infty} ds ,
	\end{equation*}
	and by Gr\"onwall lemma
	\begin{eqnarray}\label{ineq_u_non_linear}
	\|\mathbf u (t,\cdot) \|_{L^\infty}
	&\leq& \Big (t \|\mathbf f\|_{L^\infty}+ \|\mathbf g\|_{L^\infty}+ \int_0^t \mathscr M_{\mathbf P}(\|\mathbf D \mathbf u(s,\cdot )\|_{L^\infty})ds \Big )
	\exp \Big ( \int_0^t \mathscr M_{\mathbf P}(\|\mathbf D \mathbf u(s,\cdot )\|_{L^\infty})+ |\mathbf c(s)| ds \Big )
	\nonumber \\
	&\leq &
	 \Big (t \|\mathbf f\|_{L^\infty}+ \|\mathbf g\|_{L^\infty}+ \int_0^t \mathscr M_{\mathbf P}\big (\mathbf N_{\eqref{ineq_Du_non_linear}}(t,\|\mathbf D \mathbf f\|_{L^\infty}, \|\mathbf D \mathbf g\|_{L^\infty}, \|\mathbf c\|_{L^\infty})\big )ds \Big )
\nonumber \\
&&\times 	\exp \Big ( \int_0^t \mathscr M_{\mathbf P}\big (\mathbf N_{\eqref{ineq_Du_non_linear}}(t,\|\mathbf D \mathbf f\|_{L^\infty}, \|\mathbf D \mathbf g\|_{L^\infty}, \|\mathbf c\|_{L^\infty})\big )+ |\mathbf c(s)| ds \Big )
	\nonumber \\
	&=:& \mathbf N_{\eqref{ineq_u_non_linear}}(t,\| \mathbf f\|_{L^\infty (C^1_b)}, \| \mathbf g\|_{C^1_b}, \|\mathbf c\|_{L^\infty})
	.
	\end{eqnarray}
	
	\textbf{Control of $\|\mathbf D^2 \mathbf u (t,\cdot)\|_{L^\infty}$}
	\\
	
	For the Hessian estimates, we differentiate twice Duhamel formulation and we get
			\begin{eqnarray*}
&&	
\|\mathbf D^2 \mathbf  u (t,\cdot) \|_{L^\infty}
\nonumber \\
	&\leq&  C [\nu^{-1} t]^{\frac{1}{2}}\|\mathbf D \mathbf f\|_{L^\infty}+ \|\mathbf D^2 \mathbf g\|_{L^\infty}+C \int_0^t [\nu (t-s)]^{-\frac{1}{2}} \|\mathbf D \mathbf P(\mathbf u(s),\mathbf D \mathbf u(s,\cdot ))\|_{L^\infty} \|\mathbf D^2 \mathbf u(s,\cdot )\|_{L^\infty} ds 
	\nonumber \\
	&&+C  \int_{0}^t [\nu (t-s)]^{-\frac{1}{2}} |\mathbf c(s)| \times \|\mathbf D \mathbf  u (s,\cdot)\|_{L^\infty} ds 
	\nonumber \\
		&\leq&  C [\nu^{-1} t]^{\frac{1}{2}}\|\mathbf D \mathbf f\|_{L^\infty}+ \|\mathbf D^2 \mathbf g\|_{L^\infty}
		\nonumber \\
		&&+ \int_0^t [\nu (t-s)]^{-\frac{1}{2}} \mathscr M_{ \mathbf P}(\|\mathbf D \mathbf u(s,\cdot )\|_{L^\infty})\big (1+\mathbf N_{\eqref{ineq_u_non_linear}}(t,\| \mathbf f\|_{L^\infty (C^1_b)}, \| \mathbf g\|_{C^1_b}, \|\mathbf c\|_{L^\infty}) \big ) \|\mathbf D^2 \mathbf u(s,\cdot )\|_{L^\infty} ds 
	\nonumber \\
	&&+C  \|\mathbf c\|_{L^\infty} \mathbf N_{\eqref{ineq_Du_non_linear}}(t,\| \mathbf D \mathbf f\|_{L^\infty (C^1_b)}, \| \mathbf D \mathbf g\|_{C^1_b}, \|\mathbf c\|_{L^\infty}) ;
		\end{eqnarray*}
also Gr\"onwall's lemma yields
				\begin{eqnarray}\label{ineq_D2_non_linear}
		\|\mathbf D^2 \mathbf  u (t,\cdot) \|_{L^\infty}
		&\leq&  \Big ([\nu ^{-1} t]^{\frac{1}{2}}\|\mathbf D \mathbf f\|_{L^\infty}+ \|\mathbf D^2 \mathbf g\|_{L^\infty}+ + \|\mathbf c\|_{L^\infty} \mathbf N_{\eqref{ineq_u_non_linear}}(t,\| \mathbf f\|_{L^\infty (C^1_b)}, \| \mathbf g\|_{C^1_b}, \|\mathbf c\|_{L^\infty}) \Big )
		\nonumber \\
		&&\times \exp \Big ( \int_0^t [\nu (t-s)]^{-\frac{1}{2}} \mathscr M_{\mathbf D \mathbf P} \big (N_{\eqref{ineq_Du_non_linear}}(t,\|\mathbf D \mathbf f\|_{L^\infty}, \|\mathbf D \mathbf g\|_{L^\infty},\|\mathbf c\|_{L^\infty})\big )
		\nonumber \\
		 &&\big (1+\mathbf N_{\eqref{ineq_u_non_linear}}(t,\| \mathbf f\|_{L^\infty (C^1_b)}, \| \mathbf g\|_{C^1_b},\|\mathbf c\|_{L^\infty}) \big )  ds \Big )
		\nonumber \\
		&=:& \mathbf N_{\eqref{ineq_D2_non_linear}}(t,\| \mathbf f\|_{L^\infty (C^1_b)}, \| \mathbf g\|_{C^1_b}, \|\mathbf c\|_{L^\infty} ).
	\end{eqnarray}
	The remaining Schauder estimates directly derives from Theorem \ref{THEO_SCHAU}.
	
\end{proof}

	\appendix

	\mysection{Reminders on the heat kernel properties}
	\label{sec_technical_lemma}

Let us detail now the proof of Proposition \ref{prop_norm_hold_heat_kernel} thanks to the recalled properties stated in Section \ref{sec_Gaussian_properties}.
	\begin{proof}[Proof of Proposition \ref{prop_norm_hold_heat_kernel}]
		
		For all $x \in \R^d$, $t \in [0,T]$ and $\alpha \in \N^d_0$, we recall
		\begin{equation*}
		|D_x^\alpha\tilde G \zeta (t,x)|
		=
		\Big | \int_0^t \int_{\R^d} D_x^\alpha \tilde p (s,t,x,y)  \zeta(s,y) dy \ ds \Big |.
		\end{equation*}
	The first uniform norm is direct.
		\\
		
		\textbf{Uniform norms of the spatial derivatives}
		\\
		
		
		\begin{itemize}
			
			\item Control of $\|D \tilde G \zeta\|_{L^\infty} $

			By cancellation , i.e. for any $s \in [0,t] $ we have $ \int_{\R^d} D_x \tilde p (s,t,x,y)  \zeta(s,x) dy =0$ , we get
			\begin{eqnarray*}
				|D \tilde G \zeta (t,x)|
				&=&
				\Big | \int_0^t \int_{\R^d} D_x \tilde p (s,t,x,y) [ \zeta(s,y)-\zeta(s,x)] dy \ ds \Big |
				\nonumber \\
				&\leq & C
				\|\zeta\|_{L^\infty(C^\gamma)}
				\int_0^t \int_{\R^d} [\nu (t-s)]^{-\frac 12} \bar p (s,t,x,y)|y-x|^\gamma dy \ ds ,
			\end{eqnarray*}
			from \eqref{FIRST_deriv_CTR_DENS} and the regularity of $\zeta$.
			We deduce from \eqref{ineq_absorb}:
			\begin{eqnarray}\label{ineq_grad_sup_prop_heat_kernel}
			|D\tilde G \zeta (t,x)|
			&\leq & C
			\|\zeta\|_{L^\infty(C^\gamma)}
			\int_0^t \int_{\R^d}(t-s)^{\frac{-1+\gamma}{2}} \bar p (s,t,x,y) dy \ ds \nonumber \\
			&\leq &
			C 
			\|\zeta\|_{L^\infty(C^\gamma)}
			\int_0^t  [\nu (t-s)]^{\frac{-1+\gamma}{2}}  ds 
			\nonumber \\
			&\leq &
			C 
			\|\zeta\|_{L^\infty(C^\gamma)}
			\nu ^{\frac{-1+\gamma}{2}}  
			t^{\frac{1+\gamma}{2}} .
			\end{eqnarray}
			
			\item Control of $\|D^2\tilde G \zeta\|_{L^\infty} $
			
			Similarly, by cancellation, we obtain
			\begin{eqnarray*}
				|D_x^2\tilde G \zeta (t,x)|
				&=&
				\Big | \int_0^t \int_{\R^d} D_x^2 \tilde p (s,t,x,y) [ \zeta(s,y)-\zeta(s,x)] dy \ ds \Big |
				\nonumber \\
				&\leq & C
				\|\zeta\|_{L^\infty(C^\gamma)}
				\int_0^t \int_{\R^d} [\nu (t-s)]^{-1} \bar p (s,t,x,y)|y-x|^\gamma dy \ ds ,
			\end{eqnarray*}
			from \eqref{FIRST_deriv_CTR_DENS} and the regularity of $\zeta$.
			We deduce from identity \eqref{ineq_absorb}
			\begin{eqnarray}\label{ineq_D2_sup_prop_heat_kernel}
			|D_x^2\tilde G \zeta (t,x)|
			&\leq & C
			\|\zeta\|_{L^\infty(C^\gamma)}
			\int_0^t \int_{\R^d} [\nu (t-s)]^{-1+\frac{\gamma}{2}} \bar p (s,t,x,y) dy \ ds \nonumber \\
			&\leq &
			C 
			\|\zeta\|_{L^\infty(C^\gamma)}
			\int_0^t  [\nu (t-s)]^{-1+\frac{\gamma}{2}}  ds 
			\nonumber \\
			&\leq &
			C 
			\|\zeta\|_{L^\infty(C^\gamma)} \nu ^{-1+\frac{\gamma}{2}} 
			t^{\frac{\gamma}{2}} .
			\end{eqnarray}
			
			\textbf{H\"older moduli of the spatial variable}

			\item Control of $\|D^2\tilde G \zeta\|_{L^\infty(C^\gamma)} $
			
			For any $(x,x') \in \R^d \times \R^d$, we aim to prove that there is $C>0$ such that
			\begin{equation*}
			\sup_{t \in [0,T]} |D^2\tilde G \zeta(t,x)-D^2\tilde G \zeta(t,x')| \leq C \nu^{-1+ \frac{\gamma}{2}}(1+\nu^{-\frac{1}{2}})
			\|\zeta\|_{L^\infty(C^\gamma)} |x-x'|^\gamma.
			\end{equation*}
			We basically differentiate the \textit{diagonal}/\textit{off-diagonal} regimes.
			Let us define the associated Green kernels:
			\begin{eqnarray}\label{def_G_diag_G_off_diag}
			\tilde G^{\text{diag}}\zeta(t,x) &:=&  \int_0^t \int_{\R^d} \tilde p (s,t,x,y) \zeta(s,y) \mathbb{I}_{t-s\leq [x-x'|^2} dy \ ds,
			\nonumber \\
			\tilde G^{\text{off-diag}}\zeta(t,x) &:=& \int_0^t \int_{\R^d} \tilde p (s,t,x,y) \zeta(s,y) \mathbb{I}_{t-s> [x-x'|^2} dy \ ds.
			\end{eqnarray}
			
			$*$ \textit{diagonal}: if $(t-s)\leq [x-x'|^2$, we get by triangular inequality
			\begin{equation*}
			|D^2\tilde G^{\text{diag}} \zeta(t,x)-D^2\tilde G^{\text{diag}} \zeta(t,x')|
			\leq |D^2\tilde G^{\text{diag}} \zeta(t,x)|+|D^2\tilde G^{\text{diag}} \zeta(t,x')|.
			\end{equation*}
			Inequality \eqref{ineq_D2_sup_prop_heat_kernel} and the \textit{diagonal} considered regime yield
			\begin{eqnarray}\label{ineq_D2_hold_prop_heat_kernel_diag}
			|D^2\tilde G^{\text{diag}} \zeta(t,x)-D^2\tilde G^{\text{diag}} \zeta(t,x')|
			&\leq& 2  C 
			\|\zeta\|_{L^\infty(C^\gamma)}
		\nu ^{-1+ \frac{\gamma}{2}}	t^{\frac{\gamma}{2}} \mathbb{I}_{t\leq [x-x'|^2}
			\nonumber \\
			&\leq&  
			C \nu^{-1+ \frac{\gamma}{2}}
			\|\zeta\|_{L^\infty(C^\gamma)}
			|x-x'|^{\gamma} .
			\end{eqnarray}
			
			$*$ \textit{off-diagonal}: if $(t-s)> [x-x'|^2$, we write by a Taylor expansion:
			\begin{eqnarray*}
				&&|D^2\tilde G^{\text{off-diag}}_ \zeta(t,x)-D^2\tilde G^{\text{off-diag}}_ \zeta(t,x')|
				\nonumber\\
				&=&
				\Big | \int_0^{t-|x-x'|^2} \int_{\R^d} [D_x^2 \tilde p (s,t,x,y)-D_x^2 \tilde p (s,t,x',y)]  \zeta(s,y) dy \ ds \Big |
				\nonumber \\
				&=&
				\Big | \int_0^{t-|x-x'|^2} \int_{\R^d} \int_0^1 (x-x') \cdot D_x^3 \tilde p \big (s,t,x'+\mu(x-x'),y \big ) d\mu \ \zeta(s,y) dy \ ds \Big |.
			\end{eqnarray*}
			We write by cancellation,
			\begin{eqnarray*}
				&&|D^2\tilde G^{\text{off-diag}} \zeta(t,x)-D^2\tilde G^{\text{off-diag}} \zeta(t,x')|
				\nonumber \\
				&=&
				\Big | \int_0^{t-|x-x'|^2} \int_{\R^d} \int_0^1 (x-x') \cdot D_x^3 \tilde p \big (s,t,x'+\mu(x-x'),y \big ) 
				\nonumber \\
				&& \big [\zeta\big (s,y \big ) - \zeta \big ( s, x'+\mu(x-x') \big ) \big ]d\mu \ dy \ ds \Big |
				\nonumber \\
				&\leq & C \|\zeta\|_{L^\infty(C^\gamma)} |x-x'|
				\nonumber \\
				&&  \int_0^{t-|x-x'|^2} \int_{\R^d} \int_0^1 
				[\nu (t-s)]^{-\frac 32 }\bar p \big (s,t,x'+\mu(x-x'),y \big ) \big |y- x'+\mu(x-x') \big ) \big |^\gamma d\mu \ dy \ ds .
			\end{eqnarray*}
			We get by \eqref{ineq_absorb},
			\begin{eqnarray}\label{ineq_D2_hold_prop_heat_kernel_off_diag}
			&&|D^2\tilde G^{\text{off-diag}} \zeta(t,x)-D^2\tilde G^{\text{off-diag}} \zeta(t,x')|
			\nonumber \\
			&\leq & C \|\zeta\|_{L^\infty(C^\gamma)} |x-x'|
			\nonumber \\
			&&
			\int_0^{t-|x-x'|^2} \int_{\R^d} \int_0^1 
			[\nu (t-s)]^{\frac {-3+\gamma}2 }\bar p \big (s,t,x'+\mu(x-x'),y \big )d\mu \ dy \ ds 
			\nonumber \\
			&\leq & C \|\zeta\|_{L^\infty(C^\gamma)} |x-x'|
			\int_0^{t-|x-x'|^2} [\nu (t-s)]^{\frac {-3+\gamma}2 } ds
			\nonumber \\
			&\leq & C \nu ^{\frac {-3+\gamma}2 } \|\zeta\|_{L^\infty(C^\gamma)} |x-x'|^\gamma
			.
			\end{eqnarray}
			Finally, by inequalities \eqref{ineq_D2_hold_prop_heat_kernel_diag} and \eqref{ineq_D2_hold_prop_heat_kernel_off_diag}, we get:
			\begin{equation}\label{ineq_D2_hold_prop_heat_kernel}
			\|D^2\tilde G \zeta\|_{L^\infty(C^\gamma)}
			\leq C \nu^{-1+ \frac{\gamma}{2}}(1+\nu^{-\frac{1}{2}}) 
			 \|\zeta\|_{L^\infty(C^\gamma)} 
			.
			\end{equation}

		\item Control of $\|\partial_t \tilde G \zeta\|_{L^\infty} $
			

			By chain rules,
			\begin{equation}\label{eq_partial_t_tilde_G}
			\partial_t \tilde G \zeta(t,x)
			= 
			- \zeta(t,y) 
			+  \int_0^t \int_{\R^d} \partial_t \tilde p (s,t,x,y) \zeta(s,y) dy \ ds
			.
			\end{equation}
			We already know that $ \zeta \in L^\infty(C^\gamma)$, we then have to show that the second contribution in the r.h.s. lies in the same H\"older space.
			To do that, let us first remark that from the heat equation: 
			\begin{equation}\label{eq_heat_equ}
			\partial_t \tilde p (s,t,x,y)= 
			{\rm Tr} \big (  D^2_x \tilde p (s,t,x,y) a(s)\big ) .
			\end{equation}
			Hence we can rewrite \eqref{eq_partial_t_tilde_G} by:
			\begin{equation}\label{eq_partial_t_tilde_G_bis}
			\partial_t \tilde G \zeta(t,x)
			= 
			- \zeta(t,y) 
			+ 
			{\rm Tr} \Big (  \int_0^t \int_{\R^d} D^2_x \tilde p (s,t,x,y) a(s) \zeta(s,y) dy \, ds \,  \Big )
			.
			\end{equation}
				From \eqref{ineq_D2_hold_prop_heat_kernel} and \eqref{eq_partial_t_tilde_G}, we readily have:
			\begin{equation}
			\| \partial_t \tilde G \zeta \|_{L^\infty}
			\leq 
			\|\zeta\|_{L^\infty}+ C 
			[\nu T]^{\frac{\gamma}{2}}  \|\zeta\|_{L^\infty(C^\gamma)}
			.
			\end{equation}

			\item Control of $\|\partial_t \tilde G \zeta\|_{L^\infty(C^\gamma)} $

			We directly deduce from \eqref{ineq_D2_hold_prop_heat_kernel}:
			\begin{equation}\label{eq_partial_t_tilde_G_ter}
			\| \partial_t \tilde G \zeta \|_{L^\infty(C^\gamma)}
			\leq 
			\nu^{ \frac{\gamma}{2}}(1+\nu^{-\frac{1}{2}})
			\|\zeta\|_{L^\infty(C^\gamma)}
			\leq  C \|\zeta\|_{L^\infty(C^\gamma)}
			.
			\end{equation}

		\end{itemize}
	\end{proof}

\bibliographystyle{alpha}
	\bibliography{bibli}
	
\end{document}